\documentclass[11pt]{amsart}
\usepackage{amsfonts}
\usepackage{enumerate}
\usepackage{xcolor}
\def\cal{\mathcal}
\newtheorem{theorem}[subsection]{Theorem}
\newtheorem{proposition}[subsection]{Proposition}
\newtheorem{corollary}[subsection]{Corollary}
\newtheorem{lemma}[subsection]{Lemma}
\newtheorem{definition}[subsubsection]{Definition}
\newtheorem{remark}[subsubsection]{Remark}
\newtheorem{remarks}[subsubsection]{Remarks}

\def\ep{\noindent{\hfill $\Box$}}
\newcommand{\Comment}[1]{{\small\em
#1}{\color{red}}} 
\renewcommand{\Comment}[1]{}

\begin{document}
\numberwithin{equation}{subsection}

%\runningheads{ Mary Rees}{Persistent Markov partitions}  
\begin{center}
\title{Persistent Markov partitions for rational maps}
\author{Mary Rees}
\address{ Dept. of Math. Sci., University of Liverpool, Mathematics Building, Peach St., Liverpool L69 7ZL, U.K.,}
\email {maryrees@liv.ac.uk}
\end{center}
\maketitle

\begin{abstract}
A construction is given of  Markov partitions for some rational maps, which persist over regions of parameter space, not confined to single hyperbolic components. The set on which the Markov partition exists, and its boundary, are analysed.
\end{abstract}

\Comment{\color{blue} {It is difficult to know where to put this acknowledgement and likely to be premature in any case. But I would like to thank the referee(s) and editor(s) for their assistance and extraordinary patience during the preparation of this paper.

In view of the number of comments, I am using this annotated version of the resubmission as a covering letter. My responses to comments are in red throughout the first section of this paper. I very much regret that I have not had time to look at later sections at all. Section 1 has simply taken all my time.
 In the process of responding to comments, some new lemmas have been introduced, namely \ref{1.18}, \ref{1.17}, \ref{1.19}. In addition an important new argument has been added in at the end of \ref{1.9}. \ref{1.18} is part of the explanation of an argument used towards the main construction in \ref{1.10}. \ref{1.17} and \ref{1.19} lead to some simplifications, but were really  introduced to correct obscurities -- probably errors -- in the main construction.}}

\section{Construction of partitions}\label{1}

The first result of this paper is a construction of Markov partitions for some rational maps, including some non-hyperbolic rational maps  (Theorem \ref{1.1}). Of course, results of this type have been around for many decades. We comment on this below. There is considerable freedom in the construction. In particular, the construction can be made so that the partition varies isotopically to a partition for all maps in a sufficiently small neighbourhood of the original one (Lemma \ref{2.1}). So the partition is not specific, like the Yoccoz puzzle, and also less specific than other partitions which have been developed to exploit the ideas on analysis of dynamical planes and parameter space, which were pioneered using the Yoccoz puzzle. We then investigate the boundary of the set of rational maps for which the partition exists in section \ref{2}, in particular in Theorem \ref{2.2}. We also explore the set in which the partition does exist, in section \ref{3}, in particular in Theorem \ref{3.2}. We show how parameter space is partitioned, using a partition which is related to the Markov partitions of dynamical planes --- in much the usual manner --- and show that all the sets in the partition are nonempty. We are able to apply some of the results of \cite{Roesch} in our setting, in particular in the analysis of dynamical planes. The main tool used in the results about the partitioning the subset of parameter space admitting a fixed Markov partition is the $\lambda $-lemma \cite{M-S-S}.

It is natural to start our study with hyperbolic rational maps. For some integer $N$ which depends on $f$, the iterate $f^N$ of a hyperbolic map $f$ is expanding on the Julia set $J=J(f)$ with respect to the spherical metric. 
The full expanding property does not hold for a parabolic rational map on its Julia set, but a minor adjustment of it does. Given any closed subset of the Julia set disjoint from the parabolic orbits, the map $f^N$ is still expanding with respect to the spherical metric, for a suitable $N$.

We shall use the spherical metric throughout the proof, unless otherwise stated, when referring to distance or diameter. Also,  $d(\cdot ,\cdot )$ is the spherical metric, unless otherwise stated. Similarly, expansion and magnitude of derivative are with respect to the spherical metric. 

\Comment{\color{red}{This has been added in here, and consequently references to the spherical metric have been removed at various places elsewhere.}}

Throughout this paper, an {\em{arc}} is a homeomorphic image of $[0,1]$. The {\em{endpoints}} of the arc are the images of $0$ and $1$. An {\em{open arc }}  is a homeomorphic image of $(0,1)$. 

\Comment{\color{red} {Again this has been added in here -- recommended  in Comment (35)}}

 We shall use the following definition of Markov partition for a rational map $f:\overline{\mathbb C}\to \overline{\mathbb C}$.

\begin{definition}
A {\em{Markov partition for $f$}} is a set ${\cal{P}}=\{ P_1,\cdots P_n\} $ such that:
\begin{itemize}
\item $\overline{{\rm{int}}(P_i)}=P_i$;
\item $P_i$ and $P_j$ have disjoint interiors if $i\ne j$;
\item $\bigcup _{i=1}^nP_i=\overline{\mathbb C}$;
\item each $P_i$ is a union of connected components of $f^{-1}(P_j)$ for varying $j$. 
\end{itemize}
\end{definition}

Our first theorem applies to a familiar ``easy'' class of rational maps. In particular, we assume that every critical orbit is attracted to an attractive or parabolic periodic orbit. The most important property of the Markov partitions yielded by this theorem, however, is that the set of rational maps for which they exist is open -- and if this open set contains a rational map with at least one parabolic periodic point, the open set is not contained in a single hyperbolic component.
 
\begin{theorem}\label{1.1}  Let $f:\overline{\mathbb C}\to \overline{\mathbb C}$ be a rational map such that every critical point is in the Fatou set, and such that the closure of any Fatou component is a closed topological disc, and all of these are disjoint. Let $F_0$ be the union of the periodic Fatou components. Let $Z$ be a finite forward invariant set which includes all parabolic points. Let $G_0\subset \overline{\mathbb C}$ be a graph  such that the following hold. 
\begin{itemize}
\item $G_0$ is connected and has finitely many vertices and edges.
\item $G_0$ is piecewise $C^1$, that is, the closure of each edge is a piecewise $C^1$ arc.

\item  All components of $\overline{\mathbb C}\setminus G_0$ are topological discs, as are the closures of these components.
\item $G_0\cap (F_0\cup Z)=\emptyset $, any component of $\overline{\mathbb C}\setminus G_0$ contains at most  one component of $F_0\cup Z$. 
\item $G_0$ is trivalent, that is, exactly three edges meet at each vertex.
\item The closures of any two components of\  \ $\overline{\mathbb C}\setminus G_0$ intersect in at most a single component, which, if it exists, must be either an edge together with the endpoints of this edge, or a single vertex, by the previous conditions. 
\end{itemize}

Then there exists $G'\subset \overline{\mathbb C}\setminus (F_0\cup Z)$ which is ambient isotopic to $G_0$ such that $G'\subset f^{-N}(G')$ for some $N$. Moreover, given any $\varepsilon >0$, by choosing $N$ sufficiently large,   the isotopy of $\overline{\mathbb C}$ which maps $G_0$ to $G'$ can be chosen to be a composition of one isotopy which is the identity outside finitely many components of  the Fatou set, and a second which maps  all points a spherical distance $\le \varepsilon $, and with both isotopies fixing  $F_0\cup Z$. 
\Comment{\color{red}{Comment (1)``which is ambient'' added before ``isotopic''. I think that any two isotopic  finite graphs in a
compact surface --- such as $\overline{\mathbb C}$ --- are also ambient isotopic. But I have not yet found a reference for this. 

Comment (2)The phrase ``in $\overline{\mathbb C}\setminus (F_0\cup Z)$has been removed and replaced by ``and to fix $F_0\cup Z$'' at the end. ``$G$ and'' has been removed. The property of the isotopy has been weakened from moving all points a spherical distance $\le \varepsilon$, to compensate for removing the property of $G_0$ of having at most one component of intersection with each component of the Fatou set.}}

Moreover, there is a  connected graph $G$
 with finitely many vertices and edges, with $G\subset f^{-1}(G)$, and such that the following hold, where $\varepsilon '>0$ can be taken arbitrarily small given $\varepsilon $ and $\alpha _0$.

 \Comment{\color{red}{Comment (3) The full stop has been removed after $\varepsilon $ above. 
 
 $\delta _0$ has been renamed $\alpha _0$ as $\delta _0$ has been used many times with a different meaning in other places.
  
 The property that all vertices of $G$ are bivalent or trivalent has been removed. The proof certainly does not show this. Fortunately it does not seem to be necessary for the results in the later sections
 
 In the last property, $\varepsilon _1$ has been replaced by $\varepsilon '$.}}
 
 \begin{itemize}
 \item[1.] $G$ is in the $\varepsilon '$-neighbourhood of  $\cup _{i=0}^{N-1}f^{-i}(G')$.
 \item[2.] Any simple closed loop of $G'$ bounding a disc  of diameter $\ge \alpha _0$ on both sides is within $\varepsilon '$ of a closed loop of $G$. 
 
 \end{itemize}
 Hence ${\cal{P}}=\{\overline{U}:U\mbox{ is a component of }\overline{\mathbb C}\setminus G\} $ is a Markov partition for $f$, such that each set in the partition contains at most one periodic Fatou component.  \end{theorem}

For quite some time, I thought that there was no general result of this type in the literature, that is, no general result giving the existence of such a graph and related Markov partition for a map $f$ with expanding properties.  To some extent, this is true. One would expect to have a result of this type for smooth expanding maps of compact Riemannian manifolds,  for which the derivative has norm greater than one with respect to the Riemannian metric. I shall call such maps {\em{expanding local diffeomorphisms}}. Of course, an expanding map of a compact metric space is never invertible. Also, a rational map is never expanding on the whole Riemann sphere, unless one allows the metric to have singularities --- because of the critical points of the map. A hyperbolic rational map is an expanding local diffeomorphism on a neighbourhood of the Julia set, but such a neighbourhood is not forward invariant. The invertible analogue of expanding local diffeomorphisms is Axiom A diffeomorphisms. There is, of course, an extensive literature on these, dating from the 1960's and '70's.  The existence of  Markov partitions for Axiom A diffeomorphisms was proved by Rufus Bowen \cite{B1}, who developed the whole theory of describing invertible hyperbolic systems in terms of their symbolic dynamics in a remarkable series of papers. Bowen's results are in all dimensions. The construction of the sets in these Markov partitions is quite general, and the sets are not shown to have nice properties. In fact, results appear to be in the opposite direction: \cite{B2}, for example, showing that boundaries of Markov partitions of Anosov toral diffeomorphisms of the three-torus are never smooth  -- a relatively mild, but interesting pathology, which, in itself, has generated an extensive literature. 

The existence of Markov partitions for expanding maps of compact metric spaces appears as Theorem 4.5.2 in the recent book by Przytycki and Urbanski \cite{P-U}.  But there is no statement, there, about topological properties of the sets in the partition. I only learnt relatively recently (from Feliks Pzrytycki, among others) about the work of F.T. Farrell and L.E. Jones on expanding local diffeomorphisms, in particular about their result in dimension two \cite{F-J}. Their result is a  version of the statement in Theorem \ref{1.1} -- more general in some respects -- about an invariant graph $G$ for $f^N$ for a suitably large $N$. In the Farrell-Jones set-up, $f$ is an expanding local diffeomorphism of a compact two-manifold. Unaware of their result, the first version of this paper included my own proof of the theorem above -- which of course has different hypotheses from the Farrell-Jones result. Other such results have also been obtained relatively recently in other contexts, for example by Bonk and Meyer in \cite{B-M}, where Theorem 1.2 states that the $n$'th iterate $F=f^n$ of an expanding Thurston map $f$ admits an invariant Jordan curve, if $n$ is sufficiently large, and consequently, by Corollary 1.5, $F$ admits cellular Markov partitions of a certain type. My proof makes an assumption of conformality which is not in the Farrell-Jones result, and part of the proof of \ref{1.5} is rather different from that of Farrell-Jones. I also claimed a proof for $f$, rather than $f^N$. There is no such result in \cite{F-J}. If $G'\subset \overline{\mathbb C}$ is a graph  satisfying $G'\subset f^{-N}(G')$, then the set $G^0=\bigcup _{i=0}^{N-1}f^{-i}(G')$ satisfies $G^0\subset f^{-1}(G^0)$. But $G^0$ might not be a graph with finitely many edges and vertices. In the first version of this paper a proof was given that $G^0$ was, nevertheless, such a graph. However, on seeing the result, Mario Bonk and others warned that the method of proof did not appear to take account of counter-examples in similar contexts, and was likely to be flawed - as indeed it was. 

The statement now proved is not that $G^0$ itself is a graph with finitely many vertices and edges, but that there is such a graph $G$ with $G\subset f^{-1}(G)$, arbitrarily close to $G^0$ in the Hausdorff metric, and with closed loops arbitrarily close to closed loops in $G^0$ bounding discs of diameter bounded from $0$.

As a corollary of Theorem \ref{1.1}, we have the following.

\begin{corollary}\label{1.2} Let $f:\overline{\mathbb C}\to \overline{\mathbb C}$ be a rational map with connected Julia set $J$, such that the forward orbit of each critical point is attracted to an attractive or parabolic periodic orbit, and such that the closure of any Fatou component is a closed topological disc, and all of these are disjoint. Then there exists a  connected graph $G\subset \overline{\mathbb C}$ such that the following hold.
\begin{enumerate}
\item  $G\subset f^{-1}G$. 
\item $G$ does not intersect the closure  of any periodic Fatou component.
\item Any component of $\overline{\mathbb C}\setminus G$ contains at most one periodic Fatou  component of $f$.
  \end{enumerate}
 In particular, the set of closures of components of  $\overline{\mathbb C}\setminus G$ is a Markov partition for $f$. \end{corollary}

\noindent{\em{Proof.}} We can choose the graph $G_0$ of Theorem \ref{1.1} to satisfy the conditions of \ref{1.1} and also to not intersect the closure of any periodic Fatou component, and to separate periodic Fatou components. Theorem \ref{1.1} then gives Property 1 above.  By taking $\varepsilon >0$ sufficiently small, we can ensure that $G'$ also does not intersect the closure of any periodic Fatou component, and separates them. Then the same is true for $f^{-i}(G')$ for each 
$0\le i<N$. By choosing $\varepsilon '>0$ sufficiently small that $f^{-i}(G')$ does not intersect the $2\varepsilon '$-neighbourhood of any periodic Fatou component, Property 1 of \ref{1.1} gives Property 2 of this corollary. Then choosing $\delta _0$ such that every periodic Fatou component contains a disc of diameter $\ge \delta_0$, Property 2 of \ref{1.1} gives that every closed loop in $G'$ which separates some periodic Fatou component from all others is approximated by a similar loop in $G$. This gives Property 3 above.

\Comment{\color{red}{COMMENT
Correction made in response to referee's comment (4), yes, $\delta _0$ needs to give a lower bound on the diameter of discs in periodic Fatou components rather than an upper bound on diameter.

 $\varepsilon _1$ replaced by $\varepsilon '$.
 
More detail given on obtaining property 3.}}
\ep

The first step in the proof of \ref{1.1} is a lemma about the existence of subgraphs -- which, as already stated, parallels methods in Farrell-Jones, section 1 of \cite{F-J}. 

\begin{lemma}\label{1.3}
Let $f$, $F_0$, $Z$ and $G_0$ be as in \ref{1.1}. Let $\delta _1>0$ be given. Let $F(G_0)$ denote the union of $G_0$ and all sets $\overline{F}$ such that $F$ is a Fatou component intersected by $G_0$ of diameter $\ge \delta _1/2$. 
Then the following holds for $\delta $ sufficiently small given $\delta _1$. Let $\Gamma $ be another graph which has the same properties as $G_0$, and such that every component of $\overline{\mathbb C}\setminus \Gamma $ that has nontrivial intersection with the $2\delta _1$-neighbourhood of $F(G_0)$
 is either contained in the $\delta $-neighbourhood of some  Fatou component whose closure is  intersected by $G_0$, or has diameter $<\delta $. Then there is a subgraph $G_1$ of $\Gamma $ which is in the $\delta _1$-neighbourhood of $F(G_0)$, such that $G_1$ can be isotoped to $G_0$ by an isotopy $\varphi _t$ ($t\in [0,1]$) which is the identity outside this neighbourhood, and such that  the diameter of any set $\{\varphi _t(x):t\in [0,1]\}$ is $\le 2M_0$, where $M_0$ is the maximum diameter of any Fatou component.
\end{lemma}

\Comment{\color{red}{Following the comment (5) in the referee's report, the  occurrence of ``within a'' in line 5 of the statement of Lemma 1.3 has been replaced by ``that has nontrivial intersection with the'' 
    Then ``within the'' has been replaced by ``contained in the ''. The bound on the lengths of isotopy paths has been added.}}

\begin{remark} Many of the vertices of $G_1$ are likely to be bivalent rather than trivalent, but these are the only types which occur.
\end{remark}

\noindent{\em{Proof.}} \Comment{\color{red}{This proof has been quite extensively rewritten in response to comments (6) to (12). One important change is that the attempt to start by isotoping $G_0$ into the Julia set has been abandoned. As the referee rightly pointed out, this is incompatible with making the image of $G_0$ under the isotopy piecewise $C^1$. 
}}
 
 The idea of the proof is to choose the graph $G_1$ in the union of boundaries of components of $\overline{\mathbb C}\setminus \Gamma $ which intersect $G_0$. Care is needed at vertices of $G_0$ and at intersections of $G_0$ with Fatou components. 

We can assume without loss of generality that $\delta _1$ is sufficiently small that $4\delta _1$-neighbourhoods of vertices of $G_0$ are disjoint, and that any Fatou component which contains a vertex of $G_0$ has diameter $\ge 4\delta _1$.  Choose $\delta _2>0$ such that  the endpoints of any arc on $G_0$ of diameter $\ge \delta _1/2$ are distance $\ge 4\delta _2$ apart. Similarly choose $\delta _3>0$ and then $\delta _4>0$ such that  the endpoints of any arc on $G_0$ of  diameter $\ge \delta _i/2$ are distance $\ge 4\delta _{i+1}$ apart, and the distance between any two distinct Fatou components of diameter $\ge \delta _i/2$ is $\ge 4\delta _{i+1}$ for $i=2$, $3$.   So we have $\delta _{i+1}\le \delta _i/8$ for $i=1$, $2$, $3$. 

Let $\mathcal{C}_0$ be the set of components of intersection $\gamma $ of $G_0$ with the $\delta _3$-neighbourhood of some Fatou component of diameter $\ge \delta _2$ which also intersect the $\delta _3/2$ neighbourhood of this component, We will construct  an isotopy $\psi _t$ ($t\in [0,1/2]$)  which is the identity outside the $\delta _3/2$-neighbourhood of Fatou components of diameter $\ge \delta _2$ but moves all components of $\bigcup \mathcal{C}_0$ outside the $\delta $-neighbourhoods of these Fatou components, where $\delta $ is yet to be chosen.  We  will denote by $G_0'$ the image of $G_0$ under this isotopy. There is a bound $N_0$, where $N_0$ is bounded in terms of $\delta _4$ and $\delta _2$,  on the number of sets in $\mathcal{C}_0$ intersecting the $\delta _3/2$-neighbourhood  of some  Fatou component $V$ of diameter $\ge \delta _2$. By isotoping one set  $\gamma $ in  $\mathcal{C}_0$ intersecting $V$ at a time, we can ensure that for a suitable $\delta _j$ for $j\le N_0+4$, for $\gamma \in{\mathcal{C}}_0$  for some $3\le i\le N_0+2$,  the image of  $\gamma $ is contained in the $\delta _i$ neighbourhood of $V$  and is disjoint from the $2\delta _{i+1}$ neighbourhood and does not intersect the $\delta _{i+2}$-neighbourhood of  any Fatou component of diameter $\ge \delta _{i+1}$ and there is just one such  $\gamma $ for each $i$ and $V$. We call $G_{0,i}'$ the union, over all $V$, of images of these sets $\gamma $. We can also ensure that the $\delta _j$ are decreasing sufficiently fast that the endpoints of any arc of $G_{0,i}'$ of diameter $\ge \delta _{i+1}/2$ are distance $\ge 4\delta _{i+2}$ apart. In particular, $8\delta _{j+1}\le \delta _j$ for all $j\le N_0+3$.  So for each vertex $v$ of $G_{0,i}'$. there is an open neighbourhood $U_v$ of a connected  $G_v\subset G_{0,i}'$ containing  $v$ such that $U_v$ contains the $\delta _{i+2}$-neighbourhood of $G_v$, the points of $G_v\cap \partial U_v$ are distance $\le \delta _{i+1}$ and $\ge \delta _{i+1}/2$ from $v$ and the sets $U_v$ are all disjoint. The diameters of path $\{\psi _t(x):x\in [0,1/2]\} $ is $\le M_0+\delta _3$. 

Now we assume that $\delta <\delta _i/2$ for all $i\le N_0+4$. 

Now we isotope $G_0'$ to a subgraph of $\Gamma $. First we consider points of 
$$G_0'\setminus \bigcup _{3\le i\le N_0+2}G_{0,i}'.$$ We choose a finite subset $G_{0,1}''$ of 
$$(G_0'\setminus \bigcup _{3\le i\le N_0+2}G_{0,i}')\cap \Gamma$$
 such that if two points $x$ and $y$ in $G_{0,1}''$ are not separated in $G_0'$ by other points of $G_{0,1}''$ or by  $\bigcup _{3\le i\le N_0+2}G_{0,i}'$ then the following hold.
\begin{itemize}
\item There are components  $U_1$ and $U_2$ of $\overline{\mathbb C}\setminus \Gamma $  such that $\overline{U_1}\cap \overline{U_2}\ne \emptyset $ and $x\in\partial U_1$ and $y\in \partial U_2$.
\item If $U_1=U_2$, there is no other point of  $G_{0,1}''$ in $\partial U_1=\partial U_2$. If $U_1\ne U_2$ there is no other point of $G_{0,1}''$ in $\partial U_1\cup \partial U_2$ except when $x$ and $y$ are separated by a vertex  $v$ of $G_0'$. In this case, a third point $z$ is allowed in $\partial U_1\cup \partial U_2$ such that there is one of  $x$, $y$ and $z$  each of the arms of $G_0'$ attached to $v$. At most one of the three points $x$, $y$ and $z$ is in $\partial U_1\cap \partial U_2$. 
\end{itemize}
We  choose $G_{0,1}''$ by successively removing points from an initial  choice of the set, as necessary. Then the arc of $G_0'$ between two points of $G_{0,1}''$ which are not separated by other points of $G_{0,1}''$ or by $\bigcup _{i\ge 3}G_{0,i}'$ has diameter $\le \delta _1/2$. We can join two adjacent points $x$ and $y$ as above by an arc in $\partial U_1\cup \partial U_2\subset \Gamma$. Since $U_i $ has diameter $\le \delta _2+2\delta _3$, the diameter of $U_1\cup U_2$ is $<3\delta _2$, whether or not $U_1=U_2$.  So the arc of $\Gamma $ between $x$ and $y$ has diameter $<3\delta _2<\delta _1/2$. The arc in $G_0'$ between $x$ and $y$, which is also in $G_0$, has diameter $\le \delta _1/2$, and an isotopy between the arcs in $G_0$ and $\Gamma $ can be chosen to be the identity outside a $\delta _1$ neighbourhood of $G_0$. In the case when there is a vertex $v$ between $x$ and $y$ and hence there is a third point $z$ on the other edge of $G_0$, the arcs from $y$ and $z$ to $x$ in $\Gamma $, can be chosen to meet at a vertex of $\Gamma $ and continue on the same arc to $x$. In this case $x$, $y$ and $z$ can be in the closures of two or three distinct components of $\overline{\mathbb C}\setminus \Gamma $. 

Similarly we can find a set of points $G_{0,i}''\subset G_{0,i}'$ for $3\le i\le N_0+2$ such that similar properties hold for points $x$, $y\in G_{0,i}''$ in the same component of $G_{0,i}'$ which are not separated in $G_{0,i}'$ by any other points in $G_{0,i}''$. The components $U_1$ and $U_2$ of $\overline{\mathbb C}\setminus \Gamma $ containing $x$ and $y$ then have diameter $\le \delta _{i+1}+2\delta _{i+2}$, and so the diameter of $U_1\cup U_2$ is, once again, $<3\delta _i$, whether or not $U_1=U_2$. So the arc of $G_{0,i}'$ between $x$ and $y$ has diameter $<\delta _{i}/2$, and an isotopy from the arc  of  $G_{0,i}'$ between $x$ and $y$ to the arc of  $\Gamma $ can be chosen to be the identity outside the $\delta _i$-neighbourhood of $G_{0,i}'$. We construct the isotopy mapping $G_{0,i}'$ into $\Gamma $ on arcs between $x$ and $y$ containing a vertex of $G_0'$ in the same way as in the case $i=1$. 

So we can construct  an isotopy $\{ \xi _t:t\in [0,1/2]\} $ which maps $G_0'$ to $\Gamma $ and such that the diameter of any set $\{ \xi _t(x):t\in [0,1/2]\} $ is $\le \delta _1$. Combining this with the isotopy $\psi _t$ which maps $G_0$ to $G_0'$ and which is the identity outside the $\delta _3/2$-neighbourhood of the union of Fatou components of diameter $\ge \delta _2$, we obtain the required isotopy $\varphi _t$ ($t\in [0,1]$), which is the identity outside the $\delta _1$-neighbourhood of $F(G_0)$, and such that the diameter of any path $\varphi _t(x):t\in [0,1]\} $ is $\le M_0+\delta _3+\delta _1<2M_0$, assuming as we may do that  $\delta _1<M_0/2$.

\ep

We will prove Theorem \ref{1.1} using Lemma \ref{1.4}. A homeomorphism $h$ of $\overline{\mathbb C}$ is {\em{piecewise $C^1$}} if there is a partition of $\overline{\mathbb C}$ into sets with piecewise $C^1$ boundary (the boundary is a finite union of closed $C^1$ arcs) such that $h$ is $C^1$ restricted to each set in the partition. 

\begin{lemma}\label{1.4} Let $f$, $Z$, $G_0$ be as in \ref{1.1} to \ref{1.3}. As in \ref{1.3}, let $F(G_0)$ be the union of $G_0$ and the closures of any components of the Fatou set of $f$ which are intersected by $G_0$. Let $U$ be a closed  neighbourhood of $F(G_0)$ with $C^1$ boundary such that for some $M$,
the diameter of any component of  $\mbox{int}(F(G_0))$ is at most $M$ times the distance of any point on the boundary of the component to  $\partial U$. 

\Comment{\color{red}{Comment (13): missing space after comma added in. Comment (14): redundant $i$ removed }}
 
 Let $\varepsilon >0$ and $0<\lambda <1$ be given. Then for all sufficiently large $N$, depending on $G_0$, $M$, $\varepsilon $ and $\lambda $, there are a graph $G_1$, a neighbourhood $U_1$ of $G_1$, a constant $C_1$, a piecewise $C^1$ homeomorphism $h$ of $\overline{\mathbb C}$ and and $h$ such that the following hold.
 \begin{itemize}
 \item $G_1\subset f^{-N}(G_0)$ and $G_1$ is contained in the $\varepsilon $-neighbourhood of $F(G_0)$.
\item  $h$ is  the identity outside  $U$, is isotopic to the identity via an isotopy $\theta _t$ ($t\in [0,1]$), the diameter of $\{ \theta _t(x):t\in [0,1]\} $ is $\le C_1$ for all $x\in \overline{\mathbb C}$,  and $h(G_0)=G_1$.
\item $f^N(U_1)\subset U$, and $g=f^N\circ h$ is expanding on $h^{-1}(U_1)$, and $f^N$ is expanding on $U_1$  both with expansion constant $\ge \lambda ^{-1}$.
\end{itemize}
\end{lemma}

\Comment{\color{red}{ Statement modified to include existence and properties of $U_1$, and to include information about the diameter of the isotopy between the identity and $h_0$

Comment (14): redundant $U$ removed.

Comment (15), concerning the start of the proof. I think actually $B_1$ and $B_2$ do not need to be made more precise.The example given of $B_2$ being a periodic Fatou component and $S$ a local inverse of $f^n$ mapping $B_2$ to $f^i(B_2)$ for some $i\ge 0$ does not apply because such a local inverse cannot be univalent. }}

\noindent{\em{Proof.}}  If $N$ is sufficiently large given $\delta $, then every component of $f^{-N}(\overline{\mathbb C}\setminus G_0)$ either  has   diameter $<\delta $, or is within the $\delta $-neighbourhood of some Fatou component. This is simply because, if $B_1$ is any closed set, and $S$ is any univalent local inverse of $f^n$ defined on an open set $B_2$ containing $B_1$, then the diameter of $S(B_1)$ tends to zero uniformly with $n$, independent of $S$.  The proof is from Brolin \cite{Br}. If the diameter does not tend to $0$ then by Montel's Theorem there is an open neighbourhood $B_3$ of $B_1$ with $\overline{B_3}\subset B_2$ and a subsequence $S_{n_i}$ such that $S_{n_i}B_3$ converges to a set bounded from $0$ and such that $f^{n_i-n_{i-1}}\circ S_{n_i} =S_{n_{i-1}}$. This is only possible if $S_{n_i}B_3$ converges to a subset of the full orbit of a Siegel disk or Herman ring, neither of which exists, under our assumptions on $f$. In fact, our assumptions ensure that we can take $B_2$ to be any open set which is disjoint from the closures of the critical forward orbits. In particular, we can take $B_2$ to be a sufficiently small neighbourhood of the closure $B_1$ of any  component of $\overline{\mathbb C}\setminus G_0$ which does not contain a periodic Fatou component. We can also take $B_1$ to be any closed simply-connected set in $\overline{W_1\setminus W_2}$ for any component $W_1$ of  $\overline{\mathbb C}\setminus G_0$ and periodic Fatou component $W_2$ with $W_2\subset W_1$, and in the complement of a neighbourhood of the set of parabolic points. We now assume that $N=N_0k_0$ for some $N_0$ sufficiently large, and with $k_0$ sufficiently large given $N_0$, in senses to be specified later.
 It follows from the fact that $G_0$ satisfies the properties of \ref{1.3}, that $f^{-N_0}(G_0)$ satisfies the properties of $\Gamma $ of \ref{1.3} if $\delta $ is sufficiently small given $\delta _1$. So, for $\delta _1<\varepsilon /2$, we choose 
$$G_{1,N_0}\subset f^{-N_0}(G_0)\cap B_{\delta _1}(G_0)$$ 
as in \ref{1.3}. In particular, $G_{1,N_0}$ is isotopic to $G_0$, and the isotopy can be performed within a $\delta _1$-neighbourhood of $F(G_0)$. We assume that $\delta _1$ is sufficiently small that the $2\delta _1$- neighbourhood of $F(G_0)$ is  contained in $U$.  We assume without loss of generality that $U$ is  a tubular neigbourhood of $G_0$ with piecewise $C^1$ boundary.  We then define $G_{1,iN_0}$ for $2\le i\le k_0$, and an isotopy of $G_{1,(i-1)N_0}$ to $G_{1,iN_0}$ inductively by: $G_{1,iN_0}\subset f^{-N_0}(G_{1,(i-1)N_0})$ is the image of $G_{1,(i-1)N_0}$ for the isotopy which is the lift under $f^{N_0}$ of the isotopy between $G_{1,(i-2)N_0}$ and $G_{1,(i-1)N_0}$. Then $G_1=G_{1,N}=G_{1,k_0N_0}$.

 It remains to choose $U_1$ and the piecewise $C^1$ homeomorphism   $h$ of $\overline{\mathbb C}$ mapping $G_0$ to $G_1$ such that $h$ is a piecewise $C^1$ homeomorphism, and we have the required expanding properties of $f^N\circ h$. For this, it suffices to bound  the derivative of $h$ on $h^{-1}(U_{1,N})$ from $0$, independently of $N$, because the minimum of the derivative of $f^N$ on $f^{-N}(U)$ tends to $\infty $ with $N$. We will choose $h$ to be the identity outside $U$. 
 
 Choose a finite   partition ${\mathcal{R}}_0$ of   $\overline{U}$ into topological rectangles with piecewise $C^1$ boundary such that each rectangle is a square, up to bounded distortion, has two edges in $U$, and intersects $G_0$ in a single arc in an edge of $G_0$. 
We have $G_{1,iN_0}\subset f^{-iN_0}(G_0)$. The interior of the union of the rectangles in ${\mathcal{R}}_0$ is of course the set $U$.  We also write $U=U_{1,0}$. Let $U_{1,iN_0}$ be the union  of components of $f^{-iN_0}(R)$, for $R\in {\mathcal{R}}_0$, which  intersect $G_{1, iN_0}$, including those which intersect $G_{1,iN_0}$ in just a single point,  adjacent  to a component intersecting $G_{iN_0}$ in an edge. 

\Comment{\color{red}{``but add in'' corrected to ``including''
 and ``to an adjacent' corrected to ``adjacent to a''. $G_{iN_0}$ corrected to $G_{1,iN_0}$ a few times. Further corrections made in response to comment (16). As for comment (17), the proof has been extensively rewritten. I am sorry I have  not provided diagrams. I simply have not had time, despite working on this pretty much full-time since February and close to that since December. 
}}

There might be some sets in $f^{-iN_0}(\mathcal{R}_0)$, which intersect $G_{1,iN_0}$ in a single point which is a vertex of $f^{-iN_0}(G_0)$ but not a vertex of $G_{1,iN_0}$. All the sets in $f^{-iN}({\mathcal{R}}_{0})$ which intersect $G_{1,iN_0}$ in a nontrivial arc in an  edge are then topological rectangles, if we regard the edges of the rectangle  as the two components  of intersection with  $\partial U_{1,iN_0}$, and the two complementary  components in the boundary. We call these sets of $f^{-iN_0}(\mathcal{R}_0)$ {\em{first type sets}} and the others are {\em{second type sets.}} A second type set $R$ intersects $G_{1,iN_0}$ in a single point in $\partial R$ which is a vertex of $f^{-iN_0}(G_{1,0})$ but not a vertex of $G_{1,iN_0}$. The set $\partial R\cap \partial U_{1,iN_0}$ is connected. 
 We can assume that second type sets in   $f^{-iN_0}(\mathcal{R}_0)$ are always adjacent to rectangles of the first type. and then add each set $R$ of the second type to an adjacent one of the first type, say $R'$. We also combine the set $\partial R\cap \partial U_{1,iN_0}$ to the adjacent component of $\partial R' \cap \partial U_{1,iN_0}$ so that  the combined set is a topological rectangle with two edges in $\partial U_{1,iN_0}$. We write  ${\mathcal{R}}_{iN_0}$ for this set of topological rectangles, each of which maps forward under $f^{iN_0}$ to either a topological rectangle in ${\mathcal{R}}_0$ or a union of two topological rectangles in ${\mathcal{R}}_0$ whose boundaries intersect in an arc including a vertex of $G_0=G_{1,0}$. Thus the number of possibilities for the images under $f^{iN_0}$ of rectangles in ${\mathcal{R}}_{iN_0}$ is bounded in terms of the number of sets in ${\mathcal{R}}_0$.  If $N_0$ is sufficiently large, we have $U_{1,N_0}\subset \mbox{int}(U_{1,0})$, and then $U_{1,(i+1)N_0}\subset \mbox{int}(U_{1,iN_0})$ for all $0\le i<k_0$. By construction,  $U_{1,iN_0}$ is a closed  neighbourhood of $G_{1,jN_0}$ for all $j\ge i$. In particular, $U_1=U_{1,N}$ is a neighbourhood of $G_1=G_{1,N}$. 

\Comment{\color{red} {$U_j$ changed to $U_{1,j}$ as we need $U_1=U_{1,N}$. Below $U_{iN_0}$ is corrected to $U_{1,iN_0}$ a few times. }}

 To construct $h$, we first construct two foliations ${\mathcal{F}}_0$ and ${\mathcal{F}}_1$ of $U_{1,0}=U$. The map $h$ will then map leaves of ${\mathcal{F}}_0$ to leaves of ${\mathcal{F}}_1$. For ${\mathcal{F}}_0$, for each rectangle $R$, there will be a piecewise $C^1$ homeomorphism from $[a(R),b(R)]\times [-1,1]$ to $R$, with derivative bounded and bounded from $0$, where it is defined, such that the leaves of the foliation are the images of the sets $\{ x\} \times [-1,1]$, and the arc of an edge of $G_0$ in $R$ is the image of $[a(R),b(R)]\times \{ 0\}$, and $R\cap \partial U$ is the image of $[a(R),b(R)]\times \{ 1,-1\} $. Thus each leaf of ${\mathcal{F}}_0$ in $R$ crosses $G_0$ exactly once. It is clear that we can construct ${\mathcal{F}}_0$ because it just depends on the homeomorphisms between rectangles in $U$ and $[a(R),b(R)]\times [-1,1]$. It is convenient to choose $a(R)$, $b(R)$ and the homeomorphism between the topological and geometric rectangles as follows. For each edge $e$ of $G_0$, there is a finite number $n(e)$ of topological rectangles in ${\mathcal {R}}_0$ containing $e$. Number these rectangles $R_i(e)$ for $1\le i\le n(e)$ in the order in which they are placed along $e$. Choose a  component $\partial _1R_i(e)$ of $\partial R_i(e)\cap \partial U_{1,0}$ so that $\partial _1R_i(e)$ is  in the same component of $\partial U_{1,0}$ for all $i\le n(e)$. Then for $R=R_i(e)$ for some $i$, let $b(R)-a(R)$ be the length of $\partial _1R$ in the spherical metric and let $a(R_1(e))=0$ and $a(R_{i+1}(e))=b(R_i(e))$ for $1\le i<n(e)$. Then choose the piecewise $C^1$ homeomorphism from $R$ to $[a(R),b(R)]\times [-1,1]$ to map spherical  length on $\partial _1R$ to Euclidean length on $[a(R),b(R)]$ and to have a uniform bound on derivative. 
 
 We make a similar parametrisation of the sets in ${\mathcal{R}}_{iN_0}$, all of which are topological rectangles. Each one intersects in its interior a single arc from a single edge of $G_{1,iN_0}$ which is  ambient isotopic to an edge $e$ of $G_0=G_{1,0}$. We write $R_j(e,i)$, for $1\le j\le n(e,i)$ for the rectangles in ${\mathcal{R}}_{iN_0}$ which intersect $G_{1,iN_0}$ in an arc from the edge which is ambient isotopic to $e\subset G_0$. We choose the component $\partial _1R$ of $\partial R\cap \partial U_{1,iN_0}$,  for $R=R_j(e,i)$, so that $\partial _1R$ is in the component of $\partial U_{1,iN_0}$ which is not separated from $\partial _1R'$ by $G_{1,iN_0}$, for $R'=R_k(e)\in{\mathcal{R}}_0$ for any $k$. Then for $R=R(e,i)$, we let $b(R)-a(R)$ be the spherical length of $\partial _1R$ and we choose a piecewise $C^1$ homeomorphism from $R$ to $[a(R),b(R)]\times [-c(R),c(R)]$ where $c(R)$ is the length of one of the  components of $\partial R\cap U_{1,iN_0}$. It does not matter which one because the lengths of the two components differ by at most a bounded multiplicative constant. We also choose the homeomorphism to map $G_{1,iN_0}$ to $[a(R),b(R)]\times \{ 0\} $ We can choose this piecewise $C^1$ homeomorphism to have distortion bounded independently of $i$ and $j$, because $f^{iN_0}$ is univalent from a neighbourhood of $R_j(e,i)$ onto the neighbourhood of either a rectangle of ${\mathcal{R}}_0$ or the union of two rectangles of ${\mathcal{R}}_0$. 
 
 So now we consider ${\mathcal{F}}_1$.  Each  leaf of ${\mathcal{F}}_1$ in $U$ will cross $G_1$ exactly once, and will also cross $\partial U_{1,iN_0}$ exactly twice for each $0\le i\le k_0$, once on each side of $G_1$. 
 Each leaf of ${\mathcal{F}}_1$ will have the same two endpoints as some leaf of ${\mathcal{F}}_0$. Intersections of leaves with $\partial U_{1,iN_0}$ will be transverse for each $1\le i\le k_0$. For each edge $e$ of $G_0$ and each $i<k_0$, leaf segments from $x\in \partial _1R_k(e,i)$ in $U_{1,iN_0}\setminus U_{1,(i+1)N_0}$ will end at $\tau (x)\in \partial _1R_\ell (e,i+1)$ for some $\ell \le n(e,i+1)$. The endpoint $\tau (x)$ is determined from $x$ by using the parametrisations of $\bigcup _{k\le n(e,i)}\partial_1R_k(e,i)$ and $\bigcup_{\ell \le n(e,i+1)}\partial _1R_\ell (e,i+1)$ by $[a_1(e,i),b_{n(e,i)}(e,i)]$ and $[a_1(e,i+1),b_{n(e,i+1)}(e,i+1)]$ respectively. We choose the map $\tau $ to be of the form $t\mapsto \beta t$ which respect to this parametrisation. Thus, $\tau $ multiplies length by a constant depending only on $e$ and $i$. We take the same parametrisation on the other components of $\partial R_k(e,i)\cap \partial U_{1,iN_0}$ and $\partial R_\ell (e,i+1)\cap\partial U_{1,(i+1)N_0}$ and again use the map $t\mapsto \beta t$ to determine the endpoints of the leaf segments of leaves in ${\mathcal{F}}_1$ in this component of $U_{1,iN_0}\setminus U_{1,(i+1)N_0}$. This time the map does not multiply length by a constant, but does do so up to bounded distortion. We then foliate each component of $U_{1,iN_0}\setminus U_{1,(i+1)N_0}$ by leaf segments with these endpoints. For adjacent edges $e$ and $e'$, where $\partial R_{n(e,i)}(e,i)\cap \partial R_1(e'i)\ne \emptyset $ (or similarly with $e$ and $e'$ interchanged) then we need to make sure that we choose the segment joining the point $\partial _1R_{n(e,i)}(e,i)\cap \partial _1R_1(e',i)\cap \partial U_{iN_0}$ to  $\partial _1R_{n(e,i+1)}(e,i+1)\cap \partial _1R_1(e',i+1)\cap \partial U_{(i+1)N_0}$ extends $C^1$  continuously on both sides in  the component of  $U_{1,iN_0}\setminus U_{1,(i+1)N_0}$, and similarly on the other components of   $\partial R_{n(e,i)}(e,i)\cap \partial R_1(e',i)\cap \partial U_{iN_0}$ and 
$\partial R_{n(e,i+1)}(e,i+1)\cap \partial R_1(e',i+1)\cap \partial U_{(i+1)N_0}$. In $U_{1,k_0N_0}$ the leaves of ${\mathcal{F}}_1$ in $R$ are simply the images of the lines $\{x\}\times[-c(R),c(R)]$ in the rectangles.  Because all rectangles in ${\mathcal{R}}_{iN_0}$ and ${\mathcal{R}}_{(i+1)N_0}$ map forward under $f^{iN_0}$  to either rectangles or the union of two rectangles in ${\mathcal{R}}_{0}$ and ${\mathcal{R}}_{N_0}$, we can choose the leaves of $\mathcal{F}_1$ to  be piecewise $C^1$ with bounded derivative independent of $i$. Moreover, the length of leaves of ${\mathcal{F}}_1$ is bounded independently of $k_0$.

  Since $h$ is to be the identity on $\partial U$,  we know exactly which leaf  of ${\mathcal{F}}_1$ is the image under $h$ of any given leaf in $L\in {\mathcal{F}}_0$. We also choose $h$ to map the point $L\cap G_0$ to $h(L)\cap G_{1,k_0N_0}$. Now we claim that the derivative of $h:G_0\to G_{1,k_0N_0}$ is bounded from $0$ independently of $k_0$. This follows from the lower bound of the derivative of the map  along leaves of ${\mathcal{F}}_1$ from $\partial _1R$ to $\partial _1R'$ for $R\in{\mathcal{R}}_0$ and $R'\in{\mathcal{R}}_{k_0N_0}$ for $R=R_j(e)$ and $R'=R_\ell (e,k_0)$ for some $k_0$. Whenever there are such leaves they have been chosen so that the map along leaves from $\partial _1R$ to $\partial _1R'$ multiplies length by a constant depending only on lengths of $\bigcup _{i\le n(e)}R_i(e)$ and $\bigcup_{i\le n(e,k_0)}R_i(e,k_0)$. Since the first of these has length bounded above and below and the second has length bounded below, and the rectangles in ${\mathcal{R}}_{k_0N_0}$ map under $f^{k_0N_0}$ to rectangles or unions of two rectangles in ${\mathcal{R}}_0$ with bounded distortion, we obtain the required lower bound on  the derivative of $h:G_0=G_{1,0}\to G_1=G_{1,k_0N_0}$. To obtain the  required lower bound of the derivative of $h$ on $h^{-1}(U_1)$, we simply choose the segments of leaves of ${\mathcal{F}}_0$ containing $G_0$ which map to $U_1$ to be sufficiently short that the derivative of $h$ along these segments is $\ge 1$: we have this freedom in choosing $h^{-1}(U_1)$ and $h$ The bound on the diameter of leaves of ${\mathcal{F}}_0$ and ${\mathcal{F}}_1$, independent of $k_0$, gives the bound on the diameter of  $\{ \theta _t(x):t\in [0,1]\} $ for a suitable $C_1$ and isotopy $\theta _t$ between the identity and $h$..

 % \ep

\subsection{Proof of Theorem \ref{1.1} for some $N$.}\label{1.5}

Let $G_0$ and $G_1$ be the graphs as in Lemma \ref{1.4}, and let $U_1=U_{1,N}=U_{1,k_0N_0}$, and  $h$ be as in Lemma \ref{1.4}, and $U=U_0$, so that $ f^N\circ h=g$ is expanding on $G_0$, and $h$ is  the identity outside $U$ and isotopic to the identity on $\overline{\mathbb C}$ via an isotopy which is the identity outside $U$. More generally, let $U_{1,iN_0}$ be as in \ref{1.4}  and write $U_n=U_{1,nN}$, remembering that $N=k_0N_0$. So $U_{n+1}\subset U_n$ for all $n$ and $f^N(U_{n+1})=U_n$. We also choose a neighbourhood $U_n'$ of $U_n$ for $n\ge 1$ with $U_n'\subset U_{n-1}$ as follows. Recall froom \ref{1.4} that $U_1$ is a finite union of topological rectangles and that $U_n$ is  covered by rectangles of the form $SR$ where $S$ is a univalent  local inverse of  $f^{(n-1)N}$  on $R$, and $R$ is one of the rectangles in the finite union. We then let $U_1'$ be a finite union of  open balls $B$ covering $U_1$ such that $U_1'\subset U_0$ and such that any local inverse $S$ of $f^{(n-1)}N$ on $B$ is univalent, and let $U_n'$ be the union of such balls $SB$ such that $SB\cap U_n\ne \emptyset $. We thus have 
$$U_n\subset U_n'\subset U_{n-1}.$$
Recall that $d$ is the spherical metric on $\overline{\mathbb C}$. Now we are going to define homeomorphisms $h_n$ on $\overline{\mathbb C}$ such that $h_n$ is the identity outside $U$ and such that 
\begin{equation}\label{1.5.3}\mbox{Max}\{ d(h_n(x),x):x\in \overline{\mathbb C}\} \le C_1\lambda ^n,\end{equation}
where $C_1$ is as in \ref{1.4}, 
\begin{equation}\label{1.5.4}h_n\circ f^N=f^N\circ h_{n+1}\mbox{ on }U_{n+1}\mbox{ for }n\ge 0,\end{equation}
\begin{equation}\label{1.5.5}h_n=\mbox{ identity\ \ outside\ \ } U_{n}'\mbox{ for }n\ge 1.\end{equation}

\Comment{\color{red}{Comment(18) $U_n$ should be $U_{1,nk_0N_0}=U_{1,N}$ (since $k_0N_0=N$) where $U_{1,iN_0}$ is as in Lemma 1.4 -- renamed from $U_{iN_0}$ in version 5. This has now been made clear.
 
 Comment (19) $h_{n+1}$ is now defined on $U_{n+1}$ by $f^N\circ h_{n+1}=h_n\circ f^N$ on $U_{n+1}\subset f^{-N}(U_n$, and by being isotopic to the identity on $U_{n+1}$, using that $h_n$ is isotopic to the identity on $U_n$. The claim that $h_n$ can be the identity outside $U_n$ is dropped, although it is chosen to be the identity outside $U_{n-1}$ for $n\ge 1$.}} 

 Write   $h=h_0$. Then $h_0$ is the identity outside $U=U_0$ and isotopic to the identity. Now let $n\ge 0$ and suppose that $h_{n}$ has been defined satisfying the hypotheses. In particular $h_n$ is the identity outside  $U_n'$. Now $f^N:f^{-N}(U_{n-1})\to (U_{n-1})$ is a covering map.  Define  $\widetilde{h}_n:f^{-N}(U_{n-1})\to f^{-N}(U_{n-1})$ by the properties 
 \begin{equation}\label{1.5.7}f^N\circ \widetilde{h}_n=h_{n}\circ f^N\mbox{ on }f^{-N}(U_{n-1})\end{equation}
  and $\widetilde{h}_n$ is isotopic to the identity on $f^{-N}(U_{n-1})$. This implies that $\widetilde{h}_n=\mbox{identity}$ on $\partial f^{-N}(U_{n-1})$. As $\lambda ^{-1}>1$ is the expansion constant of $f^N$ on $U_1$, we have 
\begin{equation}\label{1.5.6}d(x,\widetilde{h}_n(x))\le \lambda d(f^N(x),h_{n}(f^N(x)))\le C_1\lambda ^{n+1}\end{equation}
for all $x\in f^{-N}(U)$, for $C_1$ as in \ref{1.4}. Also, by the inductive hypothesis, $\widetilde{h}_n$ is the identity outside the $f^{-N}(U_n')\subset f^{-N}(U_{n-1})$.  So $\widetilde{h}_n$ is equal to the identity on $\partial U_{n+1}'\cap \partial f^{-N}(U_{n}')$. So  we can define a  homeomorphism $h_{n+1}$ by $h_{n+1}=\widetilde{h}_n$  on $U_n$ except on $SB\setminus U_{n+1}$ where $SB$ is one of the balls covering $U_{n+1}'$ which intersects $\partial U_{n+1}\cap \mbox{int}(f^{-N}(U_n))$. Then we can  extend $h_{n+1}$ from $U_{n+1}$ over such balls so that   (\ref{1.5.3}) to  (\ref{1.5.5}) hold with $n+1$ replacing $n$, using (\ref{1.5.6}) and (\ref{1.5.7})  

\Comment{\color{red}{Comment (20) agreed: $G_n=h_n(G_{n-1})$ replaced by $G_{n+1}=h_n(G_{n})$. Comment (21): the rest of this proof has now been rewritten.}}
 
 Write $f^N\circ h_0=g$ and $G_n=G_{1,nN}$ where $G_{1,iN_0}$ is as in \ref{1.4}. Then  $G_{n+1}=h_n(G_{n})$.  Let 
   $\varphi _n=h_n\circ \cdots h_0$  for all $n\ge 0$. Then
   $$d(\varphi _n(x),\varphi _{n-1}(x))\le C_1\lambda ^n$$
   for all $x\in \mathbb C$. It follows that $\varphi _n$ converges uniformly on $\overline{\mathbb C}$ to a continuous map $\varphi :{\overline{\mathbb C}}\to {\overline{\mathbb C}}$. The set $G'=\varphi (G_0)$ is then the required graph with $G'\subset f^{-N}(G')$, provided that $\varphi $ is a homeomorphism. We have 
  \begin{equation}\label{1.5.9}\varphi _n^{-1}(U_n)=\varphi _{n-1}^{-1}(h_n^{-1}(U_n))\subset \varphi _{n-1}^{-1}(U_n')\subset \varphi _{n-1}^{-1}(U_{n-1}).\end{equation}
  Using (\ref{1.5.9}) for  $i$ replacing $n$ for each $k\le i\le n$ and also using (\ref{1.5.4}) with $i$ replacing $n$, for $k-1\le i\le n-1$, we have, for each $0\le k\le n-1$,
 
  \begin{equation}\label{1.5.8}f^N\circ \varphi _n=h_{n-1}\circ \cdots \circ h_k\circ f^N\circ \varphi _k\mbox{ on }\varphi _n^{-1}(U_n)\end{equation}
  This gives  
\begin{equation}\label{1.5.1}f^N\circ \varphi _n=\varphi _{n-1}\circ g\mbox{ on }\varphi _{n}^{-1}(U_n),\end{equation}
 for all $n\ge 1$. Hence 
 \begin{equation}\label{1.5.9}f^{kN}\circ \varphi _n=\varphi _{n-k}\circ g^k\mbox{ on }\varphi _n^{-1}(U_n).\end{equation}
 In particular
 \begin{equation}\label{1.5.10}f^{nN}\circ \varphi _n=h\circ g^n\mbox{ on }\varphi _n^{-1}(U_n).\end{equation}
 and
 $$g^n(\varphi _n^{-1}(U_n))=h^{-1}(U_0)=U_0$$
 Since  by \ref{1.4} any local inverse of $g$ on $U_0$ mapping into $U_0$ has contraction constant $\lambda $,  we deduce that
$$\bigcap _k\varphi _k^{-1}(U_{k})=G'$$
Taking limits in (\ref{1.5.1}), we obtain
 $$f^N\circ \varphi =\varphi \circ g\mbox{ on }G_0.$$

 Now since $\varphi $ is the limit of the homeomorphisms $\varphi _n$, it follows that $\varphi $ is a monotone map, that is, $\varphi ^{-1}(x)$ is connected for all $x$. (This is because we can find a sequence $\varepsilon _n$ decreasing to $0$ such that $\varphi_n^{-1}(\{ y:d(x,y)\le \varepsilon _n\} )$ is a decreasing sequence of connected sets which is equal to $\varphi ^{-1}(x)$.) We have $h_n=\mbox{identity}$ outside $U_{n-1}$ for $n\ge 1$ and hence $\varphi_n=\varphi _k=\varphi $ outside $\varphi _k^{-1}(U_{k})$ for all $n\ge k$, and $\varphi $ is a homeomorphism except possibly  on $\bigcap _k\varphi _k^{-1}(U_{k})=G'$. So if $\varphi $ is not a homeomorphism then $\varphi ^{-1}(x)$, for some $x\in G'$, is a nontrivial connected set in $G_0$. Then $\varphi ^{-1}(f^{nN}(x))=g^{n}(\varphi ^{-1}(x))$ is nontrivial connected for each $n\ge 0$ and since $g$ is expanding on $G_0$ we obtain that $\varphi $ is constant on a subgraph of  $G'$ which separates $F(G_0)$, an obvious contradiction.

\subsection{Nested sequences of arcs.}\label{1.6}

We are now ready to start studying intersections between $f^{-i}(G')$ and $f^{-j}(G')$ for $0\le i,j<N$ and $i\ne j$.  Since $f^i:f^{-i}(G')\to G'$ is a finite  covering and $G'$ is a finite graph, $f^{-i}(G')$ is a finite graph  for all $i\ge 0$.

\Comment{\color{red}{Comment(35): Definitions of arc  added in in the introduction. Comment (36) is a general comment about the nature of $f^{-i}(G')$, Hopefully the nature of $f^{-i}(G')$ has been clarified above. }}

Fix $\varepsilon _0$ suitably small that the distance between any two vertices of $f^{-k}(G')$ is $\ge 4\varepsilon _0$, for $0\le k<N$.  Let $0\le i,j<N$, with $i\ne j$. A sequence of arcs $\gamma _n\subset f^{-j}(G')$ ($n\ge n_0$), with $\gamma _n$ of diameter $\le \varepsilon _0$ for all $n$, such that $\gamma _n$ is disjoint from $f^{-i}(G')$ apart from having both endpoints in $f^{-j}(G')\cap \eta $, where $\eta \subset f^{-i}(G')$ has diameter $\le \varepsilon _0$,
 and $\gamma _{n+1}$ is inside the topological disc bounded by $\gamma _n$ and $\eta $, is called a {\em{nested sequence of arcs for $f^{-i}(G')$ and $f^{-j}(G')$}}.
 
\Comment{\color{red}{The definition above of nested sequence of arcs has been rewritten. Comment (23)  $f^{-j}(G')$ changed to $f^{-j}(G')\cap\eta $  Restrictions on $\varepsilon _0$ made more precise. Comments (22), (25): ``is'' added in (twice)  Comment (24) $f{-j}(G')$ corrected to $f^{-j}(G')$. }}

More generally we will talk about a {\em{nested sequence of arcs for $(f^{-i}(G'),f^{-\ell }(G'))$ and $f^{-j}(G')$}} if  $\gamma _n\subset f^{-j}(G')$ is disjoint from $f^{-i}(G')\cup f^{-\ell }(G')$ apart from having one endpoint in  $f^{-j}(G')\cap f^{-i}(G')$ and the other in $f^{-j}(G')\cap f^{-\ell }(G')$
  and $\gamma _{n+1}$ is inside the topological disc bounded by $\gamma _n$ and $\eta $, where $\gamma _n$ has diameter $\le \varepsilon _0$ for all $n$,  $\eta =\eta ^1\cup \eta ^2$ has diameter $\le \varepsilon _0$ and where $\eta ^1\subset f^{-i}(G')$ and $\eta ^2\subset f^{-\ell }(G')$ are disjoint apart from a common endpoint in $f^{-i}(G')\cap f^{-\ell }(G')$. 

\Comment{ \color{red}{ Similar corrections to previous paragraph. $\gamma _n$ expanded to $\gamma _n\subset f^{-j}(G')$, $f^{-j}(G')\cap f^{-\ell }(G')$ corrected  from $f^{-\ell }(G')$, $f^{-i}(G')\cap f^{-\ell }(G')$ corrected  from  $f^{-i}(G')\cup f^{-\ell }(G')$. Condition on diameter of $\gamma _n$ added in.}}

Locally, and abstractly, it is not difficult to construct nested sequences of arcs, and it might well be possible to make the constructions in our context. Our strategy will be to isolate nested sequences before replacement to make another graph.  First we need some analysis of them. 
 
\begin{lemma}\label{1.7}Fix $0\le i,j<N$ with $i\ne j$  and any sufficiently small $\varepsilon >0$. 
\begin{enumerate} 
\item There is a finite set $Y_1$ of periodic points (under $f$) in $\bigcup _{i=0}^{N-1}f^{-i}(G')$ such that the following holds. Let $0\le i,j<N$ with $i\ne j $ and let $\gamma $ be any arc with  $\mbox{int}(\gamma )\subset f^{-j}(G')\setminus f^{-i}(G')$  and at least one endpoint of $\gamma $, $x_1$, is in $f^{-i}(G')\cap f^{-j}(G')$. Then there is $n>0$ and  such that $f^{n}(x_1)\in Y_1$. 

\Comment{\color{red}{COMMENT Comment (26) of the report is clearly regarded as important. It looked all right to me but I think there must have been some ambiguity about what was written. So I have replaced the finite set $Y$, which appears in the proof, by $Y_0$ in this statement. $Y_0$ is the set of periodic orbits in the forward orbit of the set $Y$ of  eventually periodic points. }}

\item There is a finite set ${\mathcal{C}}$ of pairs $(\gamma ',\{ x_1',x_2'\} )$ of arcs $\gamma '$ with endpoints $x_1'$ and $x_2'$  such that the following hold. Suppose that  $\gamma $ is an arc of  diameter $< \varepsilon$ with   $\mbox{int}(\gamma )\subset f^{-j}(G')\setminus f^{-i}(G')$  and with endpoints $x_1$ and $x_2$. Then there is $n>0$ and $(\gamma ',\{ x_1',x_2'\})\in{\mathcal{C}}$ such that $f^{nN}(\gamma )=\gamma '$ and $\{ f^{nN}(x_1),f^{nN}(x_2)\} =\{ x_1',x_2'\} $.  

Consequently, if $\gamma _n$ ($n\ge n_0$) is any nested sequence of arcs, and $\gamma _n$ has endpoints $x_{n,1}$ and $x_{n,2}$ and $\gamma _n$ has diameter $\le \varepsilon $ for $n\ge n_0$, then  there are $n$ and $n+k$ with $n_0\le  n<n+k\le n_0+ \#({\mathcal{C}})$ and $m$ and $\ell >0$ and $(\gamma ',\{ x_1',x_2'\} )\in {\mathcal{C}}$ with $f^{mN}(\gamma _n)=f^{(m+\ell )N}(\gamma _{n+k})=\gamma '$,  and $f^{mN}(x_{n,t})=f^{(m+\ell )N}(x_{n+k,t})=x_t'$ for $t=1$, $2$, and some ordering  of endpoints such that $x_{n+k,1}$ is nearer to $x_{n,1}$ than the other endpoint of $\gamma _n$.

\Comment{\color{red}{
Added in in the last paragraph that $x_{n,1}$ and $x_{n,2}$ are the endpoints of $\gamma _n$.

Corrections in punctuation in the last sentence. 

Comment (27): I don't think I want to replace ``with'' by ``apart from''  

Comment(28): $(\gamma ', x_1',x_2')$ replaced by $(\gamma ',\{ x_1',x_2'\})$ 

Comment (29) $n\le n_0$ replaced by $n\ge n_0$

 Comment (30)$(\gamma ',x')$ changed to  $(\gamma ',\{ x_1',x_2'\} )$. }}

Moreover, either $f^{mN}(\gamma _p)$ ($p\ge n$) is a nested sequence of arcs for $f^{-i}(G')$ and $f^{-j}(G')$, or $f^{mN}(\gamma _n)$ is in a small neighbourhood of a vertex of $f^{-i}(G')$.

\Comment{\color{red}{Comment (31) ``of'' added after ``vertex''. Also $f^{mN}(\gamma )$ is corrected to $f^{mN}(\gamma _n)$.}}

\item Similarly, for fixed distinct $0<i,\ell,j<N$  and $\varepsilon >0$, there is a finite set ${\mathcal{C}}'$ of pairs 
$(\gamma',\{ x_1',x_2'\} )$ of arcs $\gamma '$ with endpoints $x_1'$ and $x_2'$ such that the following hold.  If $\gamma _n$ is any nested sequence of arcs for $(f^{-i}(G'),f^{-\ell }(G'))$ and
 $f^{-j}(G')$, and $\gamma _n$ has diameter $\le \varepsilon $ for $n\ge n_0$, then there are $n$ and $n+k$ with $n_0\le  n<n+k\le n_0+\#({\mathcal{C}}')$ and $m$ and $\ell >0$ and $(\gamma ',\{ x_1',x_2'\})\in {\mathcal{C}}'$ with $f^{mN}(\gamma _n)=f^{(m+\ell)N}(\gamma _{n+k})=\gamma '$ and $f^{mN}(x_{n,t})=f^{(m+\ell)N}(x_{n+k,t})=x_t'$, where $x_{n,1}'$, $x_{n+k,1}'$ and $x_1'$ are the respective endpoints in $f^{-i}(G')$ and $x_{n,2}'$, $x_{n+k,2}'$ and $x_2'$ are the respective endpoints in $f^{-\ell }(G')$. 

  \Comment{\color{red}{comment (34)$x_{n+k,1})$ replaced by $x_{n+k,1}'$ Comment(32) $(\gamma ', x_1',x_2')$ replaced by $(\gamma ',\{ x_1',x_2'\})$  and ${\mathcal{C}}$ replaced by ${\mathcal{C}}'$, twice. Comment(33) 
$f^{m+\ell}$replaced by $f^{(m+\ell)N}$.}}

\end{enumerate}

 \end{lemma}
\begin{remark}
\item Note that the statements of 2 and  3 are effectively  about {\em{finite}} nested sequences. It is possible that the set of limits of {\em{infinite}} nested sequences is uncountable. 
\end{remark}

\noindent{\em{Proof.}} 

\noindent  We choose a constant $C$ so that the maximum of the spherical derivative of $f^N$ is $\le C$. We choose $\varepsilon $ such that the distance between any two vertices of $f^{-i}(G')$ is $\ge 100C\varepsilon $, for $0\le i\le 2N$, and similarly for the minimum distance between any two edges of  $f^{-i}(G')$ which do not have a common vertex.  

\Comment{\color{red}{Restrictions on $\varepsilon $ made more precise -- and taking into account comment (50).}}

Let  $Y$ be  the set of all endpoints, in $f^{-i}(G')\cap f^{-j}(G')$, of arcs of diameter $\ge \varepsilon $ and $\le C\varepsilon $ in  $f^{-j}(G')\setminus f^{-i}(G')$  but with at least one endpoint in $f^{-i}(G')\cap f^{-j}(G')$. The set $Y$ is finite,  because there is a lower bound, depending on $G'$, on the Hausdorff distance between two such disjoint arcs.  

 The set ${\mathcal{C}}$ contains all  $(\gamma ,\{ x_1,x_2\} )$  such that $\gamma $ is an arc  in $f^{-j}(G')\setminus f^{-i}(G')$ apart from endpoints $x_1$ and $x_2$, with  endpoint $x_1\in f^{-i}(G')$, and  spherical diameter $\ge \varepsilon $ and $\le C\varepsilon $. The set of such arcs is finite because each edge of $f^{-j}(G')$ is an arc.  The remaining pairs  
 $(\gamma ,\{x_1,x_2\} )$ in ${\mathcal{C}}$ are of the form $\gamma =\gamma _1\cup \gamma _3\cup \gamma _2$ (arcs in this order) such that:
 
\begin{itemize}
\item  the diameter of $\gamma $ is $\ge \varepsilon $ and $\le C\varepsilon $ ;
\item  $\gamma _1$ is either a single vertex of $f^{-i}(G')$  or  is an arc  in $f^{-j}(G')\setminus f^{-i}(G')$  apart from {\em{both}} endpoints in $f^{-i}(G')$ in different edges of $f^{-i}(G')$ which meet at a vertex $x$ of $f^{-i}(G')$ and similarly for $\gamma _2$;
\item $x_k$ is the endpoint of $\gamma $ which is also an endpoint of $\gamma _k$ for $k=1$ and $2$;

\Comment{\color{red} {Comment (39): $i$ replaced by $k$ as suggested. Relevant to comment (38) and also (41): ``near $x$'' replaced by which meet at a vertex $x$ of $f^{-i}(G')$ }}

\item $\gamma _3$  does not intersect the edges of $f^{-i}(G')$ (one or two edges) containing $x_1$ and $x_2$ --  unless $x_1=\gamma _1=x$, in which case $\gamma _3$ does not intersect the interior of the edge of $f^{-i}(G')$ containing $x_2$, and similarly if $x_2=\gamma _2$

\Comment{\color{red}{Comment (40): ``edges'' expanded to ``edges of $f^{-i}(G')$''. 
Comment (38) and (41): reference to $x$ has now been removed from this item and is in the item above it (second item). Since the larger  arc $\gamma $ has diameter $\le C\varepsilon $ and $\gamma _1$ an $\gamma _2$ contain points in different edges of $f^{-i}(G')$ meeting at the vertex $x$, $\gamma _3$ is automatically going to be in a small neighbourhood of the vertex $x$. 
``$\gamma _3\cap f^{-i}(G')$'' has been replaced by ``$\gamma _3$''.  
Comment (42) The last item has been expanded to take account of the case when $\gamma _1$ of $\gamma _2$ is a trivial arc at a vertex of $f^{-i}(G')$ and similarly for $\gamma _2$.}}\end{itemize}  

Once again this set of arcs on  $f^{-j}(G')$ is finite, because any two such arcs $\gamma _1$ are either equal or disjoint apart from endpoints, and similarly for $\gamma _2$, and the decomposition of $\gamma $ as $\gamma _1\cup \gamma _3\cup \gamma _2$ ( $\gamma _i$ in this order along $\gamma $) is canonical. 

\Comment{\color{red}{Comment (37): it is now emphasised that the decomposition $\gamma =\gamma _1\cup \gamma _3\cup \gamma _2$ is in this order . In the last sentence above the order has been corrected to this. Comment (43)``be'' and ``apart'' added in on the line below.}}

Now let  $\gamma $ be an arc in $f^{-j}(G')\setminus f^{-i}(G')$ apart  from, possibly,  the endpoints, 
 and the endpoint $x_1$ is in $f^{-i}(G')\cap f^{-j}(G')$. Since we are only interested in $x_1$ for the moment we can reduce $\gamma $ if necessary and assume that it has diameter $\le \varepsilon $. We want to show that $f^{nN}(x_1)\in Y$ for some $n>0$. Since this includes the case $x_1\in Y$, this will show that points in $Y$ are eventually periodic and will prove 1.  It suffices to show that there is $n>0$ such that  $f^{nN}(\gamma )$ contains an arc with endpoint at $f^{nN}(x_1)$ but otherwise in $f^{-j}(G')\setminus f^{-i}(G')$ and of diameter $\ge \varepsilon $. 
 
 \Comment{\color{red}{Comment (44) $f^{nN}(x)$ replaced by $f^{nN}(x_1)$. Definition of $Y_1$ added in.   }}

 Either there is $n>0$ such that  
 \begin{equation}\label{1.7.1}f^{rN}(\gamma )\setminus \{ f^{rN}(x_1)\} \subset f^{-j}(G')\setminus f^{-i}(G')\end{equation}
for  $0\le r\le n$ and $f^{nN}(\gamma )$ has diameter $\ge \varepsilon $ and $\le C\varepsilon $,  or for some least  $0\le r$, with $r<n$,   $f^{rN}(\gamma )$ has diameter $<\varepsilon $, (\ref{1.7.1}) holds, but 
$$f^{rN}(\gamma )\setminus \{f^{rN}(x_1)\} \cap f^{-i-N}(G')\setminus f^{-i}(G')\ne \emptyset. $$

\Comment{\color{red}{Comment (45) ``some least $0\le r<n$'' replaced by ``some least  $0\le r$, with $r<n$,''
Comment (46) It is now made clear at the start of this section  that  $f^{-i}(G')$ is a finite graph for each $i\ge 0$, and, at the start of this proof, that the distance between vertices of $f^{-i}(G')$ is $\ge 10C\varepsilon $ for $0\le i\le N$. This means that $f^{rN}(\gamma )$ cannot be the shape of a $\sigma $ }}
  
  There must be an arc of diameter $<\delta $ in $f^{-i-N}(G')$ which starts at  $x_3$, continues on to a vertex  $x_4$ of $f^{-i-N}(G')$ in $f^{-i}(G')$ before continuing in $f^{-i}(G')$ to   the endpoint  $f^{rN}(x_1)\in f^{-i}(G')$ of $f^{rN}(\gamma )$. Here $\delta $ can be taken arbitrarily small by taking $\varepsilon $ arbitrarily small. If $x_4$ coincides with $f^{rN}(x_1)$  then of course $x_1$ is eventually periodic.  Otherwise, let $x_5\in f^{-i-N}(G')\setminus f^{-i}(G')$ be the nearest such  point on $f^{rN}(\gamma )$ to $f^{rN}(x_1)$. Let $\gamma _1'$ be the subarc of $f^{rN}(\gamma )$ between $f^{rN}(x_1)$ and $x_5$.  Then $\gamma _1'$ is disjoint from $f^{-i-N}(G')$ apart from endpoints, and $f^N(\gamma _1')$ is disjoint from $f^{-i}(G')$ apart from endpoints, but bounds a disc together with an arc   of $f^{-i}(G')$ which contains the vertex $f^N(x_4)$ of $f^{-i}(G')$ which is eventually periodic. Then $f^{sN}(\gamma _1')$ cannot be close to a vertex of $f^{-i-N}(G')\setminus f^{-i}(G')$ until it has expanded out of a neighbourhood of  $f^N(x_4)$, by which time it has length $\ge \varepsilon $ and $<C\varepsilon $. If $n=r+s$ then $f^{sN}(\gamma _1')$ is the required subarc of $f^{nN}(\gamma )$ with endpoint at $f^{nN}(x_1)$ and the proof of 1 is completed.
  
  \Comment{\color{red}{Comment (47)``$x_1$ and $x_5$'' replaced by ``$f^{rN}(x_1)$ and $x_5$'' Comment (48) ``$f^{n}(\gamma )$ with endpoint at $f^{n}(\gamma )$'' replaced by $f^{nN}(\gamma )$ with endpoint at $f^{nN}(x_1)$. Comment (49) ``$\gamma \subset f^{-j}(G')\setminus f^{-i}(G')$'' replaced by ``$\mbox{Int}(\gamma ) \subset f^{-j}(G')\setminus f^{-i}(G')$'' (immediately below this). }}
 
 To prove 2, we start with $\mbox{Int}(\gamma ) \subset f^{-j}(G')\setminus f^{-i}(G')$ of diameter $\le \varepsilon $ and with both  endpoints $x_1$ and $x_2$ in $f^{-i}(G')\cap f^{-j}(G')$. Then if $f^{rN}(\gamma )$ intersects $f^{-i-N}(G')\setminus f^{-i}(G')$, we let $x_6$ be the point of $f^{rN}(\gamma )\cap f^{-i-N}(G')$ which is nearest to $f^{rN}(x_2)$ on $f^{rN}(\gamma )$, and $\gamma _2'$ is the arc of $f^{rN}(\gamma )$ between $f^{rN}(x_2)$ and $x_6$. Then $\gamma _3'$ is the arc of $f^{rN}(\gamma )$ between $\gamma _1'$ and $\gamma _2'$ and does not intersect $f^{-i}(G')$. It follows that $f^{(r+1)N}(\gamma )$ satisfies the conditions for an arc in ${\mathcal{C}}$ apart, possibly, from being of diameter $\ge \varepsilon $. Then, once again, $f^{(t +r)N}(\gamma )$ remains near a vertex of $f^{-i}(G')$ and hence bounded from edges   of $f^{-i-N}(G')\setminus f^{-i}(G')$ for $t\le s$ for the first $s$ such that $f^{(s +r)N}(\gamma )$ has diameter $\ge \varepsilon $.
 
  \Comment{\color{red}{$f^{(s +r)}(\gamma )$ corrected to $f^{(s +r)N}(\gamma )$Comment (50). Restrictions on $\varepsilon $ are now made at the start of the proof and incorporate Comment (50). The restrictions on $\varepsilon $ that were stated in the paragraph below have therefore been removed.}}

The set ${\mathcal{C}}'$ used to prove the result about nested sequences for $(f^{-i}(G'),f^{-\ell }(G'))$
and $f^{-j}(G')$ is a set of arcs  in $f^{-j}(G')$ of diameter $\ge \varepsilon $ and $\le C\varepsilon $ and with endpoints in $f^{-i}(G')$ and $f^{-\ell }(G')$.  Then the conditions for  $(\gamma ,\{ x_1,x_2\} )\in {\mathcal{C}}'$ are refined to either  
$$\gamma \setminus \{ x_1,x_2\}\subset  f^{-j}(G')\setminus (f^{-i}(G')\cup f^{-\ell }(G'))$$ or 
$$\gamma =\gamma _1\cup \gamma _3\cup \gamma _2$$ with 
$$\gamma _1\cup \gamma _2 \setminus \{ x_1,x_2\}\subset  f^{-j}(G')\setminus (f^{-i}(G')\cup f^{-\ell }(G')),$$
and $\gamma _3$ is disjoint from the edge of $f^{-i}(G')$ containing $x_1$ and the edge of $f^{-\ell }(G')$ containing $x_2$. The proof is then exactly similar to that of 2.

\Comment{\color{red}{Comment (51) right bracket added in to  $(f^{-i}(G'),f^{-\ell }(G'))$ near the start of the paragraph and right bracket added at the end of the last display above, together with a comma. Comment (52) $(\gamma ,x_1,x_2)$ corrected to $(\gamma ,\{ x_1,x_2\} )$.  ``defined'' corrected to ``refined''}}
\ep

Now to proceed further we have a sequence of simple lemmas  about the graphs $f^{-i}(G')$ for $0\le i<N$ which use the expanding property and bounded distortion of iterates of $f$. 

\begin{lemma}\label{1.13}
There is a constant $K_1$ and $\delta _0>0$, such that if $x$ and $y$ are two points on $f^{-i}(G')$, for $0\le i<N$, and if $d(x,y)\le \delta _0$, then there is an arc in $f^{-i}(G')$ between $x$ and $y$ of diameter $\le K_1d(x,y)$: that is, arcs of $f^{-i}(G')$ are quasi-arcs.\end{lemma} 

\noindent{\em{Proof}}. Choose $L>0$ such that any closed loop in $f^{-i}(G')$ has diameter $\ge 4L$ and such that  the $4L$-neighbourhood of any point of $f^{-i}(G')$ is disjoint from the postcritical set of $f$. Choose $\delta _0$ so that if $x$ and $y$ are points on $f^{-i}(G')$
 which are distance $\le \delta _0$ apart, then there is an arc of $f^{-i}(G')$ joining $x$ and $y$ of diameter $\le L$.  If $d(x,y)\le \delta _0$,
  then apply $f^{nN}$ for the largest $n$ such that $d(f^{kN}(x),f^{kN}(y))\le \delta _0$ for $k\le n$. Then use the bounded distortion $S$ on the ball of radius $L$ centred on $f^{nN}(x)$, where $S$ is the local inverse of $f^{nN}$ mapping $f^{nN}(x)$ and $f^{nN}(y)$ to $x$ and $y$.
  
  \Comment{\color{red}{ Related to Comment (54), the first sentence is now about the choice of $L$. The second sentence is now about the choice of $\delta _0$ given $L$, and is about the existence of an arc between $x$ and $y$ of diameter $\le L$. Comment (53)`` $f{-i}(G')$'' is corrected to ``$f^{-i}(G')$'' ``$<<\delta $'' is corrected to ``$\le \delta _0 $''.}}
  \ep
  
  \Comment{ \color{red} The following lemma is new, largely prompted by Comment (119).}
  
  \begin{lemma}\label{1.18} There exist $K_1>1$ and $\delta _0>0$ such that if $x_1\in f^{-i}(G')$ and  $x_2\in f^{-j}(G')$ for some $0<i, j<N$ and $d(x_1,x_2)\le \delta _0$ then the following holds. There is $x_3\in f^{-i}(G')\cap f^{-j}(G')$ such that $d(x_k,x_3)\le K_1d(x_1,x_2)$ for $k=1$, $2$, and if $i=j$ and $x_1$ and $x_2$ are in different edges of $f^{-i}(G')$ then $x_3$ is a vertex at which these two edges meet. \end{lemma}
  
  \noindent{\em{Proof}}
 
Choose $\varepsilon   _0>0$ and $K_0>1$ such that local inverses $S$ of $f^n$ on balls $B$ of radius $2K_0\varepsilon  _0$ round points of $f^{-i}(G')$, for $0\le i<N$, are univalent, of distortion bounded by $K_0$ and contract distance for all $n\ge N$. So if $x$, $y\in B$,
 $$K_0^{-1}\frac{|S'(x)|}{|S'(y)|}\le K_0.$$
Suppose also that  the spherical derivative $|(f^N)'|\le K_0$ on such balls $B$. First suppose $i\ne j$ Then there is $\eta _0>0$ such that if $y_1\in f^{-i}(G')$ and $y_2\in f^{-j}(G')$  for $0\le i,j<N$ and $d(y_1,y_2)<\eta _0$ then there is $y_3\in f^{-i}(G')\cap f^{-j}(G')$ such that $d(y_1,y_3)<\varepsilon  _0$ and $d(y_2,y_3)<\varepsilon _0$. For suppose not. Then there are points $y_{1,n}\in f^{-i}(G')$ and $y_{2,n}\in f^{-j}(G')$ such that $d(y_{1,n},y_{2,n})<2^{-n}$ and $d(y_{1,n},y_3)\ge \varepsilon _0$ for all $y_3\in f^{-i}(G')\cap f^{-j}(G')$. But choosing a subsequence we can assume that $\lim _{n\to \infty }y_{1,n}=y_3$. Then $\lim _{n\to \infty }y_{2,n}=y_3$ also, and so $y_3\in f^{-i}(G')\cap f^{-j}(G')$, and for all sufficiently large $n$, $d(y_{k,n},y_3)<\varepsilon _0$ for $k=1$, $2$, a contradiction. 

Similarly, given $\varepsilon _0$, if $\eta _0$ is sufficiently small and $y_1$, $y_2$  are in different edges of $f^{-i}(G')$, and $d(y_1,y_2)\le \eta _0$, then these two edges meet at a vertex $y_3$ with $d(y_k,y_3)\le \varepsilon _0$ for $k=1$, $2$. So there is $\nu _0\le \eta _0$ such that if $y_1$ and $y_2$ are in different edges of $f^{-i-N}(G')$ with $d(y_1,y_2)\le \nu _0$, then they are in edges with a common endpoint $y_3$, a vertex of $f^{-i-N}(G')$, with $d(y_k,y_3)\le \eta _0/K_0$.

Returning to the case of $i\ne j$, for suitably chosen $\eta _0$, with $d(y_1,y_2)\le \eta _0$, we can assume that $y_3\in f^{-i}(G')\cap f^{-j}(G')$ is either a vertex of $f^{-i}(G')$ or a vertex of $f^{-j}(G')$, or any arc in $f^{-i}(G')$ of diameter $\le K_0\varepsilon  _0$ between $y_1$ and $y_3$  does not contain a vertex, and similarly for $y_2$ and $y_3$ in $f^{-j}(G')$. 

Now we consider the general results. Let $d(x_1,x_2)\le \eta _0/K_0$. Choose the least $n$ such that $\eta _0/K_0\le d(f^{nN}(x_1),f^{nN}(x_2))\le \eta _0$. Then write $y_k=f^{nN}(x_k)$ for $k=1$, $2$. Let $y_3$ be as in the first paragraph, so that $d(y_3,y_k)\le \varepsilon _0$.   Let $x_3=S(y_3)$, where $S$ is the local inverse of $f^{nN}$ with $S(y_k)=x_k$ for $k=1$, $2$. 
Let $S_m$ be the local inverse of $f^{mN}$ with $S_m(f^nN(x_1))=f^{(n-m)N}(x_1)$, so that $S=S_n$. 

First we consider the case $i=j$. We claim by induction that $S_m(y_3)$ is a vertex of $f^{-i}(G')$, with $S_m(y_1)$ and $S_m(y_2)$ on adjacent edges of $f^{-i}(G')$. Clearly this is true for $S_0$. Suppose it is true for $S_m$. If $S_{m+1}(y_3)$ is not a vertex of $f^{-i}(G')$ then it is a vertex of $f^{-i-N}(G')$, but still a point of $f^{-i}(G')$, since $S_{m+1}(y_1)$ and $S_{m+1}(y_2)\in f^{-i}(G')$. But that would imply that $S_{m+1}(y_1)$ and $S_{m+1}(y_2)$ are in the same edge of $f^{-i}(G')$, and then $x_1$ and $x_2$ are in the same edge of $f^{-i}(G')$, giving a contradiction. So $S_{m+1}(y_3)$ is a vertex of $f^{-i}(G')$, with $S_{m+1}(y_1)$ and $S_{m+1}(y_2)$ in adjacent edges. So by induction the same is true for $x_k=S_n(y_k)$ for $k=1$, $2$, $3$. 

To obtain the general result for $i\ne j$, we proceed similarly. We can now assume by choice of $\varepsilon _0$ that any two distinct vertices of $f^{-i-N}(G')$ and $f^{-j-N}(G')$ are distance $\ge 2K_0\varepsilon _0$ apart. We need to prove by induction on $m$ that $S_m(y_3)\in f^{-i}(G')\cap f^{-j}(G')$. But if this is not true for a least $m$, then without loss of generality $S_m(y_3)\notin f^{-i}(G')$. Then $S_m(y_3)\in f^{-i-N}(G')$, and  an arc between $S_m(y_3)$ and $S_m(y_1)$ in $f^{-i-N}(G')$ must contain a vertex of $f^{-i-N}(G')$ in the boundary between $f^{-i}(G')$ and $f^{-i-N}(G')$. But then $S_m(y_3)$ must be this vertex, and $S_m(y_3)$ is in $f^{-i}(G')$ after all. 

We therefore take $\delta _0=\eta _0/K_0$ and $K_1=K_0\varepsilon _0/\eta _0$. \ep

\begin{lemma}\label{1.14}
If $C_2$ is sufficiently large, if $x\in  G'\cap f^{-i}(G')$ and $\delta >0$ is given, then for any arc $\zeta \subset G'$ with endpoint at $x$ and of diameter $\ge C_2\delta $, there is an arc  $\zeta _1\subset \zeta $  of diameter $\ge \delta $, which is either in $G'\cap f^{-i}(G')$ or is disjoint from $f^{-i}(G')$, for each $0<i<N$. \end{lemma}

\noindent{\em{Proof}}.

 Given an arc $\zeta \subset G'$ with endpoint at $x$, $f^{nN}(\zeta )$ contains an  edge of $G'$ for all sufficiently large $n$. Suppose $n$ is minimal with this happening, and let $e$ be the edge contained. The number of components of $f^{-nN}(e)\cap \zeta $ is bounded in terms of the degree of $f^N$ on $G'$. There is an arc $\zeta '\subset e$  such that, for each $0<i<N$, $\zeta '$ is either disjoint from   $f^{-i}(G')$ or contained in $f^{-i}(G')$. To see this, if $e\cap f^{-i}(G')$ has empty interior for each $0<i<N$, then $e\cap \bigcup _{0<i<N}f^{-i}(G')$ has empty interior, and we simply choose 
  $$\zeta '\subset e\setminus \bigcup _{0<i<N}f^{-i}(G').$$
   Otherwise, choose a maximal set $A\subset \{ i:0<i<N\}$ such that $e\cap \bigcup _{i\in A}f^{-i}(G')$ has interior. Let $e'$ be a component of this interior and then choose 
   $$\zeta '\subset e'\setminus \bigcup _{0<i<N,i\notin A}f^{-i}(G').$$
   Thus, the choice of  $\zeta '$ only depends on $e$ and the diameter of $\zeta '$ depends only on $e$
    Let $S$ be the local inverse of $f^{nN}$ mapping $f^{nN}(x)$ to $x$. Any component of $f^{-nN}(\zeta ')$ is either in $f^{-i}(G')$ or disjoint from $f^{-i}(G')$, noting that a component of $f^{-N}(G'\cap f^{-i}(G'))$ can be in $G'\cap (f^{-i-N}(G'))\setminus f^{-i}(G')$. Then  we use the fact that $S$  has bounded distortion independent of $n$ to obtain $C_2$, which is bounded in terms of the choice of $\zeta '$ for the edge $e$, for each of the finitely many edges $e$. 
\ep 

\Comment{\color{red}{Comment (55) Full stop added at end of second sentence. 

Comment (56) $\zeta $ corrected to $\zeta '$. 

Comment (57). Explanation is expanded on why we can choose $\zeta '$. This is really not a very strong statement. It was probably misleading to talk about the constant $\delta _0$ because it suggests that $\zeta '$ can be chosen to have diameter greater than some previously given bound. That is not the case and is not needed. So the reference to the diameter of $\zeta '$ has been removed. 

Comment (58) Missing bracket added on $(f^{-i-N}(G'))$ }}

\begin{lemma}\label{1.15}
Let $K_1>1$ be sufficiently large and $\delta _0>0$ sufficiently small, and $0<i<N$. Let $\gamma _1$ and $\gamma _2$ be two arcs of $f^{-i}(G')$ of diameter $\le \delta _0$ with endpoints $x_1$ and $y_1$, $x_2$ and $y_2$ in $G'$,  but otherwise not intersecting $G'$ and with $\gamma _2$ inside the disc bounded by $\gamma _1$,  and $G'$, with $x_2$ separating $y_1$ and $y_2$ from $x_1$.  Then 
$$\mbox{Min}(d(x_1,x_2),d(y_1,y_2))\ge K_1^{-1}\mbox{diam}(\gamma _1).$$ \end{lemma}

\Comment{\color{red}{$K_2$ of version 5 is changed to $K_1$ here.  $0<i<N$ added in the first sentence.  A number of commas added, and broken up into more sentences.}}

\noindent{\em{Proof.}} Let $K_1$ and $\delta _0$ be as in \ref{1.13} We can assume $K_1>1$. Also,  we can assume that $\delta _0$ is small enough that
$$3\delta _0\le \mbox{min}\{ d(v_1,v_2)\mid v_1,\ v_2\mbox{ are distinct vertices of }f^{-i}(G')\} ,$$
where $d$ denotes spherical distance. If $x_i$ and $y_i$ are the endpoints of $\gamma_i$ and $d(x_1,x_2)<K_1^{-1}\mbox{diam}(\gamma _1)$ then there is an arc $\gamma '$ on $f^{-i}(G')$ of diameter $<\mbox{diam}(\gamma _1)$ joining $x_1$ and $x_2$. Clearly $\gamma '\setminus \{ x_1\} $ cannot contain $\gamma _1\setminus \{ x_1\} $, and hence is disjoint from $\gamma _1$, since the endpoint $x_1$ is also an endpoint of $\gamma _1$, and the other endpoint of $\gamma '$ is not in $\gamma _1$.  
Then $\gamma =\gamma _1\cup \gamma'\cup \gamma _2$ is an arc of diameter  $\ge \mbox{diam}(\gamma _1)$  and $<3\delta _0$ with endpoints $y_1$ and $y_2$, whether or not $\gamma '$ contains $\gamma _2$. By the restrictions on $\delta _0$, any other arc in $f^{-i}(G')$ joining $y_1$ and $y_2$ has diameter $\ge 3\delta _0$,  and hence has diameter greater than $\gamma $ .  So $d(y_1,y_2)\ge K_1^{-1}\mbox{diam}(\gamma _1)$
 as required. 
 
\Comment{\color{red}{Comment (59) $\gamma '$ replaced by $\gamma '\setminus \{ x_1\} $ and other adjustments made to reflect the fact that $\gamma '$ and $\gamma _1$ have $x_1$ as a common endpoint

Comment (60) $K_1^{-2}\ge \mbox{diam}(\gamma _1) $ corrected to $K_1^{-1}\mbox{diam}(\gamma _1) $. I think now that $K_1^{-2}$ can be replaced by $K_1^{-1}$. 

Comment (61). This restriction on $\delta _0$ (or a slight adjustment) has been added in at the start of the proof - thanks. }}
\ep

\begin{lemma}\label{1.16} There is $\lambda <1$ and $C$ such that  for any nested sequence $\gamma _n$ for $G'$ and $f^{-i}(G')$, with $0<i<N$,
$$\mbox{diam}(\gamma _n)\le C\lambda ^{n-m}\mbox{diam}(\gamma _m)$$ 
for all $n>m$.\end{lemma}

\Comment{\color{red}{Comment (62)$f^i$ corrected to $f^{-i}$}}

\noindent{\em{Proof}}.

 \Comment{\color{red}{ Comment (63) ``$\gamma _n$.'' replaced by ``$\gamma _n$,'' .  ``with $0<i<N$,'' added in. After this the proof of this lemma has been rewritten, principally to deal with the problem identified in Comment (65). But first more detail is given on why it is enough to simply show that $\mbox{diam}(\gamma _{m+\ell })<\frac{1}{2}\mbox{diam}(\gamma _{m})$ for some $\ell \le k$. 
 
 I think Comment (64) has been dealt with in the course of rewriting prompted by Comment (66) - though this was one of the problems I identified before I got the report. The proof now uses an area argument There were two different uses of $C$ in the previous version of the proof.}}

 It suffices to show that  there exists $k$ such that $\mbox{diam}(\gamma _{m+\ell })<\frac{1}{2}\mbox{diam}(\gamma _{m})$ for some $\ell \le k$, for any nested sequence $\gamma _n$, and any $m$. 
  For then there are $\gamma _{m_j}$ with $m_0=m$ and $0<m_{j+1}-m_j\le k$ and 
 $$\mbox{diam}(\gamma _{m_{j+1}})< \frac{1}{2}\mbox{diam}(\gamma _{m_j})$$
  for all $j\ge 0$. Then if $\eta _j$ is the arc on $G'$ with the same endpoints as $\gamma _{m_j}$, we have 
 $$\mbox{diam}(\eta _j)\le K_1\mbox{diam}(\gamma _{m_j})$$
and
$$\mbox{diam}(\gamma _\ell )\le K_1^2\mbox{diam}(\gamma _{m_j}),\ \ m_j\le \ell ,$$
which gives the full result. 

Let $\varepsilon _0>0$ be such that the distance between vertices of $f^{-i}(G')$ is $\ge 4K_1^2\varepsilon _0$. We can  assume that  the diameter of $\gamma _n$ is $\le \varepsilon _0$ for all $n$. Let $\eta $ be the arc of $G'$ of diameter $\le K_1\mbox{diam}(\gamma _m)\le K_1\varepsilon _0$ which has the same endpoints as $\gamma _m$.  Then, for each $n$, $\gamma _n$ and $\gamma _m$ lie in an arc of $f^{-i}(G')$ of diameter $\le K_1^2\varepsilon _0$ with both endpoints in $\eta$. By breaking into two subsequences if necessary,  but both containing $\gamma _m$, we can assume that all the $\gamma _n$ are in an arc $\zeta $ of $f^{-i}(G')$ of diameter $\le 2K_1^2\varepsilon _0$ which contains at most one vertex of $f^{-i}(G')$.

So suppose that $\mbox{diam}(\gamma _{m+\ell})\ge \frac{1}{2}\mbox{diam}(\gamma _{m})$ for $\ell \le k$.   It is convenient to use the Euclidean metric rather than spherical metric at this point, assuming as we may do that the constants $K_1$ and $\delta _0$ are respectively large enough and small enough to work for both metrics. So diameter now refers to Euclidean diameter.  Divide each $\gamma _{m+\ell }$ into two arcs $\gamma _{m+\ell,1}$ and $\gamma _{m+\ell,2}$ which are disjoint apart from having a common endpoint, and each having diameter $\ge \frac{1}{4}\mbox{diam}(\gamma _{m})$. Thus, each of $\gamma _{m+\ell,1}$ and $\gamma _{m+\ell,2}$ has one endpoint in $\eta $. We call these endpoints $x_{m+\ell,1}$ and $x_{m+\ell,2}$ respectively. We choose numbering so that, if $\ell _1<\ell _2$, then $x_{m+\ell_2,1}$ separates $x_{m+\ell_1,1}$ from $x_{m+\ell_1,2}$ and $x_{m+\ell_2,2}$ in $\eta $. 
 Consider $K_1^{-1}\mbox{diam}(\gamma _m)/2^4$-neighbourhoods $N_{m+\ell ,j}$ of each $\gamma _{m+\ell,j }$ for $1\le \ell \le k$. These sets all lie in the set of diameter $\le 2K_1\mbox{diam}(\gamma _m)$ bounded by $\gamma _m$ and $\eta $, which in turn is contained in a Euclidean square with side length $\le 2K_1\mbox{diam}(\gamma _m)$ and hence the Euclidean area of the set bounded by $\gamma _m$ and $\eta $ is  $\le 4K_1^2(\mbox{diam}(\gamma _m))^2$. Meanwhile, the Euclidean area of $N_{m+\ell ,j}$ satisfies
 $$\mbox{area}(N_{m+\ell,j})\ge \frac{1}{2^2}\mbox{diam}(\gamma _{m})\cdot \frac{1}{2^4}K_1^{-1}\mbox{diam}(\gamma _m)$$ 
 $$=\frac{K_1^{-1}}{2^6}(\mbox{diam}(\gamma _m))^2.$$
 If $k>2^9K_1^3$, then there are distinct integers $\ell _j$, for $j=1$, $2$, $3$, with $0< \ell _i\le k$ with $N_{m+\ell _j,1}\cap N_{m+\ell _1,1}\ne \emptyset $ for $j=2$, $3$. This means that there are points $x_2$ and $x_3$ on $\gamma _{m+\ell _1,1}$ which are distance  $\le K_1^{-1}\mbox{diam}(\gamma _m)/8$ from points $y_2$ and $y_3$ on $\gamma _{m+\ell _2,1} $ and $\gamma _{m+\ell _3,1}$. If $y_2$ and $y_3$ are not separated in $\zeta $ from $x_2$ and $x_3$, then one of $y_2$ and $y_3$ separates the other from $x_2$ and $x_3$. So then either the arc in $\zeta $ between $x_2$ and $y_2$ contains $\gamma _{m+\ell _3}$, or the arc between $x_3$ and $y_3$ in $\zeta $ contains $\gamma _{m+\ell _2}$. If $y_2$ and $y_3$ are separated in $\zeta $ by $x_2$ and $x_3$, then either the arc in $\zeta $ between $x_2$ and $y_2$ contains $\gamma _{m+\ell _1,2}$ or the arc in $\zeta $ between $x_2$ and $y_2$  contains $\gamma _{m+\ell _1,2}$.  In all cases, in $\zeta $ between $x_2$ and $y_2$ or between $x_3$ and $y_3$  has diameter $\ge \frac{1}{4}\mbox{diam}(\gamma _m)$. This gives the required contradiction. 
\ep

\Comment{\color{red}{ In the first sentence of Lemma \ref{1.8}, more precise conditions are put on $\varepsilon _0$. The sentence about spherical metric has been removed, as that is now in the introduction.  I think Comment (68), about precision on the quantifier $N_1$, refers to the fact that the definition of nested sequence depends on a constant $\varepsilon _0$. This has now been incorporated into the second sentence of \ref{1.8}.}}
   
\begin{lemma}\label{1.8} Let $\varepsilon _0>0$ be given. Let the definition of nested sequences in \ref{1.6} be relative to this $\varepsilon _0$.  Then there exists $N_1$ and a finite  collection ${\mathcal{B}}$ of disjoint  closed contractible sets with locally connected boundaries, such that the following hold.
\begin{enumerate}
\item $\mbox{diam}(B)<\varepsilon _0$ for each $B\in{\mathcal{B}}$.
\item If $B_1\in {\mathcal{B}}$ and $B_2$ is a component of $f^{-1}(B_1)$, then either $B_2\cap B=\emptyset $ for all $B\in {\mathcal{B}}$ or $B_2\subset B_3$ for some $B_3\in{\mathcal{B}}$. 

\Comment{\color{red}{ Comment (67) $B_2\in B_3$ corrected to $B_2\subset B_3$. $f^{-1}$ replaces $f^{-n}$ above. }}

\item Let $\gamma _n$ be any nested sequence of arcs for $f^{-i}(G')$ and $f^{-j}(G')$ for any $0\le i,  j<N$ and $i\ne j$, or for $(f^{-i}(G'),f^{-\ell }(G'))$ and $f^{-j}(G')$ for distinct $i$, $\ell $, $j$ with $0\le i,j,\ell <N$. Then there is $B\in {\mathcal {B}}$ and $m$ and a component $B'$ of $f^{-m}(B)$   such that $\gamma _n\subset B'$ for all $n\ge N_1$. 
\item If $B_1\in{\mathcal{B}}$ then $B_1$ contains a component of $f^{-1}(B_0)$ for at least one $B_0\in {\mathcal{B}}$.

\Comment{\color{red} {This last property has added in -- although by \ref{1.7} it follows from Property 3.}}
\end{enumerate}
 \end{lemma}
 
 \noindent{\em{Proof.}}  Let $0\le i\le N$. By \ref{1.7} and \ref{1.16}, given $\varepsilon _1$, there is a finite set of nested sequences for $G'$ and $f^{-i}(G')$, say $\gamma _{n,j,i}$ ($n\ge 1$) for $1\le j\le r(i)$, where  $\gamma _{n,j,i}$ has the endpoints as the arc $\eta _{n,i,j}\subset G'$ of diameter $\le \varepsilon _1$,  and there is $N_1$ depending on $\varepsilon _0$  and $\varepsilon _1$ such that the following holds. Let   $D_{j,i}$, be the closed topological disc bounded by $\gamma _{1,j,i}\cup \eta _{1,j,i}$.
Let  $\gamma _n$  be any nested sequence for $G'$ and $f^{-i}(G')$. Let $\eta _n$ be the arc of $G'$ of diameter $\le K_1\varepsilon _0$ with the same endpoints at $\gamma _n$.   Let $D_n$ be the topological disc bounded by $\eta _n\cup \gamma _n$. Then there is $j\le r(i)$ and $m\ge 0$ such that, for all $n\ge N_1$,  
$$D_n\subset f^{-mN}(D_{j,i}).$$

\Comment{\color{red}{Rewritten in response to Comment (69) and also to put a precise bound on the diameter of any nested sequence. The definition of $D_{j,i}$ has been changed to be bounded by $\gamma _{1,j,i}\cup \eta _{1,j,i}$ and similarly  for $D_{j,i,\ell }$ in the paragraph below.  In the process, Comment (70) has been dealt with. }}
 
  We have a similar finite  set of nested sequences for $(G',f^{-i}(G'))$ and $f^{-\ell }(G')$. We call these sequences $\gamma _{n,j,i,\ell }$ ($n\ge 1$) for $1\le j\le r(i,\ell )$. We write  $\eta _{n,j,i,\ell }$ for  the arc of diameter $\le \varepsilon _0$ which is the union of an arc in $G'$ and an arc in $f^{-i}(G')$ which has the same endpoints in $f^{-\ell }(G')$ as $\gamma _{n,j,i,\ell }$.  Let $D_{j,i,\ell }$ be the topological disc bounded by $\gamma _{1,j,i,\ell }\cup \eta _{1,j,i,\ell }$. Let $\gamma _n$ be any nested sequence for $(G',f^{-i}(G'))$ and $f^{-\ell }(G')$ such that that $\gamma _n$ shares endpoints with an arc $\eta _n$ of diameter $\le \varepsilon _0$ which is a union of an arc in $G'$ and an arc in $f^{-i}(G')$. Let $D_n$ be the topological disc bounded by $\gamma _n\cup \eta _n$. 
Then there is $j\le r(i,\ell )$ and $m\ge 0$ such that, for all $n\ge N_1$,  
$$D_n\subset f^{-mN}(D_{j,i,\ell }).$$
 
   Let $K_1$ and $\delta _0$ satisfy the conditions of  \ref{1.13} and \ref{1.15}. 
 Let ${\mathcal{S}}_n$ denote the set of univalent local inverses of $f^n$ with domains of diameter $\le 2K_1\delta _0$. We also assume that $\delta _0$ is small enough, and $K_1>1$ large enough, that  $|S'|\le K_1$ for any $S\in{\mathcal{S}}_n$ with domain intersecting $f^{-\ell }(G')$ for any $n$, $\ell \ge 0$, and $|S'|<1$ for $S\in {\mathcal{S}}_{nN}$, for any $n>0$.
%     Write
%  $${\mathcal{B}}_{0,i}=\{ S(D_{j,i}):j\le r(i),\ S\in\bigcup _{n\ge 0}{\mathcal{S}}_{nN},\ \partial S(D_{j,i})\subset G'\cup f^{-i}(G')\} ,$$
%  $${\mathcal{B}}_{0,i,\ell}=\{ S(D_{j,i,\ell}):1\le j\le r(i,\ell), S\in\bigcup _{n\ge 0}{\mathcal{S}}_{nN}, \ \partial S(D_{j,i})\subset G'\cup f^{-i}(G')\cup f^{-\ell }(G')\} ,$$
%  $${\mathcal{B}}_{n,i}=\{ S(D):D\in{\mathcal{B}}_{0,i},\  S\in {\mathcal{S}}_n \} ,$$
%  $${\mathcal{B}}_{n,i,\ell}=\{ S(D):D\in{\mathcal{B}}_{0,i,\ell},\  S\in {\mathcal{S}}_n \} ,$$
%   $${\mathcal{B}}_n=\bigcup _{0<i<N}{\mathcal{B}}_{n,i}\cup \bigcup _{0< i,\ell<N,i\ne \ell}{\mathcal{B}}_{n,i,\ell}.$$
%  We define
%$$\Omega _{n,i}=\bigcup {\mathcal{B}}_{n,i},$$
%$$\Omega _{n,i,\ell }=\bigcup {\mathcal{B}}_{n,i,\ell },$$
%$$\Omega _n=\bigcup {\mathcal{B}}_n,$$

% $$\Omega =\bigcup _{n=0}^\infty \Omega _n.$$

We define
$${\mathcal{B}}_0=\{ D_{i,j}:0<i<N,1\le j\le r(i)\} \cup \{ D_{i,\ell,j}:0<i,\ell<N,i\ne \ell,1\le j\le r(i,\ell)\} ,$$
and 
$$\Omega (0)=\bigcup _{n=0}^\infty \bigcup \{ S(D):S\in {\mathcal{S}}_{nN}, D\in {\mathcal{B}}_0,\ S(\partial D)\subset G^0\} ,$$
$$\Omega (n)=\bigcup _{i=0}^{n}f^{-i}(\Omega (0)),\ \ \  0\le n\le \infty ,$$
$$\Omega (k;N)=\bigcup_{i=0}^kf^{-iN}(\Omega (0)),$$
$$\Omega (r,k;N)=\bigcup _{i=0}^kf^{-iN}(\Omega (r)).$$
Thus, $\Omega (N-1,k-1;N)=\Omega (kN-1)$. 
 
 \Comment{\color{red}{ Notation has been changed. I have used the bracket notation e.g. $\Omega (0)$ to distinguish from the notation $\Omega _0$ in \ref{1.9}.The sets ${\mathcal{B}}_{ , }$ have been removed, apart from ${\mathcal{B}}_0$ which has been redefined.
 
 Comment (71): the notation  ${\mathcal{S}}_n$ has been introduced to correct and clarify the definition of $\Omega (0)$, which replaces $\Omega _0$
 
 Comment (72) $\Omega (n)$, which essentially replaces $\Omega _n$, is now defined differently, but is essentially what $\Omega _n$ was before (or should have been).
}}

    Now we claim that  there is an integer $k_1$, and a constant $C_1$ independent of $\varepsilon _1$ and $k_1$, such that, if $\varepsilon _1$ is sufficiently small, then all components of $N_{\varepsilon _1}(\Omega (k_1N-1))$ have diameter $\le C_1\varepsilon _1$, where $N_\varepsilon (X)$ denotes the $\varepsilon $-neighbourhood of $X$. We shall then show that the components of $\Omega(\infty ) $ which intersect $\Omega (k_1N-1)$ are contained in $N_{\varepsilon _1/2}(\Omega (k_1N-1))$. The closure of each such component then bounds a closed contractible set with locally connected boundary of diameter $\le C_1\varepsilon _1$. We then choose $\varepsilon _1$ so that $C_1\varepsilon _1<\varepsilon _0$.
   Our required set ${\mathcal{B}}$ is then the set of maximal unions $B$  of closures of components of $\Omega (\infty )$ that intersect $\bigcup {\mathcal{B}}_0$,   and  complementary components  of diameter $\le \varepsilon _0$ bounded by them.  Properties 1 and 2 then hold by construction. Thus, each such $B$ contains at least one of the discs $D_{j,i}$ or $D_{j,i,\ell}$, for $j\le r(i)$  for some $0< i<N$, or $j\le r(i,\ell )$ for $0<i,\ell <N$, $i\ne \ell $. Thus, the number of sets in ${\mathcal{B}}$ is finite. Properties 3 and 4 hold by \ref{1.7}.  Note that although  sets in  ${\mathcal{B}}_{0}$ are topological discs, they are not necessarily disjoint. So the sets in $\Omega (0)$ need not be topological discs, the more so for the sets in $\Omega (n)$.
   
   \Comment{\color{red}{ Reference to $N_1$ been removed in the paragraph above. Instead it is emphasised that $C_1$ is independent of $\varepsilon _1$. The definition of ${\mathcal{B}}$ has been rewritten,  to make it clear that the set ${\mathcal{B}}$ is finite, the sets in ${\mathcal{B}}$ are disjoint and of diameter $<\varepsilon _0$. This is in response to questions in Comment (73). In response to Comment (74), an explanation has been added to explain why the first choice of components might not be topological discs. }}
      
  Note that the results of \ref{1.13} and \ref{1.5} work for $f^{-nN-i}(G')$ provided that $K_1\delta _0$ is replaced by half the minimal distance between vertices of $f^{-nN-i}(G')$.  Let $C_2$ be as in \ref{1.14}. Also, using bounded distortion of local inverses of $f^{nN}$ on $G'$, we can choose $C_2$ to work with $G'$ replaced by $f^{-nN}(G')$ in \ref{1.14}  for any $n\ge N$, that is, any arc $\zeta \subset f^{-nN}(G')$ of diameter $\ge C_2\delta $, which is contained in the union of two edges of $f^{-nN}(G')$, contains an arc $\zeta _1$ of diameter $\ge \delta $, such that for each $0<i<N$, $\zeta _1$ is either contained in $f^{-i-Nn}(G')$ or disjoint from $f^{-i-Nn}(G')$.
We also assume that $C_2\ge 6$. Note that $C_2$ is independent of $\varepsilon _1$. This is important, because at different stages of the proof we will want to make further restrictions on $\varepsilon _1$. For the moment, we assume that $3C_2K_1^N\varepsilon _1<\delta _0$.

\Comment{\color{red}{Most of the paragraph above has been added in.  Some lower bounds on $C_2$ have been moved to this paragraph or made more precise}}

Define $\varepsilon _k=K_1^{k-1}\varepsilon _1$ for $1\le k\le N$. Then for any set $X$ of diameter $\le \delta _0$ intersecting $f^{-n}(G')$ for any $n\ge 0$, we have $N_{\varepsilon _{k-1}}(f^{-1}(X))\subset f^{-1}(N_{\varepsilon _k}(X))$ for $2\le k\le N$. We claim that  components of $N_{\varepsilon _N}(\Omega (0))$ have diameter $\le 3C_2K_1^{N-1}\varepsilon _1$. For if $D_1$ and $D_2$ are intersecting sets  of ${\mathcal{B}}_0$, then $D_1\cup D_2$ has diameter $\le 2\varepsilon _1$, and  $D_1\cap G'$ and $D_2\cap G'$ lie in an arc of $G'$ of diameter $\le 2K_1\varepsilon _1$. But any arc of $G'$ starting from $D_1\cap G'$ of diameter $\ge (5/2)C_2K_1^{N-1}\varepsilon _1$ contains an arc of diameter $\ge (5/2)K_1^{N-1}\varepsilon _1$ which is disjoint from $\Omega (0)$. So the set of points in $G'\cap B$  for a component $B$ of $N_{\varepsilon _N}(\Omega (0))$ is contained in a union of at most two arcs of $G'$, containing at most one vertex, of diameter $\le (5/2)C_2K_1^{N-1}\varepsilon _1$. So
  $$\mbox{diam}(B)\le (5/2)C_2K_1^{N-1}\varepsilon _1+\varepsilon _1+2K_1^{N-1}\varepsilon _1\le 3C_2K_1^{N-1}\varepsilon _1.$$

 \Comment{\color{red}{Comment(75): $\varepsilon $ changed to $\varepsilon _1$.}}

Now let an integer $k_1>1$ be given. Write $C_3=(3C_2)^NK_1^{3N-3}$.  Let  $\varepsilon _1$ be  small enough given $k_1$ that $C_3\varepsilon _1<\delta _0$, and that the distance between distinct vertices of $f^{-n}(G')$ and $f^{-m}(G')$   for $0\le m,\ n\le k_1N+N$ is $\ge 2C_3\varepsilon _1$, as is the minimum distance between any two disjoint   edges of $f^{-n}(G')$ for $0\le n\le k_1N+N$.  For $D\in {\mathcal{B}}_0$, all components of $f^{-n}(D)$ have diameter $\le K_1\varepsilon _1$. Then in the same way as for $\Omega (0)$, all components of $N_{\varepsilon _N}(\Omega (k_1;N))$ have diameter $\le 3C_2K_1^{N-1}\varepsilon _1$.

  \Comment{\color{red}{ 
  The paragraph on vertices of $f^{-n}(G')$in the boundary of $f^{-n}(G')\setminus f^{N-n}(G')$ has been removed. The bound on the diameter of components of $N_{\varepsilon _1}(\Omega (k_1;N))$ is now proved more directly in the paragraph above.}}
   
 \Comment{\color{red}{Comment (77) $\Omega _{kN}$ no longer appears below, as all the names of the $\Omega $ sets have changed. The following two paragraphs correspond roughly  to the first full paragraph on page 15 of version 5. The  initial bound proved is on the diameter of the component of the union of $\varepsilon _1$ neighbourhoods. In version 5 this was glossed over and some details were missed. The general inductive method is the same but numbers have changed. $C_{1,0}$ is no longer used. Comment (76): Yes $C_{1,\varepsilon _1}$ concerning the constant $C_{1,0}$ which is no longer used. A new constant $C_3$ has been introduced. The was one instance of $C_3$ in the corresponding paragraph of version 5 but that should have been $C_2$.}}

 Now we will apply \ref{1.14} to prove that the diameter of any  component of  $N_{\varepsilon _{N-r}}(\Omega (r,k_1-1;N))$ is $\le K_1^{2r+N-1}(3C_2)^{r+1}\varepsilon _1$, by induction on $r<N$.  In this way we will show that the diameter of  any component of $N_{\varepsilon _1}(\Omega (k_1N-1))$ is $\le K_1^{3N-3}(3C_2)^N\varepsilon _1=C_3\varepsilon _1$. 

We have the bound for $r=0$.  Suppose inductively that  the bound holds for $r<N-1$. Let $B$ be a component of $N_{\varepsilon _{N-r-1}}(\Omega (r+1,k_1-1;N))$. Then $B$ is contained in a union $B'$ of components of $N_{\varepsilon _N}(\Omega (k_1-1;N))$ and components of $f^{-1}(N_{\varepsilon _{N-r}}(\Omega (r,k_1-1,N)))$. Now 
$$\mbox{diam}(B_1) \le K_1^{2r+N}(3C_{2})^{r+1}\varepsilon _1$$
for any component $B_1$ of $f^{-1}(N_{\varepsilon _{N-r}}(\Omega (r,k_1-1;N)))$.  To obtain the bound for $\mbox{diam}(B')$, we  consider  any two  components $B_0$, $B_2$ of $N_{\varepsilon _N}(\Omega (k_1-1;N))$ and  a component $B_1$ of  $f^{-1}(N_{\varepsilon _{N-r}}(\Omega (r,k_1-1;N)))$, such that $B_0\cap B_1\ne \emptyset $ and $B_1\cap B_2\ne \emptyset $. We have
$$\mbox{diam}(B_0\cup B_1\cup B_2)\le (K_1^{2r+N}(3C_{2})^{r+1}+6C_2K_1^{N-1})\varepsilon _1\le (3/2)K_1^{2r+N}(3C_2)^{r+1}\varepsilon _1,$$
using $6C_2\le 2C_2^2$. Then any arc of $G'$ with endpoints in $(B_0\cup B_2)\cap G'$ has diameter $\le 3K_1^{2r+1+N}(3C_2)^{r+1}\varepsilon _1=\delta $. Then we apply the separation property \ref{1.14} for  this $\delta $
    and $x\in B_2\cap G'$. 
  Within $C_2\delta $ of $x$ along any arc of $f^{N-k_1N}(G')$, there is an arc of diameter $\delta $ which does not intersect $f^{-k_1N-i}(G')$ for any $0<i<N$. So there is a union of at most two  arcs of $f^{N-k_1N}(G')$ (containing at most one vertex) containing all the points of $\Omega (k_1-1;N)\cap f^{N-k_1N}(G')$ in $B'$ and of diameter $\le C_2\delta $. We have 
  $\delta >  2K_1^{N-1}\varepsilon _1$.  So using the bounds on $\mbox{diam}(B_1)$ and $\mbox{diam}(B_0)$,  and $C_2\ge 4$, we have
  $$\mbox{diam}(B')\le ((3/2)C_2K_1^{2r+1+N}(3C_2)^{r+1} + 2K_1^{r+N}(3C_2)^{r+1})\varepsilon _1$$
  $$< K_1^{2r+1+N}(3C_2)^{r+2}\varepsilon _1.$$
  So the inductive step is completed.

 Now if $\varepsilon _1>0$ is sufficiently small, we can choose $k_1$ large enough that 
$$\lambda ^{k_1}C_3<\frac{1}{3},$$
where $\lambda <1$ is such that $f^{-nN}$ contracts by a factor $\lambda ^n$ on the $2K_1\delta _0$ neighbourhood of $G'$, for $n\ge k_1$.
 Now let $B$ be any component of $\Omega (k_1N-1)$. Let $B^n$ be defined inductively by $B^0=B$ and $B^{n+1}$ is the union  of $B^n$ and any components of 
 $f^{-(n+1)k_1N}(\Omega (k_1N-1))$
 that it intersects. These components have diameter $<3^{-(n+1)}\varepsilon _1$.
  Then $B^n\subset N_{\varepsilon _1/2}(B)$ for all $n$.  It follows that the component of $\Omega (\infty ) $ which contains $B$ has diameter $\le C_3\varepsilon _1+\varepsilon _1$. So we take 
  $C_1=C_3+1$. We then define ${\mathcal{B}}$ as described at the start of the proof.
 
\Comment{\color{red}{ 
 ``if $N_1$ sufficiently large'' replaced by ``if $\varepsilon _1$ is sufficiently small''. 
 
 ``$f^{-n}$ contracts by a factor $\lambda ^n$ in a sufficiently small neighbourhood of $G^0$ with respect to the spherical metric'' corrected to ``$f^{-nN}$ contracts by a factor $\lambda ^n$ on the $2K_1\delta _0$-neighbourhood of $G'$ for $n\ge k_1$'' 
 
 ``any union of components of $\bigcup _{k\le k_1N}\Omega _{kN}$'' changed to  to ``any union of components of  $\Omega (k_1N-1)$'' 
 
 ``any components of $f^{-nN}\left(\bigcup _{k\le k_1}\Omega _{kN}\right) $'' changed to ``any components of $f^{-(n+1)k_1N}(\Omega (k_1N-1))$''. 
 
 Some rewriting in the sentence after that,  leading to ``These components have diameter $<3^{-(n+1)}\varepsilon _1$'' where $3^{-n}$ has been corrected to $3^{-(n+1)}$}} \ep
 
 \Comment{\color{red}{It has become clear that the following lemma is needed, as well as the corollary which follows it. Or at least, without it, the proof of \ref{1.1} in \ref{1.14}  is unnecessarily complicated.}}
 
 Now we have the following information about intersections between $G^0$ and $\partial \Omega _0$.
 \begin{lemma}\label{1.17}$\partial ( G^0\setminus\Omega _0)$ is contained in the backward orbit of a finite set of  periodic points. \end{lemma}
\noindent{\em{Proof.}} Let ${\mathcal{B}}_0$ be as in the proof of \ref{1.8}. Write
$$\Omega _{0,0}=\bigcup{\mathcal{B}}_0.$$
First  we consider $\partial (G^0\setminus \Omega _{0,0})$. 
 Let $x\in f^{-\ell }(G')\cap \partial (G^0\setminus\Omega _{0,0})$. We can assume that $x$ is not a vertex of $f^{-m}(G')$ for any $0\le m<N$, since these are eventually periodic. Suppose $x\in \gamma \cup \eta \subset \Omega _{0,0}$, where $\gamma \subset f^{-k}(G')$ is an arc in a nested sequence for $f^{-i}(G')$ and $f^{-k}(G')$, or for $(f^{-i}(G'),f^{-j}(G'))$ and $f^{-k}(G')$, with $0\le i,j,k<N$, and $\eta $ is the arc with the same endpoints as $\gamma $ in $f^{-i}(G')$, or $\eta =\eta^1\cup \eta ^2$ with   $\eta ^1\subset f^{-i}(G')$ and  $\eta ^2\subset f^{-j}(G')$. If $x\in \gamma \cap \eta $, or $x=\eta ^1\cap \eta ^2$ in the case when $\eta =\eta ^1\cup \eta ^2$, then it follows from \ref{1.7} that $x$ is in the backward orbit of a finite set  $Y_1$ of periodic points. If $x\in \gamma \setminus \eta $, then it again follows from \ref{1.7} that $x$ is in the backward orbit of $Y_1$, because $x\in f^{-\ell }(G')\cap f^{-k}(G')$ is an endpoint of an arc in $f^{-\ell }(G')$ which is disjoint from $f^{-k}(G')$. Similarly the proof is finished by \ref{1.7} if $x\in \eta \setminus \gamma $. 

 Now each component $B$  of $\Omega _0$ is the  Hausdorff limit of some sequence $B_k$ of components  of $\bigcup _{n=0}^kf^{-n}(\Omega _{0,0})$, where $B_k\subset B_{k+1}$ and $B_0$ is a component of $\Omega _{0,0}$. It suffices to prove that $\partial (G'\setminus B)$ is contained in the backward orbit of a finite set of periodic points when $B$ is {\em{periodic}}, that is, $B$ contains at least one  component of $f^{-n}(B)$ for at least one $n>0$. For any component $C$ of $\Omega _{0,0}$ is of the form $f^{-n}(B')$ for some periodic $B'$ and some $n\ge 0$. So let $B_i$ be the periodic sets in $\Omega _0$ for $1\le i\le r$. There are finitely many maps 
$$T_j:\bigcup _{i=1}^rB_i\to \bigcup _{i=1}^rB_i,\ \ 1\le j\le s$$
such that $T_j$ is a local inverse of $f$ on each $B_i$ and $T_j(B_i)\subset B_{k_{j,i}}$ for some $1\le k_{j,i}\le r$. Then every component of $f^{-n}(B_i)$ which is contained in $B_j$ for some $j$ is of the form $T_{m_1}\circ \cdots \circ T_{m_n}$ for some $1\le m_k\le s$. Write ${\mathcal{T}}_n$ for the set of maps  of the form $T_{m_1}\circ \cdots \circ T_{m_n}$. So ${\mathcal{T}}_1=\{ T_j:1\le j\le s\} $. The set 
$$X=\bigcap_{n\ge 1}\bigcup \{ TB_i:T\in{\mathcal{T}}_n, 1\le i\le r\} $$
is closed, nonempty (it is a decreasing intersection of closed sets), and satisfies $f(X)=X$. Also, $B_i\setminus X\subset B_i'$,  where 
$B_i'=\bigcup _{m\ge 0}B_{i,m}'$ and $B_{i,0}'=B_i\cap \Omega _{0,0}$ and 
$$B_{i,m}'=\cup \bigcup _{j=1}^r\bigcup _{n=0}^m\{ T(B_{j,0}'):T\in {\mathcal{T}}_n,T(B_j)\subset B_i\} .$$
So  $\partial(G^0\setminus B_{i,m}')$ is contained in the backward orbit of $\partial (G^0\setminus\Omega _{0,0})$, that is, in the backward orbit of the set $Y_1$ mentioned above. But  $\partial (G^0\setminus B_i)\cap B_{i,m}'\subset\partial (G^0\setminus B_{i,m})$. So it remains to prove that $X\cap \partial (G^0\setminus B_i)$ is contained  in the backward orbit of a finite set of periodic points. But 
$$X\cap \partial (G^0\setminus B_i)\subset \partial (G^0\setminus X)=\bigcup _{i=0}^{N-1}\partial (f^{-i}(G')\setminus X).$$
It suffices to consider $\partial (G'\setminus X)$. Now $f^N:G'\to G'$ maps $X\cap G'$ onto $X\cap G'$ and is expanding. We now employ the same standard argument as in \ref{1.7}. Any arc $\gamma $ with interior in   $G'\setminus X$ and endpoint in $X$  is mapped by $f^{nN}$, for some $n>0$, to contain an arc  $\zeta $ of diameter $>\delta $ for some specified $\delta >0$ with interior in   $G'\setminus X$, with one endpoint of $\gamma $ mapped to one end point of $\zeta $. So the endpoint of $\gamma $ is eventually periodic, in the backward orbit of a finite set of periodic points. 

\ep

\begin{corollary}\label{1.19} There exists a set  $\Omega _0$ such that the set of components of $\Omega _0$ satisfies the properties 1 to 4 of \ref{1.8} satisfied by  ${\mathcal{B}}$, and in addition, $\partial (G^0\setminus \Omega )$ is a finite set of eventually periodic points.\end{corollary}
\noindent{\em{Proof}} Let $\Omega _0'=\bigcup {\mathcal{B}}$ be as in \ref{1.8}, but with all components of diameter $<\varepsilon _0/2$. Let $|S'|\le K_1$ for all local inverses of $S$ of $f^n$, for all $n>0$,  defined on balls of radius $\varepsilon _0$ centred on points of $G^0$. Let $\delta \le \eta \le \varepsilon _0/4$ be such that the minimum distance between  any two components $\Omega _0'$ is $\ge 4\eta $ and such that if $\gamma $ is an arc  with interior in  $G'\setminus \Omega _0'$ of diameter $\le \delta $ with endpoints in $\partial \Omega _0$ then $\gamma \cup \partial \Omega _0'$ bounds a disc $D(\gamma )$ of diameter $\le \eta $.   Let $\Omega _0$ be the union $\Omega _0'$ and of all components of $f^{-n}(D(\gamma ))$ which intersect $\partial \Omega _0'$ for all arcs $\gamma $ of diameter $\le \delta /K_1$ with interior in $G'\setminus \Omega _0'$ and endpoints in $\Omega _0'$. Then $\Omega _0$ has all the required properties.\ep

  \begin{lemma}\label{1.9} As usual let $G^0= \bigcup _{0\le i<N}f^{-i}(G')$.

\Comment{\color{red}{ Comment (78) extra ``$=G^0$'' deleted. Definition of $Y_0$ has changed in view of \ref{1.17}, \ref{1.19}. There are a number of changes in the statement of this lemma. $G^1$ no longer appears. }}

 Let  an integer $N_0$ and $\varepsilon _1>0$ be given.  Let $\Omega _0$ be as in \ref{1.19}. Write
 $$\Omega _n=\bigcup _{i=0}^nf^{-i}(\Omega _i.$$
 
Let $Y_0$ be the  union of the vertices of $f^{-i}(G')$, for $0\le i<N$, whose forward orbits do not intersect $\Omega _0$,  and of the forward orbits of $\partial(G^0\setminus  \Omega _0)$. (By \ref{1.17}, \ref{1.19}, $Y_0$ is a finite set.)

\Comment{\color{red}{ Comment (79) is now redundant in view of more substantial rewriting. The integer $p_1$ is introduced below. 
}}
 
 There exist integers $p_0$ $p_1$, and for $N_0$ sufficiently large, there exist   $Y$ with $f^{p_0-N_0}(Y_0)\subset Y\subset f^{-N_0}(Y_0)$, a set $\Omega $ which is the  union of components of $\Omega _{N_0}$ which are intersected by $Y$,   and a finite collection ${\mathcal{R}}(G^0)$ of   closed connected subsets of $G^0$,  such that the following hold.

\Comment{\color{red}{   $G^1$ no longer appears. Property 4 is a bit different from before. The bound on $P\cap Y$  is more precise than before and now $P\cap \partial \Omega $ is at most a single point --- and in $Y$. The statement of Property 6 is stronger. It had not actually  been properly proved to there is an additional section at the end of the proof }}

 \begin{enumerate}
 \item $\bigcup {\mathcal{R}}(G^0)=G^0\setminus \mbox{Int}(\Omega )$. 
  \item For each $P\in{\mathcal{R}}(G^0)$,  $G^0\setminus P$ is connected. 
 \item The interiors of sets in ${\mathcal {R}}(G^0)$, as subsets of $G^0$, are disjoint.
 \item Let  $\partial _{G^0}P$ denote the boundary of $P$ as a subset of $G^0$, for any $P\in{\mathcal{R}}(G^0)$. 
Then
\begin{equation}\label{1.9.1}\bigcup \{ \partial _{G_0}P:P\in{\mathcal{R}}(G^0)\} =Y,\end{equation}
  and 
$$ \#(P\cap Y)\le p_1,\ \ \ \#(P\cap \partial \Omega _0)\le 1.$$
 \item $P$ has diameter $<\varepsilon _1$, for each $P\in{\mathcal{R}}(G^0)$. 
 \item For each $P$, $P'\in {\mathcal{R}}(G^0)$, and local inverse $S$ of $f$ defined on $P'$, if  the interiors of $P$ and $S(P')\cap G^0$, as subsets of $G^0$, intersect, then 
 $$S(P')\cap G^0\subset P.$$
 Consequently, $Y\subset f^{-1}(Y)$.

  \end{enumerate} \end{lemma}
  
  \begin{remark}The $\varepsilon _1$ here is not the same as in the proof of \ref{1.8}\end{remark}
 
\noindent{\em{Proof.}}

   We assume, redefining $N$ if necessary, that $G'$ is not contained in $f^{-i}(G')$ for $0\le i<N$. 
     Let $\delta _0>0$ and $K_1>1$ be such that any  local inverse $S$ of $f^n$ on any  ball $B$ of radius $\delta _0$ round a point of $G^0$ is univalent and satisfies $|S'|\le K_1$ on $B$ and $|S'(x)|/|S'(y)|\le K_1$ for all $x$, $y\in B$. We also assume that $K_1$ satisfies the conclusions of \ref{1.13} and  \ref{1.15}.

% and that $10\varepsilon _1$ is $\le $ the distance between any two points on $\partial \Omega _0\cap G^{1,0}$.

%We can assume without loss of generality that $p_1$ is $\ge $ the number of intersections of $G^{1,0}$ with $\partial \Omega _0$. Then $p_1$ is also $\ge $ the number of intersections of $G^{1,i}$ with $\partial \Omega _i$, since $G^0$ is forward invariant under $f$. 

Let $\varepsilon _2>0$ be given, to be chosen sufficiently small later, given $\varepsilon _1$. Take any $N_0$ sufficiently large that so that $f^{N-N_0-i}(Y_0)\cap f^{-i}(G')$ is $\varepsilon _2 $-dense in $f^{-i}(G')$
 for each  $0\le i<N$.    Write 
 $$X_{i,i}=f^{-i}(G')\setminus f^{N-i-N_0}(Y_0),$$
and for $i\le j$, write 
$$X_{i,j}=\bigcup _{i\le \ell \le j}X_{\ell ,\ell }.$$

\Comment{\color{red}{ Comment (80):$k_1$ changed to $N_0$ several times above.  $Y$ changed to $Y_0$. Definition of $X_{i,i}$ changed, in that $\Omega _{N_0}$  no longer appears.
 
 Comment (81)Done
 
 Comment (82) Right bracket added at end of definitions of $X_{i,i}$ and $X_{i,j}$.
 
 Comment (83) First instance of $\varepsilon _2$ below is corrected from $\varepsilon $
 
 The following proof has been rewritten. The bound on the diameter of components of $X_{0,r}$ has been refined. Also there was a mistake in the proof. It was clearly incorrect to make the points $x_i$ be in the same component of $X_{1,r+1}$. There are no points $x_i$ in the new version of the proof, which is now closer to what has been rewritten in \ref{1.8}. $K_0$ has become $K_1$.}}

Now we will prove by induction on $r$ that each component $B$ of $X_{0,r}$
 \begin{equation}\label{1.9.2}\mbox{diam}(B)\le \le K_1^{2r}(3C_2)^rd(B),\end{equation}
  where $d(B)$ is the maximum diameter of a component of $X_{j,j}$ in $B$, and $K_1$ is as stated at the start, and $C_2$ is as in  \ref{1.14}. We have the result for $r=0$.  
So now suppose the result is true for $r$.  We need to prove it for $r+1$.  The technique is very similar to one used in \ref{1.8}. By the inductive hypothesis, every component $B'$  of 
$X_{1,r+1}$ has diameter $\le K_1^{2r+1}(3C_2)^rd(B')$, where $d(B')$ is the maximum diameter of a component of $X_{i,i}$ in $B'$, for $1\le i\le r+1$. So now, assuming that $C_2$, $K_1>1$, we need to consider each component $B$ of  $X_{0,r+1}$. Since $B$ is path-connected we only need to bound the diameter of  each path in $B$. We already have the bound for a path which lies in a component of  $X_{1,r+1}$. Any other path in $B$ must intersect $G'$. If $x$, $y\in G'\cap B'$ for a component $B'\subset B$ of $X_{1,r+1}$, then by \ref{1.13}, $d(x,y)\le K_1^{2r+2}(3C_2)^rd(B')$. If $x\in G'\cap B'$ and $y\in G'$ and the open arc in $G'$ between $x$ and $y$ is in $B$ and  does not intersect $X_{1,r+1}$, then this open arc is disjoint from $f^{-N_0}(Y_0)$, and so $d(x,y)\le d(B)$. So applying \ref{1.14} with $\delta =K_1^{2r+2}(3C_2)^rd(B)$ and adding in components of $X_{1,r+1}$ at the ends of the path, if necessary, we see that the diameter of any path in $B$ is $\le C_2K_1^{2r+2}(3C_2)^rd(B)+2K_1^{2r+1}(3C_2)^rd(B)$. Since $C_2\ge 1$ and $K_1\ge 1$ this completes the inductive step.

  Write $C_3=K_1^{2N}(3C_2)^{N}$. Let $B$ be a component of $G^0\setminus f^{-N_0}(Y_0)$. Since $f^{N-i-N_0}(Y_0)\subset f^{-N_0}(Y_0)$, we have
$\mbox{diam}(B)\le C_3\mbox{diam}(Q)$, for some  component $Q$ of $f^{-r}(G')\setminus f^{N-r-N_0}(Y_0)$ in $B$, for some $0\le r<N$. Now consider $f^{N_0+r-N-p_2}(B)$, chosen so that 
$$ \delta _0/(C_3K_1^2)\le \mbox{diam}( f^{N_0-p_2}(Q))\le  \delta _0/(C_3K_1).$$ 
The lower bound on the diameter of $f^{N_0+r-N-p_2}(Q)$, if $p_2>0$, gives an upper bound on $p_2$ in terms of $\delta _0$, $C_3$ and $K_1$. Also, $f^{N_0+r-N-p_2}$ maps $B$ univalently with image of diameter $\le \delta _0$.   It then follows that 
$$\#(f^{-N_0}(Y_0)\cap B)\le \#(f^{-p_2}(Y_0)).$$

\Comment{\color{red}{The paragraph above is essentially new, and makes it clear that $B$ intersects boundedly finitely many points of $f^{-N_0}(Y_0)$. The following paragraph is also new.}}

Now for any component $B$ of $X_{i,i}$, for $0\le i<N$, we have  $d(B)<\varepsilon _2$. So  all components of $G^0\setminus f^{-N_0}(Y_0)$ have diameter $C_3\varepsilon _2$. We can also choose $\varepsilon _2$ so that $C_3\varepsilon _2$ is less than the minimum distance between points of $Y_0\cap \partial \Omega _0 $.
 
 We define 
 ${\mathcal{R}}'(G^0)$ to be the collection of sets of the form $\overline{X}\cup B$, where $X$ is a component of $G^0\setminus (f^{-N_0}(Y_0)\cup \Omega _{N_0})$, and $B$ is the (possibly empty) union of components of $\Omega _{N_0}$ such that $B\cap G^0\subset X$ is separated by $X$ from $G^0\setminus X$.
 
 Then   
 ${\mathcal{R}}'(G^0)$ is finite, because  $f^{-N_0}(Y_0)$ is finite and the number of  disjoint arcs of $f^{-i}(G')$ which can meet at a point of $f^{-N_0}(Y_0)$ is $\le 3$ for  each $0\le i<N$. The sets in  ${\mathcal{R}}'(G^0)$ are connected.  Property 3 holds. If we define $Y$ by (\ref{1.9.1}), then  Property 4 holds for ${\mathcal{R}}'(G^0)$ by the proofs above, for a sutable $p_1$. Property 1 holds, if we define $\Omega $ to be the union of components of $\Omega _{N_0}$ intersected by $Y$.  Property 5 holds by the proofs above, for suitable $\varepsilon _2$, given $\varepsilon _1$. So now we need to modify the sets of ${\mathcal{R}}'(G^0)$  to obtain ${\mathcal{R}}(G^0)$ which satisfies still satisfies these properties, for suitable $Y$ and $\Omega $, and also satisfy Properties 2 and 6.
  
  \Comment{\color{red}{Comment(84) ``The components of $G^1\setminus (f^{-N_0}(Y_0)\cup \Omega _{k_1})$ satisfy all the required properties for the sets ${\mathcal{R}}(G^0)$.''  (which has some errors in any case) removed and replaced by explanations above of which properties are satisfied by ${\mathcal{R}}'(G^0)$.   in any case $k_1$ should be $N_0$ and as pointed out in Comment (84) ``sets'' should be ``sets in'' 
 
 Comment (85) This expression has been removed in the rewriting.
 
 Comment(86) $C_1$ corrected to $R_1$ (below)
 
 Comment (87) ${\mathcal{R}}''(G^0)$ corrected to ${\mathcal{R}}'(G^0)$ in the definition of $B(R)$.
 
 Comment (88): Full stop added below}}

We order the sets in  ${\mathcal{R}}'(G^0)$ by: $R_1<R_2$ if $R_2$ bounds a  disc  of diameter $\le \varepsilon _0$ containing $R_1$. Let ${\mathcal{R}}''(G^0)$ be the set of maximal sets in ${\mathcal{R}}'(G^0)$ in this ordering.
 For any $R\in{\mathcal{R}}''(G^0)$, let
 $$B(R)=R\cup \bigcup\{ R'\in{\mathcal{R}}'(G^0)\cup {\mathcal{B}}':R'<R\} .$$
  
Then we define
 $${\mathcal{R}}(G^0,N_0)=\{ B(R): R\in {\mathcal{R}}''(G^0)\} .$$
  
  \Comment{\color{red}{Comment (89): ${\mathcal{R}}''$ corrected to ${\mathcal{R}}''(G^0)$ Comment (90).  }}
 
Now ${\mathcal{R}}(G^0,N_0)$ satisfies Property 2. But Property 6 is still a problem. The required set ${\mathcal{R}}(G^0)$ will be a a collection of sets from  ${\mathcal{R}}(G^0,N_0-p)$ for different values of $p\ge 0$. Note that each set of ${\mathcal{R}}(G^0,N_0-p)$ is contains any set in the collection ${\mathcal{R}}(G^0,N_0-q)$ that it intersects, if $q\le p$. We now investigate when a set $R_1$ of ${\mathcal{R}}(G^0,N_0-p)$  contains any set $S(R)$ that it intersects,  for $R\in {\mathcal{R}}(G^0,N_0-q)$ and $S$ a local inverse of $f$. 
If $R\in {\mathcal{R}}(G^0,n)$, then $S(R)\subset R_1$ for some $R_1\in {\mathcal{R}}(G^0,n)$ unless $S(R)\cap f^{-N}(G')\ne \emptyset $, when it is possible that $S(R)\cap G^0$ is disconnected and intersects more than one set in ${\mathcal{R}}(G^0,n)$. However we claim that, if $n$ is sufficiently large, there is an integer $p$, bounded independently of $n$, such that for any $R\in {\mathcal{R}}(G^0,n)$, we have $S(R)\subset R_1$ for some $R_1\in {\mathcal{R}}(G^0,n-p)$. 

First we show  that, given a constant $C_4$, for  $p$ depending on $C_4$ but not on $n$, if $y\in f^{p-n}(Y_0)\cap G'\setminus \Omega _{n-p}$ and 
$$\delta =\mbox{Min}\{ d(z,y):z\in f^{-n}(Y_0),\ z\ne y\} ,$$ then any arc of $f^{-N}(G')$in $B(y,C_4\delta )\setminus \{ y\} $ can only intersect at most one component of $G^0\setminus (\{ y\} \cup \Omega _{N_0-p})\cap B(y,C_4\delta )$ with $y$ in its closure. This follows using bounded distortion of local inverses of $f^m$ on a neighbourhood of $G^0$, for all $m$. There is $\delta _0>0$ such that if $y_0\in Y_0$, then an arc of $f^{-N}(G')\setminus \{ y_0\} $ in $B(y_0,\delta _0)\setminus \{ y_0\} $  can intersect at most one component of  $G^0\setminus (\{ y_0\}\cup \Omega _0) \cap B(y_0,\delta _0)$ with $y_0$ in its closure. Then choose $p$ so that $f^{-p}(Y_0)$ is  $\delta _0/(K_1C_4)$-dense in $f^{-i}(G')$ for all $0\le i<N$. Assume without loss of generality that $n-p$ is divisible by $n$. Choose $y_0$ with $y=Ty_0$ for a local inverse $T$ of $f^{n-p}$.  Then $T(B(y_0,\delta _0)\supset B(y,C_4\delta )$ and the result follows. It then follows from (\ref{1.9.2}) that, for a suitable $C_4$, if $R\in {\mathcal{R}}(G^0,n)$ and $S$ is a local inverse of $f$ then, since  $S(R)\cap G^0$ is contained in a single component of $G^0\setminus f^{p_2-n}(Y_0)$, then $S(R)\subset R_1$ for $R_1\in {\mathcal{R}}(G^0,n-p_3)$, for $p_3$ bounded independently of $n$. 

Now for $x\in G^0$ we define 
$$m(x)=\#(\{ i:x\in f^{-i}(G'),\ 0\le i<N\} .$$
Then $1\le m(x)\le m(f(x))\le N$. Define
$$M(x)=\mbox{Max}\{ m(f^n(x))-m(x):n>0\} .$$
Then $0\le M(x)\le N-1$. The sets 
$X_k=\{ x\in G^0:M(x)\le k\} $ are nonempty, closed and forward invariant under $f$ for each $0\le k\le N-1$. They also satisfy $X_k\subset X_{k+1}$ and $X_{N-1}=G^0$. For $x\in X_0$, define $R(x)$ and $R'(x)$ to be the sets of ${\mathcal{R}}(G^0,N_0)$ and ${\mathcal{R}}(G^0,N_0-p_3)$  which contain $x$. Then  $S(P)\subset G^0$ if $S$ is a local inverse of $f$ and $P=R(x)$ or $R'(x)$. So let ${\mathcal{R}}(G^0,N_0, X_0)\subset {\mathcal{R}}(G^0,N_0)$ be defined by
$$X_0\subset Z_0=\bigcup {\mathcal{R}}(G^0,N_0, X_0)=\bigcup \{ R'(x):x\in X_0\} .$$ 
 So $Z_0$ is  a union of sets in ${\mathcal{R}}(G^0,N_0-p_3)$. Similarly, we define, inductively, for $x\in X_k\setminus Z_{k-1}$, sets $R(x)\in {\mathcal{R}}(G^0,N_0-kp_3)$ and $R'(x)\in  {\mathcal{R}}(G^0,N_0-(k+1)p_3)$  which contain $x$. 
 $$X_k\setminus Z_{k-1}\subset  Z_k=\bigcup {\mathcal{R}}(G^0,N_0-kp_3, X_k)=\bigcup \{ R'(x):x\in X_k\setminus Z_{k-1}\} .$$ 
 Then we define
$${\mathcal{R}}(G^0)=\bigcup _{k=0}^{N-1}{\mathcal{R}}(G^0,N_0-kp_3,X_k).$$
So Property 6 holds for ${\mathcal{R}}(G^0)$.  Write $p_0=Np_3$. Using Property 4 as a definition of  $Y$, we have $f^{p_0-N_0}(Y_0)\subset Y$. By the same arguments as  for $X_{0,N}$, we have $\mbox{diam}(P)\le C_3'd(P)$ for a suitable constant $C_3'$, and also have the bound on $\#(P\cap f^{-N_0}(Y))$, for $P\in{\mathcal{R}}(G^0)$.  So Property 4  holds for  ${\mathcal{R}}(G^0)$. Property 5 holds, if $\varepsilon _2$ is sufficiently small given $\varepsilon _1$.  We define $\Omega $ to be the union of components of $\Omega _{N_0}$ which intersect $Y$. Then Property 1 holds. Property 2 holds by construction, since each set ${\mathcal{R}}(G^0,n)$ is of the form ${\mathcal{R}}''$ with $n$ replacing $N_0$. Property  3 holds because ${\mathcal{R}}(G^0)$ is constructed to be a partition of $G^0\setminus \Omega $. 
 
\Comment{\color{red}{Property 5 changed to Property 6. Comment (90). $Y$ and $\Omega $ are now defined in the statement of  this lemma, \ref{1.9}. There was a significant gap, at this point in version 5, of the proof that $Y$ and ${\mathcal{R}}(G^0)$ satisfy the very important Property 6. This has now been filled. 
}}

\ep

\subsection{The iterative construction of $G$.}\label{1.10}

Let $\Omega $ be as in \ref{1.19}, so that $\Omega $ satisfies the conditions of \ref{1.19}, and of $\bigcup{\mathcal{B}}$ in \ref{1.8}. Let $\varepsilon _0$ be as in \ref{1.8}. Let  $Y$, $\varepsilon _1$, ${\mathcal{R}}(G^0)$ be as in \ref{1.9}. Note that $G^0$
 is connected, and in fact path-connected and locally connected. Let $K_1$ and $\delta _0$ satisfy the conclusions  of \ref{1.13}, \ref{1.15} and \ref{1.18}, and satisfy $|S'|\le K_1$ for any local inverse of $f^n$ defined on the $\delta _0$ neighbourhood of a point in $G^0$, and also $|S'(x)|/|S'(y)|\le K_1$ for any $x$ and $y$ in a ball of radius $\delta _0$ centred on $G^0$. For the moment we assume that $K_1\varepsilon _1<K_1\varepsilon _0<\delta _0$.  Later, we shall make assumptions on $\varepsilon _1$  being sufficiently small. 

\Comment{\color{red}{ Comment (91) $Y$ added in. $G^1$ is no longer used, because of results proved in \ref{1.17} and \ref{1.18}. Comma added after $\varepsilon _1$.

Comment ( 92)
For the moment only mild conditions have been placed on on $\varepsilon _1$. 

Comment (93) These sentences about nested sequences have been removed. 
Comment (94)  ``or'' changed to ``For'' below 
 
Then:  Comment (95) ``$f^{-1}(P)$'' changed to ``$f^{-m}(P)$'' }}

For each $m\ge 0$ we also write ${\mathcal{R}}_m(G^0)$ for the set of components of sets $f^{-m}(P)$ for $P\in {\mathcal{R}}(G^0)$.
 Thus, ${\mathcal{R}}(G^0)={\mathcal{R}}_0(G^0)$, and ${\mathcal{R}}_m(G^0)$ is a partition of $f^{-m}(G^0)$. 

\Comment{\color{red}{ The following is deleted.
Write
$$\Omega _n=\bigcup _{m\le n}f^{-m}(\Omega )$$
and
$$\Omega _\infty =\bigcup _{m\ge 0}f^{-m}(\Omega ).$$

In the next line, $\Omega _n$ changed to $f^{-n}(\Omega )$
  ``satisfying'' changed to ``satisfies''
  
  Comment (96): Hausdorff limit specified}}
 
 We will construct $\Gamma _n\subset f^{-n}(G^0)\cup f^{-n}(\Omega )$. If $\Omega  =\emptyset $, then $\Gamma _n$ is a finite connected graph. If $\Omega \ne \emptyset $, then the quotient $\mbox{quot}(\Gamma _n)$ of  $\Gamma _n$ obtained by collapsing components of $f^{-n}(\Omega )$ to points is a finite connected graph. We shall see 
  that  the Hausdorff limit $\lim _{n\to \infty }\Gamma _n=G$ is a finite graph satisfying
  $$G \subset f^{-1}(G).$$
  
  We define $\Gamma _0$ to be a  union of trees  $\Gamma _0(P)\subset P$ and of $\Omega $. For a component $B$ of $\Omega $, we define $\Gamma _0(B)=B$. So $\Gamma _0\subset G^0\cup \Omega $.

\Comment{\color{red}{ $\Gamma $ has been replaced by $G$. The definition of $\mbox{quot}(\Gamma _n)$ has now been made. Comment (97) `` It will be possible to extend $\Gamma $ to a finite
  invariant  graph by adding in arcs on components of $\partial \Omega $, but this will not give an invariant graph, so we will need to work to extend $\Gamma $ into $\Omega $ to produce the required graph $G$.'' the first instance  of ``invariant'' should not be there. But in any case, this paragraph is deleted and replaced by the sentences above. 

``Bounded by $P\cup \Omega $'' changed to ``bounded by $P$'' below. `` $R(P)$ is disjoint from $P')$ for $P'\notin P$'' changed to ``$R(P)$ is disjoint from $\mbox{int}(P')$ for $P'\ne P$''

In the next two paragraphs, references to points in $P\cap \partial \Omega $ and $R(P)\cap \partial \Omega $ have been removed since, by \ref{1.17}, all such points have now been proved to be eventually periodic and are in the redefined set $Y$.}}

  For $P\in {\mathcal{R}}(G^0)$, let  $R(P)$ be the union of $P$ and any topological discs of diameter $<\varepsilon _1$ which are bounded by $P$. By property 2 of \ref{1.9}, $R(P)$ is disjoint from $\mbox{int}(P')$ for $P'\ne P$, and of course $R(P)$ is contractible. 
  Let $D(P)$ be a closed topological disc with  $P\subset R(P)\subset D(P)$ such that the (finitely many) points of $\partial R(P)\cap Y$   are in $\partial D(P)$, but otherwise  $\partial D(P)$ is disjoint from $R(P)$. Also, $D(P)\cap D(P')\subset \partial D(P)\cap R(P)$ for any $P'\ne P$ with $P$, $P'\in{\mathcal{R}}(G^0)$. 

Now we define  $\Gamma _0(P)=\Gamma _0\cap P$ for each $P\in {\mathcal{R}}(G^0)$. 
For each pair of adjacent points $x$ and $y$ in $P\cap \partial D(P)$,  and component $C$ of $\partial D(P)\setminus P$ bounded by $x$ and $y$, there is a unique arc $\gamma _C\subset P$ between $x$ and $y$ such that $\gamma _C\cup C$ bounds a component of $D(P)\setminus P$. If $x$ and $y$ are the only two points in $P\cap \partial D(P)$ then there are two possibilities for $C$. Otherwise, there is only one. 
We choose $\Gamma _0(P)$ to be contained in $\bigcup _C\gamma _C$, removing some subarcs that have the same endpoints as some other subarcs. This is done by successively removing some arcs from pairs $(\gamma _{C_1},\gamma _{C_2})$, where $\gamma _{C_1}\cup \gamma _{C_2}$ is not an arc. If this is the case, then there are subarcs $\gamma _{C_1,C_2}$ and $\gamma _{C_2,C_1}$ of $\gamma _{C_1}$ and $\gamma _{C_2}$ respectively  with the same endpoints, and such that 
$$\gamma _{C_1}\cap \gamma _{C_2}=\gamma _{C_1,C_2}\cap \gamma _{C_2,C_1}.$$ These properties uniquely determine $\gamma _{C_1,C_2}$ and $\gamma _{C_2,C_1}$.  If $\gamma _{C_1,C_3}$ is another such subarc of $\gamma _{C_1}$, then $\gamma _{C_1,C_2}$ and $\gamma _{C_1,C_3}$ have at most a common endpoint --- which is also a common endpoint of $\gamma _{C_2,C_1}$ and $\gamma _{C_3,C_1}$. We obtain $\Gamma _0(P)$ from $\bigcup _C\gamma _C$ by removing one of $\gamma _{C_1,C_2}\setminus \gamma _{C_2,C_1}$ or $\gamma _{C_2,C_1}\setminus \gamma _{C_1,C_2}$  for each such pair $(C_1,C_2)$. When this has been done for all pairs $(C_1,C_2)$, the remaining set $\Gamma _0(P)\subset P$ is a  tree 
with finitely many endpoints, at all the points of $\partial _{G^0}P$.  It is therefore uniquely determined up to Whitehead equivalence, using isotopy fixing the endpoints

% Choose $\Gamma _0(P)=\Gamma _0\cap P$  to be a  finite tree such that any endpoints of $\Gamma _0(P)$ are in $\partial _{G_0}P\subset \partial D(P)$ and $Y\cap P\subset \Gamma _0(P)$. This is possible, because $P$ is itself a finite union of components of $G^0\setminus Y$, and the points of $P\cap Y$ are locally separating, from the definitions. This can be done in many ways: for example by choosing arcs in $G^0\cap P$ between any two points in $P\cap Y$ bounding a component of $P\setminus Y$, and then successively removing intersections to obtain a tree. To give an idea of how to do this, suppose we have arcs $\gamma _1$ and $\gamma _2:[0,1]\to P$ in $P$ which intersect. Let $t_1\le t_2$ be such that $\gamma _2(t_i) =\gamma _1(s_i)$ and if $\gamma _2(t)=\gamma _1(s)$ for some $s$ then $t_1\le t\le t_2$. It does not matter whether $s_1\le s_2$ or $s_1>s_2$. Then define an arc $\gamma _3$ with the same endpoints as $\gamma _2$ by $\gamma _3=\gamma _2$ outside $[t_1,t_2]$ and $\gamma _3([t_1,t_2])=\gamma _1([s_1,s_2])$.

  \Comment{\color{red}{  A number of changes in the paragraph above. ``For each $P\in {\mathcal{R}}(G^0)$'' added in.  
   
  $\Gamma _0$ is now defined to contain $\Omega $. This means that $\Gamma _0$ is connected, and homotopic to a finite graph, but not itself a finite graph, if $\Omega \ne \emptyset $. Comment (98): $\Gamma _0'$  no longer appears as  the description of $\Gamma _0(P)$ has changed. Definition of $\Gamma _0(P)=\Gamma_0\cap P$ made at this point. Consequently, $\Gamma _0\cap P$ is changed to $\Gamma _0(P)$ wherever possible below. 
  }}

 The tree $\Gamma _0(P)$,  might not be uniquely determined up to homeomorphism, but it is uniquely determined up to Whitehead equivalence, because it is a tree in $D(P)$ with a finite number of vertices, and with extreme points at specified points of $\partial D(P)$. Here, two graphs are Whitehead equivalent if one is obtained from the other by  finitely many {\em{Whitehead moves}}, followed by isotopy. Alternatively, isotopy can be performed first, or both before and after the Whitehead moves. A {\em{Whitehead move}} moves  apart two vertices with a common edge between them, or moves two such vertices together. We allow Whitehead moves which move together a vertex and an extreme point on the boundary. Any tree in a topological disc with $n$ extreme points, all on the boundary of the topological disc, is Whitehead equivalent to a tree with a single vertex. So all these trees are Whitehead equivalent. It follows that $\mbox{quot}(\Gamma _0)$ is completely determined up to Whitehead equivalence by the collection of sets ${\mathcal{R}}(G^0)$ and $\Omega $.
 
\Comment{\color{red}{Three sentences added to explain our use of Whitehead equivalence (Comment (100)).  Following paragraph rewritten.  $\Gamma _1$ now includes some components of $f^{-1}(\Omega )$, and $\mbox{quot}(\Gamma _1)$ is Whitehead equivalent to  $\mbox{quot}(\Gamma _0)$.  Comment (99): I must have corrected this but unfortunately I accidentally corrected version 5 at some point and don't know where this comma was. }}

Next we choose  $\Gamma _1\subset f^{-1}(\Gamma _0)$.  $\Gamma _1$ will be  a union of sets $\Gamma _1(P)=\Gamma _1\cap P$ for $P\in{\mathcal{R}}(G^0)$, and sets $\Gamma _1(B)$ for components $B$ of $\Omega $. $\Gamma _1(P)$ will be    finite tree, or a finite union of trees and any sets of $f^{-1}(\Omega )\setminus\Omega $ which intersect $P$. Then   $\mbox{quot}(\Gamma _1(P))$ will be a tree which is Whitehead equivalent to $\Gamma _0(P)$, using isotopy which fixes  endpoints in $Y\cap P$.
% Note that
% $$\Omega \setminus f^{-1}(\Omega )\subset \Omega _0.$$

Let $P\in{\mathcal{R}}(G^0)$. Since $\Gamma _0(P)\subset G^0$, we can cover $\Gamma _0(P)\setminus f^{-1}(\Omega )$ by sets $P_i$, $B_j$ for $1\le i\le r$ and $1\le j\le s$ for some $r$, $s$, with $P_i\in {\mathcal{R}}_1(G^0)$, where $P_i$ is a component of $f^{-1}(P_i')$ for some $P_i'\in{\mathcal{R}}(G^0)$ and $B_j$ is a component of $f^{-1}(\Omega )\setminus \Omega $. We take a minimal such covering so that $\mbox{int}_{G^0}( P)\cap P_i\ne\emptyset $ and $\mbox{int}_{G^0}( P)\cap B_j\ne\emptyset $ for $1\le i\le r$ and $1\le j\le s$ .Then we can choose a  tree 
$$\Gamma _1(P)\subset P\cap f^{-1}(\Gamma _0)\cap \left( \bigcup _{i=1}^rP_i\cup  \bigcup _{j=1}^sB_j\right) $$
with the same endpoints as $\Gamma _0(P)$, and the trees $\Gamma _0(P)$ and $\mbox{quot}(\Gamma _1(P))$ are Whitehead equivalent.   Note that, even in the case that $P\cap f^{-1}(\Omega )= \emptyset $, it is possible, and probably inevitable, that $\Gamma _1(P)$ intersects $f^{-1}(G^0)\setminus G^0$, and hence $\Gamma _1(P)$ might not be contained in $P$, although its endpoints are in $P$. 
   
  \Comment{\color{red}{Adjustments made from version 5 in the paragraph above,  to take account intersections between $\mbox{int}(P)$ and $f^{-1}(\Omega )$. Comments (101), 102, (103) redundant, in view of rewriting. }}

%In the following paragraph, $\Omega $ has been replaced by $\Omega _0$, because only components of $\Omega _0$ are not contained in $f^{-1}(\Omega )$. }}
 
 Now we define $\Gamma _1(B)\subset \Gamma _1\subset f^{-1}(\Gamma _0)\cup f^{-1}(\Omega _0)$ for a component $B$ of $\Omega _0$. We choose  $\Gamma _1(B)$   to have endpoints in common with $\Gamma _1(P)$ for $P\in{\mathcal{R}}(G^0)$ with  $P\cap \partial B\ne \emptyset $. So all   endpoints of $\Gamma _1(B)$ are in $\partial B$, and all points of $Y\cap \partial B$ are in $\Gamma _1(B)$. Interior points of edges, or vertices, are allowed. There is still a lot of choice here. Care is needed, in order to ensure that we get only finitely many vertices, and no  free vertices in the limit.  We have finitely many points on $\partial B$ which are endpoints of $\Gamma _0\setminus B\subset Y$. Each one is in $f^{-i}(G')$ for some $0\le i<N$. If there is only one point in $f^{-i}(G')\cap B$, then $f^{-i}(G')\cap \mbox{int}(B)=\emptyset $. Otherwise, if $\#(f^{-i}(G')\cap \partial B)\ge 2$, there have to be entry and exit points for $f^{-i}(G')$ into $B$.    Let $B_i$ be the components of  $f^{-1}( \Omega )\cap B$ with $B_i\cap f^{-1}(\Gamma _0)\ne \emptyset $, for $1\le i\le r$. Now we choose $\Gamma _1(B)\subset f^{-1}(\Gamma _0)$ to join points of $\Gamma _0\cap \partial B$ to $\partial B_i$ for some of the $B_i$, not necessarily all of them.  We join up  the points by finitely many arcs in $ f^{-1}(\Gamma _0)$, where these arcs do not intersect transversally, and so that no closed loops are created, and if $\mbox{End}(B_i)$ denotes the number of endpoints on $B_i$, then  
$$\mbox{End}(B)-2\ge \sum _i\mbox{Max}(\mbox{End}(B_i)-2,0).$$
 Then $\Gamma _1(B)$ is the union of these arcs and the components $C$ of $f^{-1}(\Omega _0)\cap B$ containing endpoints of the arcs. Then $\mbox{quot}(\Gamma _1(B))$ is a tree. Because of the construction of $\Omega $ in \ref{1.19}, we have  $\Gamma _1(B)\subset B$.

We then define $\Gamma _1$ to be the union of all sets $\Gamma _1(P)$, for $P\in {\mathcal{R}}(G^0)$, and $\Gamma _1(B)=B$ for components $B$ of $\Omega \setminus \Omega _0$. For $Q\in {\mathcal{R}}_1(G^0)$, we define $\Gamma _1(Q)=\Gamma _1\cap Q$.  For components $C$ of $f^{-1}(\Omega )$, we define $\Gamma _1(C)=C$.

 \Comment{\color{red}{  
 Comment (104) concerning the paragraph below: I think I must have added the missing bracket in version 5 by mistake and forgot to record the change. 
 
 The definition of $\Gamma _{n+1}(P)$ has been rewritten, making the inductive definition clearer. It is necessary to consider separately the cases $m=0$ and $m\ge 1$. Rewriting makes Comment (105) redundant.}}

Now suppose that $n\ge 1$, and that  $\Gamma _k$ has been defined for $k\le n$, and $\Gamma _k(P)$ has been defined for some of the  $P\in{\mathcal{R}}_m(G^0)$,  for each $m\le k$  such that, if $m\ge 1$,
$$\Gamma _{k}(P)\subset f^{-1}(\Gamma _{k-1}(P_1))\subset f^{1-k}(\Gamma _1(P_2))\subset f^{-k}(\Gamma _0(P_3)),$$
where $P$ is a component of $f^{-1}(P_1)$ and $P_1$ is a component of $f^{1-k}(P_2)$ and $P_2$ is a component of $f^{-1}(P_3)$, and $\Gamma _k(P)$ has the same endpoints as $\Gamma _{k-1}(P)$, for $P\in {\mathcal{R}}_m(G^0)$, for each $0\le m\le k-1$. 
First let $m\ge 1$. Then we define $\Gamma _{n+1}(P)$ to be the unique subset of  $f^{-1}(\Gamma _{n}(P_1))$  such that $\Gamma _{n+1}(P)\cap f^{-n}(Y)=\Gamma _n(P)\cap f^{-n}(Y)$,  and $\mbox{quot}(\Gamma _{n+1}(P))$ is Whitehead equivalent to $\mbox{quot}(\Gamma _n(P))$, with isotopy fixing $f^{-n}(Y)$.  Now we consider the case $m=0$ and $n\ge 1$. The definitions are exactly as before, except with $\Gamma _n(P_1)$ replaced by $\Gamma _n$.  Similarly, we define $\Gamma _{n+1}(B)$ for a component $B$ of $\Omega _m$, assuming inductively that $\Gamma _k(B)$ has been defined for $k\le n$ and any component $B$ of $\Omega _m$. Then we define $\Gamma _{n+1}$ to be the union of all $\Gamma _{n+1}(P)$ and $\Gamma _{n+1}(B)$, for $P\in{\mathcal{R}}_{n}(G^0)$ and components $B$ of $f^{-n}(\Omega )$ for which $\Gamma _{n+1}(P)$ and $\Gamma _{n+1}(B)$ are defined. Equivalently, we can use $P\in{\mathcal{R}}_{m}(G^0)$ and components $B$ of $f^{-m}(\Omega )$, for any $0\le m\le n$. Then we define $\Gamma _{n+1}(P)=\Gamma _{n+1}\cap P$ for any $P\in{\mathcal{R}}_{n+1}(G^0)$ for which $P\cap \Gamma _{n+1}\ne \emptyset $, and we define $\Gamma _{n+1}(B)=B$ for any component $B$ of $f^{-n-1}(\Omega )$ with $\Gamma _{n+1}\cap B\ne \emptyset $ --- in which case, of course, $B\subset \Gamma _{n+1}$. 

\Comment{\color{red}{
 In the paragraph below.  Comment (106). ``independent'' corrected. 
 
 The sentence ``Similarly, if $B$ is a component of $f^{1-m}(\Omega )\setminus f^{-m}(\Omega )$  then $\Gamma _n\cap B$ are all Whitehead equivalent for $n\ge m$.'' has been removed because intersections  of $\Gamma _n$ with 
 $f^{1-m}(\Omega )\setminus f^{-m}(\Omega )$ have now been incorporated into sets $\Gamma _n(P)$ for neighbouring sets $P\in {\mathcal{R}}_{m-1}(G^0)$.

 ``components $P$ of $f^{-n}({\mathcal{R}}(G^0))$'' replaced by ``sets $P\in{\mathcal{R}}_n(G^0)$'' -- so Comment (107) on missing bracket is redundant.
 
 ``$P$ is any component of $f^{-m}({\mathcal{R}}(G^0))$''  replaced by ``$P\in {\mathcal{R}}_m(G^0)$''
 
 Replaced $\Gamma _n\cap P$ by $\Gamma _n(P)$. Replaced $f^{-m}({\mathcal{R}}(G^0))$ by ${\mathcal{R}}_m(G^0)$
 
 Comment (108): ``for'' deleted
 
 Comment (109): corrected to $0<\lambda <1$
 
 Comment (110) is redundant since once again $f^{-m}({\mathcal{R}}(G^0))$ has been replaced by ${\mathcal{R}}_m(G^0)$
 
 `` or $P$ is a component of $f^{1-m}(\Omega )\setminus f^{-m}(\Omega )$'' removed again as intersections of $\Gamma _n$ with such sets have now been incorporated into $\Gamma _n(P)$ for $P\in {\mathcal{R}}_{m-1}(G^0)$ 
 
 Comment (111): comma added in.}}

For each $n$,  $\mbox{quot}(\Gamma _n)$ is   a finite connected graph with the same number of vertices: for $\Gamma _n$ itself, some vertices are replaced by components of $f^{-n}(\Omega _0)$.  For each $P\in {\mathcal{R}}_m(G^0)$, the trees $\mbox{quot}(\Gamma _n(P))$ are  finite trees in the same Whitehead equivalence class for all $n\ge m$, with isotopy fixing endpoints. Similarly, if $B$ is a component of $f^{-m}(\Omega )$, then  $\mbox{quot}(\Gamma _n(B))$ are  finite trees in the same Whitehead equivalence class for all $n>m$, with isotopy fixing endpoints. (If $m=n$ then $\Gamma _n(B)=B$.)
 So there is a bound, independent of $n$, on the number of  sets $P\in{\mathcal{R}}_n(G^0)$
 for which $\Gamma _n(P)$ is defined and  not homotopic to  an arc. Using the contraction of $f^{-n}$ in a neighbourhood of $G^0$, we see that  there are constants $C_0>0$ and $0<\lambda <1$ such that, if $d_H$ denotes Hausdorff distance, with respect to the spherical metric, and $P\in {\mathcal{R}}_m(G^0)$, then, recalling that sets in ${\mathcal{R}}(G^0)$ have diameter $<\varepsilon _1$,
\begin{equation}\label{1.10.1}d_H(\Gamma _n(P),\Gamma _{n+1}(P))\le C_0\varepsilon _1\lambda ^{n-m}\end{equation}
for all $n\ge m$. Also, if $P\in {\mathcal{R}}_n(G^0))$,

\Comment{\color{red}{last phrase added in, $\Gamma _n\cap P$ replaced by $\Gamma _n(P)$ again.

Below: ``wherever'' replaced by ``whenever'' ``components $P$ of'' replaced by ``sets $P\in $''

 }}

\begin{equation}\mbox{diam}(\Gamma _n(P))\le C_0\varepsilon _1\lambda ^n.\end{equation}
Similar inequalities hold if $P\in{\mathcal{R}}_m(G^0)$ is replaced by a component $B$ of $f^{-m}(\Omega )$ for $m\le n$, and $\varepsilon _1$ is replaced by $\varepsilon _0$. Now we claim that $\bigcup _{P\subset U}\Gamma _n(P)$ converges to an arc in $\overline{\mathbb C}$ -- which will be an arc in our graph $G$ --- whenever $U$ is a  union of sets $P\in {\mathcal{R}}_m(G^0)$ and components of $f^{-m}(\Omega )$, and the boundary of $U$ consists of two points of $f^{-m}(Y)$
   --- such that $\mbox{quot}(\Gamma _n(P))$ is  an arc for one, and hence all, $n\ge m$. It suffices to prove that for $P_1$ and $P_2\in{\mathcal{R}}_m(G^0)$ with $P_1$, $P_2\subset U$ and $\partial _{G_0}(P_1)\cap \partial _{G_0}(P_2)=\emptyset $ then $\lim _{n\to \infty }\Gamma _n(P_1)$ and $\lim _{n\to \infty }\Gamma _n(P_2)$
do not intersect. We only need to prove this when $d(x,y)\le 2C_0\varepsilon _1/(1-\lambda )$ for some   $x\in P_1$ and $y\in P_2$.  It suffices to prove it for $m=0$, since we then get the result for a general $m$ by applying local inverses of $f^m$. For if $P_1$ and $P_2$ are intersections with $G^0$ of components of $f^{-\ell  }(Q)$ for some $Q\in{\mathcal{R}}_{m-\ell }(G^0)$ and minimal $\ell $, then the components of $f^{-1}(Q)$ are bounded apart.  If $\Gamma _n(P_i)\subset P_i$ for all $n$ and for $i=1$ and $2$, then the proof is finished. But in general $\Gamma _n(P_i)$ is contained in a possibly larger set, which we need to analyse.
 
 \Comment{\color{red}{ Right bracket ended in after $f^{-m}(\Omega )$ above -- this might be comment (110) but I don't think so.  Comment (112) Obvious typo corrected.
 
 Below, the definitions of ${\mathcal{R}}(f^{-1}(G^0))$ and  of ${\mathcal{R}}(f^{-n-1}(G^0))$ have been rewritten, using $A_1(P)$ from the start. Comments (113) (114) now redundant. Sentence starting ``Note that'' has been added in. The definition of ${\mathcal{R}}(G^0)$ has been slightly changed to take account of the fact that   two disjoint open  arcs of $G^1$ might have a common endpoint in $\partial \Omega $. 
 
 Comment (115): brackets added in on 3rd, 4th, 5th, 17th lines of what was page 21. $f^{k-m}$ changed to $f^{-n}$ on what was the 5th line of page 21.}}

For $P\in {\mathcal{R}}(G^0)$, define $A_1(P)$ to be the union of  $P$ and of arcs of $f^{-N}(G')$ with endpoints in $P$. 
 Then we define  
$${\mathcal{R}}(f^{-1}(G^0))=\{ A_1(P):P\in {\mathcal{R}}(G^0)\} .$$

 Then each $A_1\in {\mathcal{R}}(f^{-1}(G^0))$ intersects $\mbox{int}_{G^0}(P)$  for a unique  $P\in {\mathcal{R}}(G^0)$. Then
  $$A_1(P)\setminus P\subset (f^{-1}(G^0)\setminus G^0)=(f^{-N}(G')\setminus G').$$
 So $A_1(P)=P$ if $A_1(P)\cap f^{-N}(G')\setminus G'=\emptyset $. 
 Similarly for $P\in {\mathcal{R}}(f^{-n}(G^0))$,  we define $A_1(P)$  to be the union  of  $P$ and of arcs in $f^{-n+1-N}(G')$ with endpoints in $P$.
   Then 
  $${\mathcal{R}}(f^{-n-1}(G^0))=\{ A_1(P):P\in {\mathcal{R}}(f^{-n}(G^0))\} .$$
  We also define $A_n(P)$ inductively by $A_{n+1}(P)=A_1(A_n(P))$.

\Comment{\color{red}{Property 5 replaced by Property 6 below. Comment (116): Hausdorff limits. To reflect earlier changed definition of $\Gamma _n(P)$, $A_n(P)$ below is replaced by $A_n(P)\cup f\Omega _n(P)$.}}

  Property 6 of \ref{1.9} still holds for ${\mathcal{R}}(f^{-n}(G^0))$, for all $n\ge 0$. If $P_1$ and $P_2\in {\mathcal{R}}(G^0)$ with $P_1\ne P_2$ then $A_n(P_1)$ and $A_n(P_2)$ intersect only in  common  boundary points of $P_1$ and $P_2$. Write $\Omega _n(P)$ for the union of components of $f^{-n}(\Omega )$ which intersect $A_n(P)$. 
 Then $\Gamma _n(P)\subset A_n(P)\cup \Omega _n(P)$ for all $P\in{\mathcal{R}}(G^0)$. Diameters of components of $f^{-n}(\Omega )$ tend to $0$ uniformly with $n$, and $\Gamma _n(P)\cup \Omega _n(P)$ is connected. So it suffices to show that the Hausdorff limits $\lim _{n\to \infty }A_n(P_1)$ and $\lim _{n\to \infty }A_n(P_2)$ are disjoint  if $P_1$ and $P_2$ are disjoint. 
 
Write $\delta =\mbox{diam}(P_1)$. We claim that, given $C_1>0$ there is a constant $C_4>0$ such that if
\begin{equation}\label{1.10.3}\mbox{Min}\{ d(x_1,x_2):x_1\in P_1,\ x_2\in P_2\} \le C_1\delta ,\end{equation}
 then 
\begin{equation}\label{1.10.2}C_4^{-1}\delta \le \mbox{diam}(P_2)\le C_4\delta .\end{equation}
We see this as follows. Note that $P_1$ and $P_2$ can be interchanged, so we can assume that $\mbox{diam}(P_2)\le \delta $ and $C_1\ge 2$, and $\mbox{diam}(P_1\cup P_2)\le 2C_1\delta $.  Let $Y(P)=P\cap f^{-N_0}(Y_0)$. The argument is similar to that in \ref{1.9}. We recall from \ref{1.9} that there is a constant $C_3$ such that $\delta \le C_3\mbox{diam}(Q)$ for a component $Q$ of $f^{-r}(G')\setminus f^{-N_0}(Y_0)$ with $Q\subset P_1$ and $0\le r<N$. Choose the least $p_2\ge 0$ such that $\mbox{diam}(f^{N_0-p_2}(Q))\le \delta _0/(2C_1C_3)$, so that if $p_2>0$, then $\mbox{diam}(f^{N_0-p_2}(Q))\ge \delta _0/(2K_1C_1C_3)$, and $p_2$ is bounded in terms of $C_1$, $K_1$ and $C_3$.  Then $\mbox{diam}(f^{N_0-p_2}(P_1\cup P_2))\le \delta _0$. So 
$$\mbox{diam}(f^{N_0-p_2}(P_2))\ge \mbox{diam}(f^{N_0-p_2}(P_2)\cap f^{-p_2}(Y_0))\ge C_4'\mbox{diam}(f^{N_0-p_2}(Q)),$$
where $C_4'$ is a constant bounded in terms of the minimum distance between points of $f^{-p_2}(Y_0)$. Then pulling back to $P_1\cup P_2$ under the local inverse of $f^{N_0-p_2}$, we obtain (\ref{1.10.2}) with $C_4=C_4'K_1$.

Similarly for any $P\in{\mathcal{R}}(G^0)$, there is a constant $C_6$ such that 
$$\mbox{diam}(A_1(P))\le K_1\mbox{diam}(P),$$
 because if $A_1(P)$ is any larger than $P$ then $A_1(P)$ is obtained from $P$ by adding in arcs of $f^{-N}(G')$ with endpoints in $P$. Arcs in $f^{-N}(G')$ are $K_1$-quasi-arcs, by \ref{1.13}. The diameter of any component of $f^{-1}(\Omega )$ in $P$ is bounded by the a constant times the distance between the boundary points in $f^{-1}(Y)$ and these points are themselves in arcs between points of $Y\cap P$, or in $Y\cap P$. 

\Comment{\color{red}{ This paragraph has been substantially rewritten.

``are'' -deleted before ``each contain''.

``contain $Y_0$'' replaced by ``intersect' $Y_0$'' . ``$f^{k_1-p_0}(Y_1)$'' replaced by``$f^{N_0-p_1}(P_2)$''

$k_1$ replaced by $N_0$ several times

Comment (118): punctuation added. Comment (117)  $d$ replaced by $\delta $.''}}
 
 So we now assume that (\ref{1.10.3}) and (\ref{1.10.2}) hold,  with $C_1=2C_0/(1-\lambda )$,and $P_1$ and $P_2$ are disjoint.
   Then there is a universal constant $C_5>0$, depending only on $G'$ and $C_4$ such that  \begin{equation}\label{1.14.1}\mbox{Min}\{ d(x_1,x_2):x_1\in P_1,\ \ x_2\in P_2\} \ge C_5^{-1} \delta .\end{equation}
 \Comment{\color{red}{Comment (119)The proof of this is expanded and now refers to\ref{1.18}, which is new.}}
  
For suppose that $d(x_1,x_2)<C_5^{-1}\delta $  and $x_k\in f^{-i_k}(G')$ for some $0\le i_k<N$. By \ref{1.18} there is $x_3\in f^{-i_1}(G')\cap f^{-i_2}(G')$ with $d(x_k,x_3)\le K_1C_5^{-1}\delta $.   
 Then either $x_3\in P_k$ or $x_3$ is separated in $f^{-i_k}(G')$ from $x_k$ by a point $y_k\in Y(P_k)$. In such a case we have $d(x_k,y_k)\le K_1^2C_3^{-1}\delta $. This is only possible for one of the $P_k$ because otherwise we would have  $d(y_1,y_2)<2K_1^2C_3^{-1}\delta $, which is impossible with $C_3$ large enough, because  it would mean $y_1=y_2$, contrary to our assumption. But it cannot be possible for just one $P_k$ either because then $P_1$ and $P_2$ have a common boundary point. 
This completes the proof of (\ref{1.14.1})

 Also, for each $n$ and $P$, any set $P'$ of ${\mathcal{R}}_n(G^0)$ which intersects $P$ has boundary in  $f^{-n}(Y)$ and with interior disjoint from this. So we have
 $$\mbox{diam}(P')\le C_0\lambda ^n\mbox{diam}(P).$$
  
\Comment{\color{red}{The following argument is rewritten, which certainly needed doing -- although, ironically, ``$f^{-j-\ell N-kN}(G')$'' was just a typo -- should have been ``$f^{-j-\ell N-kN}(G')\setminus f^{-j-kN}(G')$'', as it now is (Comment (120)). Comments (121) to (123) are now redundant because of rewriting.}}

 We now assume that $\varepsilon _1$ is  sufficiently small, that the following holds, given an integer $i_1$ and the constant $C_4$, which we assume is $\ge C_1$.  The integer $i_1$ will be specified later, but is independent of $N_0$, and hence of  $\varepsilon _1$.  If $0\le i,j<N$ with $i\ne j$ and  $ x\in f^{-i}(G')\cap f^{-j-kN}(G')$  for some  $0\le k\le (1+N)i_1$, then $x$ is distance $\ge 6C_4\varepsilon _1$ from $f^{-j-\ell N-kN}(G')\setminus f^{-j-kN}(G')$ for all but at most one  $\ell $ with   $k<\ell\le  k+(1+2N)i_1$, and if there is such an $\ell $ then, using \ref{1.17}, $x$ must be distance $\le 6K_1C_4\varepsilon _1$ from a vertex in the boundary between $f^{j-kN}(G')$ and $f^{-j-\ell N-kN}(G')\setminus f^{-j-kN}(G')$. Here, as assume, as we may do, using bounded distortion of local inverses of $f^n$, that  \ref{1.17} works with $f^{-i}(G')$ replaced by $f^{j-kN}(G')$  and $K_1\delta _0$ replaced by half the minimum distance between vertices of $f^{j-kN}(G')$.  So we also  assume that $\varepsilon _1$ is sufficiently small given $i_1$ that any two vertices for $f^{-j-mN}(G')$ are distance $12 K_1C_4\varepsilon _1$ apart or coincide, for $0\le j<N$ and $m\le (1+2N)i_1$, and any vertex of $f^{-k}(G')$ is either distance $\ge 12 K_1C_4\varepsilon _1$ from $f^{-\ell}(G')$ or contained in $f^{-\ell }(G')$ for $0\le k,\ell \le (2+2N)i_1$. 

 These conditions ensure that  there are at most $2N$  values of $n$ (depending on $P_1$) such that $n\le (2N+2)i_1$ and $A_{n+1}(P_j)\ne A_n(P_j)$ for at least one of $j=1$, $2$
 
   So now, there is $i_0$ (depending on $P_1$ and $P_2$) with $i_0\le (2N+1)i_1$  such that $A_n(P_j)=A_{i_0}(P_j)$ for $i_0\le n\le i_0+i_1$ and $j=1$, $2$. Then for any $x_j\in A_n(P_j)$ for $n\le i_0+i_1$ and $j=1$, $2$,
   $$d(x_1,x_2)\ge C_5^{-2N}\delta .$$ 
   This is the same argument as for $x_j\in P_j$ for $j=1$, $2$, together with induction replacing $P_j$ by $A_i(P_j)$ for each $i$ with $A_{i+1}(P_j)\ne A_i(P_j)$ for at least one $j$.  But for any $n$, for $P=P_1$ or $P_2$,
   $$d_H(A_n(P),A_{n+1}(P))\le C_0C_4\lambda ^n\delta $$
   So for any $x_j\in A_n(P_j)$ for $n\ge i_0+i_1$ and $j=1$, $2$,
    $$d(x_1,x_2)\ge C_5^{-2N}\delta -C_0C_4\delta \lambda ^{i_1+i_0}/(1-\lambda).$$
  So if $i_1$ is large enough given $C_0$, $\lambda $, $C_4$ and $C_5$, the Hausdorff limits $\lim _{n\to \infty }A_{n}(P_1)$ and $\lim _{n\to \infty }A_{n}(P_2)$ are disjoint, as required, and our required graph is the Hausdorff limit
   $$G=\lim _{n\to \infty }\Gamma _n.$$

For Conclusion 1 of \ref{1.1}, by (\ref{1.10.1}), $G$ is in an $O(\varepsilon _1)$ neighbourhood of $G^0$, where $\varepsilon_1$ can be taken as small as desired. For conclusion 2 of \ref{1.1}, $G$ has paths within $O(\varepsilon _1)$ of every path through ${\mathcal{R}}(G^0)$, and also passes through all components of $\Omega $, which, by  \ref{1.8}, have diameter $<\varepsilon _0$. For $0<i<N$, any arcs of $f^{-i}(G')\setminus G^1$ have diameter $\le K_1\varepsilon _0$, for $K_1$ as in \ref{1.13}, since the endpoints  of such an arc are in the same component of $\Omega $. So since $G'\subset G^0$, given $\varepsilon '$, we can choose $\varepsilon _0$ and $\varepsilon _1$ of \ref{1.8} and \ref{1.9} sufficiently small that $G$ has closed loops  within $\varepsilon '$ of  any closed loop of  $G^0$, and hence also of $G'$, which bounds a disc of diameter  $\ge \alpha _0$.

\Comment{\color{red}{Conclusion 1 of Theorem 1.1. has now been removed. So Conclusions 2 and 3 are now renumbered as Conclusions 1 and 2.  For the proof of Conclusion 2, details have been added in on the diameter of components of $f^{-i}(G')\setminus G^0$.}}
\ep

 \section{Boundary of existence of Markov partition}\label{2}
 
The main motivation for constructing Markov partitions as in Section \ref{1} is that Markov partitions with such properties exist on an open subset of a suitable parameter space. One can then use such partitions to analyse dynamical planes of maps in a subset of parameter space, and this subset of parameter space  itself, and try to follow at least part of the programme introduced by Yoccoz for quadratic polynomials, and generalised by others, including Roesch \cite{Roesch} to other families of rational maps.

We have the following theorem. 
 
\begin{theorem}\label{2.1} Let $f$ be a rational map with critical value set $Y$. Let $G\subset \overline{\mathbb C}$ be a connected finite graph,  and $r>0$ an integer such that the following hold
\begin{enumerate}
\item  $G\subset f^{-1}(G)$.
\item For each edge $e$ of $G$, $f^n(e)$ is more than a single edge of $G$, for all sufficiently large $n$.
\item  $G$ separates the points of $Y$.
\item  $Y$ is separated from $G$ by $f^{-r}(G)\setminus G$, that is, any path from a point of $Y$ to $G$ must cross $f^{-r}(G)\setminus G$.
\end{enumerate}

Then $f^N$ is expanding in some neighbourhood of $G$ with respect to the spherical metric, for all suffiiciently large $N$. Moreover, for all rational maps $g$ sufficiently close to $f$ in the uniform topology, the properties above hold with $g$ replacing $f$ and a graph $G(g)$ isotopic to the graph $G=G(f)$ above, and varying continuously with $g$.

 In particular, these properties hold for nearby $g$, if $f$ is a rational map such that the forward orbit of every critical point is attracted to an attractive or parabolic periodic orbit, the closures of any two periodic Fatou components are disjoint, and $G$ is a graph with the properties above, and  which is also disjoint from the closure of any periodic Fatou component.
\end{theorem}

\noindent{\em{Proof.}} Define
$${\mathcal{P}}_0=\{\overline{W}:W\mbox{ is a component of }\overline{\mathbb C}\setminus G\} .$$
Then define
$${\mathcal{P}}_n=\{ P':P'\mbox{ is a component of }f^{-n}(P)\mbox{ for some }P\in {\mathcal{P}}\} .$$

First we show that $f^N$ is expanding on a suitable neighbourhood of $G$, for $N$ sufficiently large. if $x\in G$, and $W$ is the union of sets of ${\mathcal{P}}_r$ containing $x$, and $S_n$ is sequence of  local inverses of   $f^n$ with $f\circ S_{n+1}=S_n$ and  $S_n(x)\in G$, then $\mbox{diameter}(S_n(W))\to 0$ as $n\to \infty $, uniformly in $x$. For suppose not. Then $\mbox{diameter}(S_{m_n}(W))\to 0\ge \varepsilon $ for some $\varepsilon >0$, some $x$ and subsequence $m_n$, and then $\limsup _{n\to \infty }S_{m_n}(\mbox{int}(W))$ is contained in the Fatou set of $f$ and has non-empty interior, which intersects $G$. It follows that $G$ intersects a periodic component of the Fatou set. This component must be an attractive or parabolic component, because a Siegel disc or Herman ring cannot intersect a forward invariant graph $G$, unless the graph is a finite union of circles in the Fatou set, which does not satisfy the stated conditions of $G$. So this Fatou component can be assumed to contain a critical value of $f$, that is, a point $y_0$  of $Y(f)$, whose forward orbit is attracted to the periodic orbit of an attractive or parabolic periodic point $x_0$. We can assume that $f^{tn}(y_0)\to x_0$ as $n\to \infty $, where $t$ is the period of $x_0$. Let $W_0$ be the union of sets of ${\mathcal{P}}_r$ containing  $x_0$. Then $W_0$ is a closed neighbourhood of $x_0$. Let $T$ be the local inverse of $f^t$ with $T(x_0)=x_0$. We have $T^n(\partial W_0)\subset W_0$ for all $n$, and $f^{tn}(T^n(\partial W_0))=\partial W_0$. Let $W_1$ be the attracting  petal of $x_0$ with $y_0\in \partial W_1$ and  $f^{tn}(z)\to x_0$ as $n\to \infty $, uniformly for $z\in W_1$. Then $T^n(\partial W_0)\cap W_1=\emptyset $ for sufficiently large $n$. This is a contradiction, because $T^n(\partial W_0)$ separates $y_0$ from $x$. So $\mbox{diameter}(S_n(W))\to 0$ as $n\to \infty $, uniformly in $x$. Let $U'=U_0'$ be the union of all sets of ${\mathcal{P}}_r$ which intersect $G$. and $U_n'$ the union of all sets of ${\mathcal{P}}_{n+r}$ which intersect $G$. Then $f^{n-t}(U_n')=U_t'$ for each $0\le t\le n$, and for all sufficiently large  $N$, we have $U_{N}'\subset \mbox{int}(U_0')$ and $f^N$ is expanding on $U_{N}'$ with respect to the spherical metric $d_0$. Then  $f$ is expanding on $U_{N}'$ with respect to the metric $d_1$, where
$$d_1(z,w)=\sum _{i=0}^{N-1}d_0(f^i(z),f^i(w)).$$

We are going to construct closed neighbourhoods $U$, $U_1$ of $G$ with $U_1\subset \mbox{int}(U)$ and $f(U_1)=U$ and such that the inclusion of $G$ in each of $U_1$ and $U$ is a homotopy equivalence. 
  Our set $U$ will  be a perturbation of $U_{0}'$, which can be taken arbitrarily close to $U_0'$. We have $U_{n+1}'\subset U_n'$  and  $U_{n+N}'\subset \mbox{int}(U_n')$ for all $n$. So  we can write $\partial U_0'$ as a union of  sets $\partial _i$ which are open in $\partial U_0'$, for $1\le i\le N$, such that $U_{i}'\cap \partial _i'=\emptyset $. To obtain the set $U$, we make successive perturbations of $U_0'$ near $\partial _i'$ to $U_{0,i}'$ for $1\le i<N$, so that 
$$U_{0,1}'\subset U_0',$$
$$U_{0,i}'\subset f(U_{0,i}',)\ 1\le i<N,$$
$$U_{0,i+1}'\subset U_{0,i}',\ 1\le i<N$$
 and 
$$f^{-1}(U_{0,i}')\cap \cup _{j=1}^{i+1}\partial _{j,i}=\emptyset ,$$
where $\partial _j$ is perturbed to $\partial _{j,i}$ in $U_{0,i}'$. Then $U_{0,N-1}'=U$, and $U_n$ the corresponding perturbation of $U_{n}'$ for each $n\ge 0$. Then $U$ and $U_1$ have the required properties. 

For $g$ sufficiently close to $f$, we can perturb $U_1$ to $U_1(g)$, which varies isotopically for $g$ near $f$, with $U_1(f)=U_1$, $U_1(g)\subset U$. We have a homeomorphism $k_1:U\to U_1$ which is the identity on $G$, and a decreasing sequence of closed neighbourhoods $U_n$ of $G$, which are the images of  homeomorphisms $k_n:U_n\to U_{n+1}$ satisfying $f\circ k_n=k_{n-1}\circ f$ and $k_n=\mbox{identity}$ on $f^{-n}(G)$. Correspondingly, we have   a homeomorphism $k_g:U\to U_1(g)$ which is the identity on $G$. By successive lifts of this homeomorphism, we obtain  sets  $U_n(g)$ with $U_{n+1}(g)\subset \mbox{int}(U_n(g))$, $g(U_{n+1}(g))=U_n(g)$ and homeomorphisms  $k_{n,g}:U_n(g)\to U_{n+1}(g)$ satisfying $g\circ k_{n+1,g}=k_{n,g}\circ g$. We also have a homeomorphism $h_g:\overline{\mathbb C}\to \overline {\mathbb C}$ which is the identity outside $U_N(g)$ and mapping $G$ to $G_1(g)\subset g^{-1}(G)$ where $G_1(g)$ is an arbitrarily small perturbation of $G_1(f)=G$, by taking $g$ arbitrarily close to $f$. Then $h_g$ can be taken arbitrarily close to the identity in the $C^1$ topology, by taking $g$ arbitrarily close to $f$. 

Then both $g$ and $g\circ h_g$ are expanding on $U_N(g)$ with respect to the metric $d_1$, for $g$ sufficiently close to $f$. We can then follow the method of proof of Theorem \ref{1.1} for $f^N$ to obtain a graph $G(g)$, which is homeomorphic  to  $G_0$ under a homeomorphism of $\overline{\mathbb C}$ which is arbitrarily close to the identity for $g$ arbitrarily close to $f$, with $G(g)\subset g^{-1}(G(g))$. \ep

So we see that there are natural conditions under which an isotopically varying graph $G(g)$ exists, with $G(g)\subset g^{-1}(G(g))$, for an open connected  set of $g$ which are not all hyperbolic. In fact these open connected sets will intersect infinitely many hyperbolic components. We also have an isotopically varying Markov partition ${\cal{P}}(g)$ given by 
$${\mathcal{P}}(g)=\{ \overline{W}:W\mbox{ is a component of }\overline{\mathbb C}\setminus G(g)\} .$$
We now proceed to investigate the boundary of the set of $g$ in which  $G(g)$ and  ${\mathcal{P}}(g)$ exist.
We define
$${\mathcal{P}}_n(g)=\{ P':P'\mbox{ is a component of }g^{-n}(P)\mbox{ for some }P\in {\mathcal{P}}(g).\} $$
We thus have ${\mathcal{P}}_0(g)={\mathcal{P}}(g)$. 

\begin{theorem}\label{2.2}  Let $V$ be a connected component of an affine  variety over $\mathbb C$ of rational maps $V$ in which the set $Y(f)$ of critical values varies isotopically. Let $V_1$ be a maximal connected subset of $V$ such that, for $g\in V_1$, there exist a finite connected graph $G(g)$,  and an  integer
 $r(g)>0$ with the following properties.
\begin{itemize} 
\item $G(g)$ varies isotopically with $g$ for $g\in V_1$.
\item $G(g)\subset g^{-1}(G(g))$.
\item For each edge $e$ of $G(g)$, $g^n(e)$ is more than one edge of $G(g)$, for all sufficiently large $n$. 
\item $G(g)$ separates points of $Y(g)$.
\item If $P\in {\mathcal{P}}_{r(g)}(g)$ and $P\cap G(g)\ne \emptyset $ then $Y(g)\cap P=\emptyset $.\end{itemize}
Then if $V_2\subset V_1$ is a set such that $\overline{V_2}\setminus V_1\ne \emptyset $, where the closure denotes closure in $V$, the integer $r(g)$ is unbounded for $g\in V_2$.  \end{theorem}

\begin{definition} We shall say that  $Y(g)$ is {\em{combinatorially bounded from $G(g)$ for $g\in V_2$}} if $r(g)$ as above is bounded for $g\in V_2$, that is, for some $r$, $Y(g)$ is separated from $G(g)$ by $g^{-r}(G(g))\setminus G(g)$ for all $g\in V_2$. 
\end{definition}

\begin{remarks}
\begin{enumerate}
\item Because the critical value set $Y(g)$ varies isotopically for $g\in V_1$, the set of critical points also varies isotopically.
\item From the definition of $V_1$, and from Theorem \ref{2.1}, $V_1$ is open.

\end{enumerate}
\end{remarks}

\subsection{Real-analytic coordinates on $G(g)$}
A key idea in the proof of \ref{2.2} is to use  real-analytic coordinates on the graph $G(g)$ for $g\in V_1$, provided by the normalisations of the sets in the complement of the graph. Let $P_i(g)\in {\mathcal{P}}(g)$.

  We have uniformising maps $\varphi _{i,g}:P_i(g)\to \{ z:|z|\le 1\} $ for each $1\le i\le k$, which are holomorphic between interiors, and unique up to post-composition with M\"obius transformations. Then we have a collection of maps $\varphi _{j,g}\circ g\circ \varphi _{i,g}^{-1}$, defined on subsets of the closed unit disc, and mapping onto the closed unit disc. Each of these maps is holomorphic on the intersection of its domain with the open unit disc , and extends by the Schwarz reflection principle to a holomorphic map on the reflection $z\mapsto \overline{z}^{-1}$ of this domain in the unit circle. In particular, each such map is real analytic on the intersection of its domain with the unit circle. 

Now $g:g^{-1}(P_i(g))\to P_i(g)$ is a branched covering, and, since $G(g)$ separates the critical values of $g$, each component of $g^{-1}(P_i(g))$ is conformally a disc, and the closure of each component is a closed topological disc. Let $I(i)$ denote the (finite) set of components of $g^{-1}(P_i(g))$.  Let 
$$\psi _{i,g}:g^{-1}(P_i(g))\to \{ z:|z|\le 1\} \times I(i)$$
 be a uniformising map, once again, holomorphic on the interior and unique up to post-composition with a M\"obius transformation on each component.  Then $\varphi _{i,g}\circ g\circ \psi _{i,g}^{-1}$ is a disc-preserving Blaschke product on each of  a finite union of discs, mapping each one to the same disc  whose degree is the degree of $g\vert P_i(g)$. 
Each map $\varphi _{i,g}\circ g\circ \varphi _{j,g}^{-1}$, where defined, is of the form $(\varphi _{i,g}\circ g\circ \psi _{i,g}^{-1})\circ \psi _{i,g}\circ \varphi _{j,g}^{-1}$. Now we establish an expansion property of these maps. 

 \begin{definition}If $D$ denotes the closed unit disc and $A\subset \partial D$ is a finite set, then we say that the moduli of $(D,A)$ are bounded if $A$ contains less than four points, or if the cross-ratio of any subset of of $A$ consisting of four points is bounded above and below. If $Q$ is a closed topological disc and $B\subset \partial Q$ is finite, then we say that the moduli of $(Q,B)$ are bounded if the moduli of $(\varphi (Q),\varphi (B))$ are bounded, where $\varphi :Q\to D$ is a homeomorphism which is holomorphic on the interior of $Q$.\end{definition}

\begin{lemma}\label{2.11}Let $X(g)$ denote the vertex set of $G(g)$. Suppose that  $N$ is such that for any $i$ and $j$ and component $Q$ of $g^{-N}(P_j(g))$ with $Q\subset P_i(g)$, at least one component of $\partial P_i(g)\setminus \partial Q$ contains at least two vertices of $G(g)$, and the moduli of 
$$\left( \bigcup _{i\in I}P_i(g),g^{-N}(X(g))\cap \partial (\bigcup _{i\in I}P_i(g))\right)$$
  are bounded  for any finite set $I$ such that $\bigcup _{i\in I}P_i(g)$ is a topological disc. Then the maps $\varphi _{i,g}\circ g^{N\ell }\circ \varphi _{j,g}^{-1}$ on arcs of the unit circle, where defined, are expanding   with respect to the Euclidean metric on the unit circle, with expansion constants bounded from $1$, for any  $\ell \ge 1$ which is sufficiently large given the moduli bounds.  \end{lemma}
  
\noindent{\em{Proof.}} It suffices to bound below, by  some $\mu >1$, the  derivative of $\varphi _{i,g}\circ g^{N}\circ \varphi _{j,g}^{-1}$, with respect to a suitable metric $d_p$ which we can show to be boundedly Lipschitz equivalent to the Euclidean metric $d_e$. Then the derivative of  $\varphi _{i,g}\circ g^{N\ell }\circ \varphi _{j,g}^{-1}$ with respect to to $d_p$ is $\ge \mu ^\ell $, and if $d_p/d_e$ is bounded between $C^{\pm 1}$ for some $C\ge 1$, we see that the derivative with respect to $d_e$ is $\ge C^{-1}\mu ^{\ell }$, giving expansion for all $\ell $ such that $C^{-1}\mu ^{\ell }>1$.  So it remains to define $d_p$ so that these properties are satisfied. This is the restriction of a Poincar\'e metric on a suitable surface, one for each component $e$ of $\partial Q\cap \partial P_i(g)$, or union of two such components round a vertex of $g^{-N}(G(g))$ in $\partial Q$, where $Q$ is the closure of a component of $\overline{\mathbb C}\setminus g^{-N}(G(g))$ with $Q\subset P_i(g)$ and $e\subset \partial Q$. For each such component, we consider a union $Q'$ of closures of components of $\overline{\mathbb C}\setminus g^{-N}(G(g))$ contained in $P_i(g)$,  such that $Q'$ is a topological disc and such that the connected component $e'$ of $\partial Q'\cap \partial P_i(g)$ which contains $e$ has $e$ in its interior. We can assume without loss of generality, replacing $G(g)$ by $g^{-M}(G(g))$  for a suitable $M$ if necessary, that  the image of   $Q'$ under $g^N$ is also a closed topological disc -- obviously of the form $\cup _{j\in J}P_j(g)$ --- and that $g^N$ is  a homeomorphism on $e'$. So there is a map of $Q'$ to $\{ z:|z|\le 1,\mbox{Im}(z)\ge 0\} $ which maps $e'$ to the interval $[-1,1]$, and which is conformal on the interior. We then take the restriction of the Poincar\'e metric on the unit disc to $(-1,1)$. This is the metric $d_p$ on $\mbox{int}(e')\supset e$.  The image of $e$ under $g^N$ is an edge of $G(g)$ in $\partial P_j(g)$, or a union of two edges round a vertex in $\partial P_j(g)$, for some $j\in J$. We take the corresponding metric $d_p$ on each edge of $g^{-N}(G(g))$ in $\partial P_j(g)$. Take any edge $e_1$ of $g^{-N}(G(g))$ or union of two edges  of $g^{-N}(G(g))$ which are subsets of edges of $G(g)$, adjacent to a vertex of $G(g)$ in $P_j(g)$, with $e_1\subset e$. Let $Q_1$ be the component of $\overline{\mathbb C}\setminus g^{-N}(G(g))$, and $e_1\subset \partial Q_1$ and $Q_1\subset g^N(Q)$.  Let $Q_1'$ be the union of closures of components of  $\overline{\mathbb C}\setminus g^{-N}(G)$ with $Q_1\subset Q_1'$ which is used to define the metric $d_p$ on $e_1$. Then $Q_1'\subset g^N(Q')$, and by the hypotheses, if we double $g^N(Q')$ across $g^N(e')$ by Schwarz reflection, and then normalise, the image of the double of $Q_1'$ within this is contained in $\{ z:|z|\le r\} $, for some $r<1$ bounded from $1$, simply because there are just finitely many edges. It follows that $g^N$ is expanding on $e$ with respect  to the metric $d_p$, with expansion constant bounded from $1$. \ep

\subsection{Real-analytic maps $h_{1,g}$ and $h_{2,g}$}\label{2.10}
Now each edge of $G(g)$ is in the image of two maps $\varphi _{i_1,g}$ and $\varphi _{i_2,g}$, where the edge is a connected component of $\partial P_{i_1}(g)\cap \partial P_{i_2}(g)$. Since $G(g)\subset g^{-1}(G(g))$, it is also the case that each edge is contained in a union of components of sets $g^{-1}(P_{j_1}(g)\cap P_{j_2}(g))$, where these sets are disjoint apart from some common endpoints. It follows that from $g$, and after imposing a direction on each edge of $G(g)$, we obtain two  real-analytic maps $h_{1,g}$ and $h_{2,g}$, defined piecewise by $\varphi_{j_1,g}\circ g\circ \varphi _{i_1,g}^{-1}$ and $\varphi _{j_2}\circ g\circ \varphi _{i_2,g}$, each mapping a finite union of intervals to itself, mapping endpoints to endpoints, except for being two-valued at finitely many interior points in the intervals, but at these points, the right and left-derivatives exist  and coincide, so that the derivative is single valued at such points, and extends continuously in the neighbourhood of any such point. These two maps are quasi-symmetrically conjugate, because the maps $\varphi _{i,g}$ are quasi-conformal. The quasi-symmetry is unique, and the pair $(\overline{\mathbb C},g^{-1}(G))$ can be reconstructed from it, up to  M\"obius transformation of $\overline{\mathbb C}$. In \ref{2.3}, we make this idea more precise. Lemma \ref{2.11} shows that the hypotheses are satisfied. 

Note that it is possible for the image of $h_{1,g}$ to intersect the domain of $h_{2,g}$, and vice versa, if $g$ maps some edge of $G$ over itself with direction reversed. In that case, since the domain and image of $h_{\ell ,g}$ are to be the same, it can happen that $h_{1,g}$ and $h_{2,g}$ agree on a nonempty intersection between their two domains. But we do not need to make any special consideration of this possibility and even if this happens the quasisymmetric conjugacy constructed in \ref{2.3} need not be the identity. 

\begin{lemma}\label{2.3} Let  $I_{i,t}$ be finite intervals for $1\le i\le k$ and $t=1$, $2$. Let
$$h_t:\bigcup _{i=1}^kI_{i,t}\to \bigcup _{i=1}^kI_{i,t}$$ 
 be  $C^2$ maps which are multivalued just at points which are mapped to endpoints of intervals, but with well-defined continuous derivatives at such points, such that $h_{t}(I_{i,t})$ is a union of intervals $I_{j,t}$ for each of $t=1$, $2$, and $I_{j,1}\subset h_1(I_{i,1})$ if and only if $I_{j,2}\subset h_2(I_{i,2})$, and  $I_{i,t}\cap h_1^{-1}(I_{j,t})$ has at most one component, for both $t=1$ and $2$. Suppose also that there is $N$ such that $h_1^n$ and $h_2^n$ are expanding with respect to the Euclidean metric for all $n\ge N$.  Then $h_1$ and $h_2$ are quasi-symmetrically conjugate, with the norm  of the quasi-symmetric  conjugacy bounded in terms of $N$ and of the bound of the expansion constants of $h_1^N$ and $h_2^N$ from $1$.\end{lemma}
\noindent{\em{Proof.}} This is standard. We simply choose 
$$\varphi _0:\bigcup _{i=1}^kI_{i,1}\to \bigcup _{i=1}^kI_{i,2}$$
 to be an  affine transformation (for example) restricted to $I_{i,1}$, mapping $I_{i,1}$ to $I_{i,2}$, for each $1\le i\le k$. Then $\varphi _n$ is defined inductively by the properties 
 $$h_2\circ \varphi _{n+1}=\varphi _n\circ h_1$$
  and 
  $$\varphi _{n+1}(I_{i,1})=I_{i,2}\mbox{ for each }1\le i\le k.$$ Then 
  $$\varphi _0\circ h_1^n=h_2^n\circ \varphi _n\mbox{  for all }n,$$
   and we deduce from this that 
   $$|\varphi _n(x)-\varphi _{n+1}(x)|\le C_2\lambda ^n$$
    for all $x$ and $n$, for some constant $C_2$ depending on  $C_1$, and some $\lambda <1$,  where 
    $$|h_2^n(x)-h_2^n(y)|\ge C_1\lambda ^{-n}$$
     for all $n$ and all $x$ and $y$ such that $h_2^m(x)$ and $h_2^m(y)$ are in the same set $I_{i_m,2}$, for all $0\le m\le n$. Then $\varphi _n$ converges uniformly to $\varphi $, with 
    $$\varphi \circ h_1=h_2\circ \varphi .$$ Similarly, using the expanding properties of $h_1$, we deduce that $\varphi _n^{-1}$ converges uniformly to $\varphi ^{-1}$. 

To prove quasi-symmetry of $\varphi $, we use the standard result that $(h_t^n)' $ varies by a bounded proportion on any interval $J$ such that $h_t^n(J)$ is a union of at most two subintervals of $\bigcup _{i=1}^kI_{i,t}$. This uses continuity of the derivative across the finitely many discontinuities of $h_t$. So then given any $x\ne y\in \bigcup _{i=1}^kI_{i,1}$ such that $|x-y|$ is sufficiently small, we choose the greatest  $n$ such that $|h_1^n(x)-h_1^n(y)|\le c$, for a suitable constant $c>0$ such that any interval of $\bigcup _{i=1}^kI_{i,1}$ which has length $\le c$ is mapped to a union of at most two intervals of $\bigcup _{i=1}^kI_{i,1}$. Then  
$$|h_1^{n+p}(x)-h_1^{n+p}(y)|$$
 is bounded above and below for any bounded $p$, and $(h_1^{n+p})'$ varies by a bounded proportion on the  interval $[x,y]$. So does  the derivative $S'$, on the smallest interval containing $h_1^{n+p}(x)$, $h_1^{n+p}(y)$, where $S$ is the branch of $h_2^{-(n+p)}$ such that $\varphi _{n+p}=S\circ \varphi _0\circ h_1^{n+p}$. We can choose $p$ so that each of  the points $h_1^n(x)$, $h_1^n(y)$, $h_1^n((x+y)/2)$ is separated by at least two  points from $\bigcup _{i=1}^kh_1^{-p}(\partial I_{i,1})$ --- but only boundedly many, by the bound on $p$. Now 
 $$\varphi _m=\varphi _{n+p}\mbox{ on }\bigcup _{i=1}^kh_1^{-(n+p)}(\partial I_{i,1})$$
  for all $m\ge n+p$, and hence 
  $$\varphi =\varphi _{n+p}\mbox{ on }\bigcup _{i=1}^kh_1^{-(n+p)}(\partial I_{i,1}).$$
  If $z_1$, $z_2$ and $z_3$ are any three distinct points of   $\bigcup _{i=1}^kh_1^{-(n+p)}(\partial I_{i,1})$ which are either between $x$ and $y$, or the nearest point on one side, then 
  $$\frac{|\varphi _{n+p}(z_1)-\varphi _{n+p}(z_2)|}{|\varphi_{n+p}(z_1)-\varphi _{n+p}(z_3)|}$$ 
  is bounded and bounded from $0$, that is, 
  $$\frac{|\varphi (z_1)-\varphi (z_2)|}{|\varphi (z_1)-\varphi (z_3)|}$$
   is bounded and bounded from $0$. But then since $|\varphi (x)-\varphi ((x+y)/2)|$ is bounded between some such $|\varphi (z_1)-\varphi (z_2)|$ and $|\varphi (z_1)-\varphi (z_3)|$, and similarly for $|\varphi (y)-\varphi ((x+y)/2)|$, we have upper and lower bounds on 
   $$\frac{|\varphi (x)-\varphi ((x+y)/2)|}{|\varphi (y)-\varphi ((x+y)/2)|},$$ and quasi-symmetry follows.\ep

We deduce the following.

\begin{lemma}\label{2.4} Let $V_1$ be as in Theoreom \ref{2.2}. For $f\in V_1$, let $P_i(f)$, $\varphi _{i,f}$ and $\psi _{i,f}$ be as previously defined. Let $\{ g_n:n\ge 0\} $  be any sequence in $V_1$ such that $Y(g_n)$ is combinatorially bounded from $G(g_n)$ for $n\ge 0$, and let $g_n\to g$. Let $X(g_n)$ denote the vertex set of $G(g_n)$. Then $g\in V_1$ if the moduli of 
\begin{equation}\label{2.4.1}\left(\bigcup _{i\in I}P_i(g_n),g^{-\ell }(X(g_n))\cap \partial \left(\bigcup _{i\in I}P_i(g_n)\right) \right)\end{equation}
 are bounded as $n\to \infty $ for any fixed $\ell $, and any finite set $I$ such that $\bigcup _{i\in I}P_i(g_n)$ is a topological disc, and, using this to normalise the maps $\varphi _{i,g_n}$ and $\psi _{i,g_n}$, the disc-preserving Blaschke products $\varphi _{i,g_n}\circ g_n\circ \psi _{i,g_n}^{-1}$ are also bounded.\end{lemma}

\noindent{\em{Proof.}} The bounds on moduli and Blaschke products ensure that the real analytic maps $h_{1,g_n}$ and $h_{2,g_n}$ have derivatives which are bounded above and below. Also, they extend to Blaschke products on neighbourhoods of intervals of the unit circle. By the hypothesis $r(g_n)\le r$ for all $n$, there is $N$ such that,  if $U'(g_n)$, $U_i'(g)$ are the unions of sets of ${\mathcal{P}}_r(g_n)$ , ${\mathcal{P}}_{r+i}(g)$ intersecting  $G(g)$, then $U_N'(g_n)\subset \mbox{int}(U'(g_n))$ and $g_n^N(U_N'(g_n))=U'(g_n)$. We have seen from \ref{2.11} and \ref{2.3} that the maps $h_{1,g_n}$ and $h_{2,g_n}$ are boundedly quasi-symmetrically conjugate, that is, there is a quasi-symmetric homeomorphism $\varphi _n$ whose domain is the domain, and contains the image,  of $h_{1,g_n}$, and whose image is the domain, and contains the image,  of $h_{2,g_n}$, that is, a finite union of intervals in each case, such that 
$$\varphi _n\circ h_{1,g_n}=h_{2,g_n}\circ \varphi _n.$$

Then $\varphi _n$ can be used to define a Beltrami differential $\mu _n$ on $\overline{\mathbb C}$, which is uniformly bounded independently of $n$, as follows. This sphere is, topologically, a finite union of discs, with the boundary of each disc written as a finite union of arcs, and with each arc identified with one other, from a different disc, by $\varphi _n$ in one direction and $\varphi _n^{-1}$ in the other.  It is convenient to identify this sphere with the Riemann sphere $\overline{\mathbb C}$, in such a way that  each of the discs has piecewise smooth boundary, and the maps identifying the copies of the closed unit disc with the image discs in  $\overline{\mathbb C}$ are piecewise smooth. The union of the images of copies of the unit circle form a graph $\Gamma \subset \overline{\mathbb C}$. We then define a quasi-conformal homeomorphism $\psi _n$ from the union of copies of the closed unit disc to $\overline{\mathbb C}$ such that, whenever  $I_1$ and $I_2$ are arcs on the boundaries of discs $D_1$ and $D_2$, identified by $\varphi _n:I_1\to I_2$,  we have $\psi _n$ on $I_2$ is defined by $\psi _n\circ \varphi _n^{-1}$, using $\varphi _n^{-1}:I_2\to I_1$ and $\psi _n:I_1\to \overline{\mathbb C}$. The q-c norm of $\psi _n$ can clearly be bounded in terms of the q-s norm of $\varphi _n$, and the identification we choose of the copies of the closed unit disc with their images in $\overline{\mathbb C}$. This means that the q-c norm of $\varphi _n$ can be bounded independently of $n$. We then define 
$$\mu _n=(\varphi _n)_*0$$
 on the image of each copy of the open unit disc, where $0$  simply denotes the Beltrami differential which is $0$ everywhere on the open unit disc. Then $\mu _n$  is defined a.e. on $\overline{\mathbb C}$, and is uniformly bounded, in $n$, in the $L_\infty $ norm.  

So there is a quasi-conformal map $\chi _n:\overline{\mathbb C}\to \overline{\mathbb C}$, with q-c norm which is uniformly bounded in $n$, such that $\mu _n=\chi _n^*0$, where, here, $0$ denotes the Beltrami differential which is $0$ everywhere on $\overline{\mathbb C}$. By construction, there is a conformal map of $\overline{\mathbb C}$ which maps $\chi _n(\Gamma )$ to $G(g_n)$. So we can assume without loss of generality that $\chi _n(\Gamma )=G(g_n)$. By taking limits, we can assume that $\chi _n$ has a limit $\chi $ in the uniform topology, which is a quasi-conformal homeomorphism. So $\chi _n(\Gamma )$ has a limit $\chi (\Gamma )$, which is also a graph, and since $G(g_n)\subset g_n^{-1}(G(g_n))$, we have $\chi (\Gamma )\subset g^{-1}(\chi (\Gamma ))$. We therefore write $G(g)=\chi (\gamma )$ and $G(g)$ is homeomorphic to $G(g_n)$ under $\chi \circ \chi _n^{-1}$. We can also assume, by restricting to a subsequence of $g_n$ if necessary,  that for each $i\le r$, and all $n$, the sets $g_n^{-i}(G(g_n))\cup Y(g_n)$ are isotopic.  
The bounds on moduli (\ref{2.4.1}) then give a lower bound on modulus of each component of $U'(g_n)\setminus U_N'(g_n)$, independent of $n$. So then $U'(g_n)$ converges to $U'(g)$, while  $U_N'(g_n)$ converges isotopically  to $U_N'(g)$ and $g\in V_1$ with $r(g)\le r+N$. Note that we could have $Y(g)\cap \partial U'(g)\ne \emptyset $. 
 \ep

Since $G(g)$ varies isotopically for $g\in V_1$, the set $X(g)$ of vertices of $G(g)$ also varies isotopically for  $g\in V_1$. But $X(g)$ is a finite forward invariant set for all $g\in V_1$. Hence $X(g)$ varies locally isotopically for $g$ in the dense open  subset $V_0$ of $V$ such that the multiplier of any periodic points in $X(g)$ is not $1$, and there are no critical points in $X(g)$. We have $V_1\subset V_0$, since $g^N$ is expanding near $G(g)$ for $g\in V_1$, for a suitable $N$, by \ref{2.1}. Now, to prove Theorem \ref{2.2}, we need to verify the conditions of Lemma \ref{2.4}.

\begin{definition} A path $\alpha $ with endpoints in $X(g)$ {\em{has homotopy length $\le M$}} if it can be isotoped, by an isotopy which is the identity on $X(g)$, to be arbitrarily uniformly close to a path in $G(g)$ which crosses $\le M$ edges of $G(g)$. 
\end{definition}

\begin{lemma}\label{2.5} Let $V$ and $V_1$ be as in \ref{2.2}. Let $V_0$ be as at the end of \ref{2.4}. Fix $g_0\in V_1$. Let $W_0$ be a compact subset of $V$ containing $g_0$, and let $M_0>0$ be given. There is $M_1=M_1(M_0,W_0)$ with the following property. Let $g\in V_1\cap W_0$.  If $e$ is an edge  of $G(g)$ and  $e'\subset e$ is a connected set which shares its first endpoint with $e$, and $\alpha $ is any  extension of $e'$ by spherical length $\le M_0$ to a path with both endpoints in $X(g)$, then $\alpha $ has homotopy length $\le M_1$.\end{lemma}

\noindent{\em{Proof.}} Let  $g_t$ be a path in $V_1$  between $g_0$ and $g=g_1$. Since $V\setminus V_0$ has codimension two, we can assume without loss of generality, enlarging $W_0$ if necessary, that $g_t\in V_0\cap W_0$ for all $t$, so that $X(g_t)$ varies isotopically. We can choose the path $g_t$ so that its length is bounded in terms of $W_0$, using any suitable Riemannian metric on $V$, for example, that coming from the embedding of $V$ in $\mathbb C^m$ (since $V$ is an affine variety). 

Now given $R >1$, there is $k$ such that $g^k(e'')$ is a union of at least $R$ edges for each edge $e''$ of $G(g)$. This is true for all $g\in V_1$, because the dynamics of the map $g:G(g)\to G(g)$ is independent of $g$. We take $R=2$. For this $k$ (or, indeed, any strictly positive integer), $\bigcup _{\ell \ge 0}g^{-\ell k}(X(g))$ is dense in $G(g)$, because, for any edge $e$ of $G(g)$, the maximum diameter of any component of $g^{-n}(e)$ tends to $0$ as $n\to \infty $, by ref{2.1}. So  it suffices to prove the lemma for $e'\subset e$ sharing first endpoint with $e$ and with the second endpoint in $g^{-\ell k}(X(g))\setminus g^{-(\ell -1)k}(X(g))$ for some $\ell \ge 0$, but we cannot obtain any bound on $\ell $. So fix such an $e'$. For each $i\le \ell $, let  $e_{ik}=e_{ik}(g)\subset e$ such that $g^{ik}(e_{ik})$ is an edge of $G(g)$, hence with endpoints in $X(g)$, such that the second endpoint of $e'$ is in $e_{ik}$, and is not the first endpoint of $e_{ik}$. 

Any point of $\overline{\mathbb C}$ is spherical distance $\le \pi $ from a point of $X(g)$ (assuming the sphere has radius $1$). Any path of bounded (spherical) distance between points of $X(g)$ is homotopically bounded, because of the bounded  distance between $X(g_0)$ and $X(g)$. We suppose for contradiction that, for some  path $\alpha _0$ of length $\le M_0$  from the second endpoint of $e'$ to a point of $X(g)$, the path $e'*\alpha _0$ has homotopy length $\ge M_1$. Then $g^k(e'*\alpha _0)$ has homotopy length $\ge 2M_1$. Now let $\alpha _k$ be a path of spherical length $\le M_0$ connecting the second endpoint of $g^k(e')$ to $X(g)$.  Now we have a bound on the homotopy length of $g^k(e'\setminus e_k)$ depending only on $k$, because this is a union of a number of edges of $G(g)$, where the number is bounded in terms of $k$. We also have a bound in terms of $k$ and  $M_0$ (and on $g_0$, but $g_0$ is fixed throughout)  on the spherical length of $\overline{\alpha _k}*g^k(\alpha _0)$, where $\overline{\alpha _k}$ denotes the reverse of $\alpha _k$. This is because the bound on the path between $g_0$ and $g$ gives a bound  on the spherical derivative of $g^k$ in terms of $M_0$ and $k$.    If $\varphi $ is the homeomorphism of $\overline{\mathbb C}$ given by the isotopy from the identity mapping $X(g)$ to $X(g_0)$, then $\varphi $ is bounded in terms of $M_0$. So we have a bound on the spherical length of $\varphi (\overline{\alpha _k}*g^k(\alpha _0))$. This is a path between points of $X(g_0)$. So we have a bound on the homotopy length of this path in terms of $M_0$ and $k$ (and $g_0$, but this is fixed throughout). But the homotopy length is the same as the homotopy  length of  $\overline{\alpha _k}*g^k(\alpha _0)$.   So  both $g^k(e'\setminus e_k)$ and $\overline{\alpha _k}*g^k(\alpha _0)$ have  homotopy length $\le M_0'$ where $M_0'$ is bounded in terms of $M_0$ and $k$. So then $g^k(e'\cap e_k)*\alpha _k$ has homotopy length $\ge 2M_1-2M_0'>M_1$ assuming that $M_1$ is sufficiently large given $M_0'$ and $k$, that is, sufficiently large given $M_0$.  Similarly, for each $i$,  $g^{k}((e'\cap e_{(i-1)k})\setminus e_{ik})$ and $\overline{\alpha _{ik}}*g^k(\alpha _{(i-1)k})$ have homotopy length $\le M_0'$, and hence we prove by induction that $g^{ik}(e_{ik}\cap e')$ has homotopy length $>M_1$ for all $i\ge 0$. For $i=\ell $ we obtain the required contradiction, because $g^{\ell k}(e'\cap e_{\ell k})$ is a single edge.\ep

\begin{corollary}\label{2.6} Let $V$, $V_1$, $g_0$, $M_0$, $W_0$ and $g$ be as in \ref{2.5}. There is $M_2>0$, depending on $M_0$, $W_0$ and $g_0$ with the following property. If $e'$ is any path in an edge of $G(g)$ then $e'$ is homotopic, via a homotopy fixing endpoints and $X(g)$, to a path of (spherical) length $\le M_2$.\end{corollary}
\noindent{\em{Proof.}} It suffices to prove this for paths with one endpoint at $X(g)$, because $e'=\overline{e_1'}*e_2'$ for two such paths in the same edge as $e'$. So now assume that $e'$ shares an endpoint with $e$.  Then by \ref{2.5}, we can extend $e'$ by spherical length $\le M_0$  to a path $\alpha $ with both endpoints in $X(g)$  so that $\alpha $ is homotopic, via  a homotopy fixing $X(g)$, to an arbitrarily small neighbourhood of a path crossing $\le M_1$ edges of $G(g)$. Because the movement of $X(g_0)$ to $X(g)$ is bounded, this means that $\alpha $ is homotopic, via a homotopy fixing $X(g)$, to a path of spherical  length $\le M_2'$. Then since $e'$ can be obtained from $\alpha $ by adding length $M_0$, we obtain the required bound on $e'$ with $M_2=M_2'+M_0$.\ep

\begin{lemma}\label{2.7}  Let $V$, $V_1$, $g_0$, $M_0$, $W_0$ and $g$ be as in \ref{2.5}. There is $\varepsilon >0$ depending on $M_0$ and $g_0$ such that  for each $i$, there is some point in  $P_i(g)$ which is distance  $\ge \varepsilon $ from $\partial P_i(g)$.\end{lemma}
\noindent{\em{Proof.}} It suffices, for some $x\in P_i(g)$ and for some fixed $n$, to find a lower bound on the length of $g^n\alpha $, where $\alpha $ is any path from $x$ to $\partial P_i(g)$. By \ref{2.6}, we can extend $g^n\alpha $ by a path $\gamma $ in some  $\partial P_j(g)\cap g^n(\partial P_i(g))$ to a point of $X(g)$, such that $\gamma $ is homotopic, via a homotopy fixing endpoints  and $X(g)$, to a path of length $\le M_2$, which is independent of $n$. But we can choose $x\in g^{-n}(X(g))$, for some $n$, so that if $\alpha '$ is any path from $x$ to $\partial P_i(g)\cap g^{-n}(X(g))$ then the homotopy length of $g^n\alpha '$ is $>M_3$, where $M_3$ is sufficiently long to force spherical length $>2M_2$. We do this using  the bound on the isotopy distance between $X(g)$ and $X(g_0)$, and the number of sets of ${\mathcal{P}}(g)$ that $g^n(\alpha ')$ must cross. Then the spherical length of $g^n\alpha $ is $>M_2$, which gives us  a strictly positive lower bound on the spherical length of $\alpha $: in terms of $n$, which means, ultimately, in terms of $M_0$. \ep

In a similar way, we can prove the following.

\begin{lemma}\label{2.8}  Let $V$, $V_1$, $g_0$, $M_0$ , $W_0$ and $g$ be as in \ref{2.5}.  Let $A$ be any embedded annulus which is a union of $N_1\ge 1$ components of sets $g^{-r}(P_i(g))$ (for varying $i$) surrounding a union of $N_2\ge 1$ components of sets $g^{-r}(P_j(g))$ (for varying $j$). Then the modulus of $A$ is bounded and bounded from $0$, where the bounds depend on $N_1$, $N_2$, $M_0$, $g_0$ and $r$.\end{lemma}
\noindent{\em{Proof.}} It suffices to prove this with $r=0$, since the result remains true under branched covers, just depending on $r$ and the degree of $g_0$.  The upper bound on modulus is clear, from the bound on the diameter of the sets $P_i(g)$ from \ref{2.5} and on the lower bound on the interior of sets $P_j(g)$ in \ref{2.7}. Actually a lower bound on the diameter of the sets $P_j(g)$ is enough, and this is easily obtained. So now we need to bound the modulus below. For this, we need to bound below the length (in the spherical metric) of any path $\gamma $ between the two boundary components of $A$. As in \ref{2.7}, it suffices to bound below the length of $g^n(\gamma )$, for some fixed $n$, and it suffices to show that this length tends to $\infty $ with $n$. As in \ref{2.7}, it suffices to prove this for paths with endpoints in $X(g)$, in distinct components of $\partial A$, and this length tends to $\infty $ because of the bounded homotopy distance of points in $X(g)$ from $X(g_0)$, and the homotopy length tends to $\infty $. \ep

Then using this, we can prove the following. 
\begin{lemma}\label{2.9} Let $V$, $V_1$, $g_0$, $M_0$, $W_0$ and $g$ be as in \ref{2.5}.  The moduli of 
$\left(\bigcup _{i\in I}P_i(g),g^{-t}(X(g))\cap \partial \left( \bigcup _{i\in I}P_i(g)\right)\right)$
 are bounded whenever $\bigcup _{i\in I}P_i(g)$ is a topological disc, with bound depending only on $M_0$, $g_0$, $t$ and $\#(I)$.  \end{lemma}

\noindent{\em{Proof.}} Write $Q=\bigcup _{i\in I}P_i(g)$, for any fixed $I$ such that $Q$ is a topological disc. If $(x_1,x_2,x_3,x_4)$ is an ordered quadruple of four points of $\partial Q\cap g^{-t}(X(g))$, with $x_1$and  $x_2$ not separated in $\partial Q$ by the set $\{ x_3,x_4\} $, then we define the modulus of $(x_1,x_2,x_3,x_4)$ to be the modulus of the rectangle $\varphi (Q)$ where $\varphi $ is conformal on the interior and the vertices are the points $\varphi (x_i)$. In turn, we define modulus to be the modulus of the annulus formed by identifying the edge of the rectangle joining $\varphi (x_1)$ and $\varphi (x_2)$ to the edge joining $\varphi (x_3)$ and $\varphi (x_4)$.  So it suffices to bound below the modulus of each such quadruple $(x_1,x_2,x_3,x_4)$. But then it suffices to do it in the case when $x_1$ and $x_2$ come from adjacent points of $g^{-t}(X(g))$ on $\partial Q$, and similarly for $x_3$ and $x_4$, because ${\rm{modulus}}(A_1)\le {\rm{modulus}}(A_2)$ if $A_1\subset A_2$ and the inclusion is injective on $\pi _1$. But if we have two disjoint edges on $\partial Q$, we can make an annulus which includes $Q$ and encloses a union of partition elements $P_j(g)$. The partition elements $P_j(g)$ are those with edges on one path in $\partial Q$ between the edges associated with $(x_1,x_2)$ and $(x_3,x_4)$.  So the lower bound on the modulus of $(x_1,x_2,x_3,x_4)$ comes from the lower bound of this annulus, which was obtained in \ref{2.8}.  \ep

\subsection{Proof of Theorem \ref{2.2}}\label{2.10}

We recall that we are making the assumption that $Y(g_n)$ is combinatorially bounded from $G(g_n)$. We need to check that the assumptions of Lemma \ref{2.4} are satisfied, since Theorem \ref{2.2} will then immediately follow. Lemma \ref{2.9} gives the bounds on the moduli of  
$$\left(\bigcup _{i\in I}P_i(g_n),g_{n}^{-t}(X(g_n))\cap \partial \left(\bigcup _{i\in I}P_i(g_n)\right)\right),$$
for any particular $t$. By \ref{2.8}, the set $Y(g_n)$ is bounded from $G(g_n)$ by a union of annuli of moduli bounded from $0$. Together with the bound on the moduli of $(P_i(g_n),X(g_n)\cap \partial P_i(g_n))$, which is just used for normalisation, this gives the required bound on the Blaschke products $\varphi _{i,g_n}\circ g_n\circ \psi _{i,g_n}^{-1}$ of \ref{2.4}, and the proof is completed.

\section{Parametrisation of existence set of Markov partition}\label{3}

\subsection{}\label{3.0}

In Section \ref{2}, the parameter space $V$ was a connected component of an affine variety over $\mathbb C$. In this section, we put more restrictions on $V$. In particular, the restrictions include that $V$ is of complex dimension one. This means that we are looking at a familiar scenario, in which it is reasonable to suppose that parameter space can be described by movement of a single critical value. It is certainly possible that the ideas generalise to higher dimensions. But there are still new features to consider, even for $V$ of complex dimension one.

We restrict to the case of  $V$ being a parameter space of quadratic rational maps $g$ with numbered critical points for which one critical point $c_1(g)$ is periodic of some fixed period and the other, $c_2(g)$, is free to vary. The family of such maps, quotiented  by M\"obius conjugation, is of complex dimension one, and is well known to have no finite  singular points. (See, for example,  Theorem 2.5 of \cite{K-R}.) So  $V$, or a natural quotient of it, is a Riemann surface, with some punctures at $\infty $, where the degree of the map degenerates. So we assume from now on that $V$ is a Riemann surface. We write $v_1(g)=g(c_1(g))$ and $v_2(g)=g(c_2(g))$ for the critical values. Fix  a postcritically finite map $g_0\in V$ for which a connected finite  graph $G(g_0)$ exists with $G(g_0)\subset g_0^{-1}(G(g_0))$ and such that the conditions of Theorem \ref{2.1} hold.   Write 
$${\cal{P}}={\cal{P}}(g_0)=\{ \overline{U}:U\mbox{ is a component of }\overline{\mathbb C}\setminus G(g_0)\} .$$
We write $V(G(g_0),g_0)$ for  the  connected set $V_1\subset V$ containing $g_0$ defined in \ref{2.2}, with $g_0$ replacing $f$.
We write $V([G(g_0)])$ for the subset of $V$, which is the union of sets $V(G(g_1),g_1)$ for which 
$$G(g_0)\cup \{ g_0^i(v_1(g_0)):i\ge 0\} \cup \{ v_2(g_0)\} $$
and
$$G(g_1)\cup \{ g_1^i(v_1(g_1)):i\ge 0\} \cup \{ v_2(g_1)\} $$
are isotopic. Thus, $V(G(g_0),g_0)$ is a component of $V([G(g_0)])$, which could, potentially, have more than one component.

 In Section \ref{2} we found a partial characterisation of the boundary of $V(G(g_0),g_0)$.  Now we want to try and obtain a parametrisation of this set. For any $g\in V([G(g_0)])$, and integer $n\ge 0$, we define
$$G_n(g)=g^{-n}(G(g)).$$
 We continue with the notations  ${\mathcal{P}}(g)$ and ${\mathcal{P}}_n(g)$ established at the end of \ref{2.1}. Thus, $G_n(g)$ is the union of boundaries of the sets of ${\mathcal{P}}_n(g)$.

\subsection{The possible graphs}\label{3.1}

Let $g_0\in V$ and $G(g_0)$ be as above. Following a common strategy, we want to use the dynamical plane of  $g_0$ to investigate the variation of dynamics in  $V(G(g_0),g_0)$. The set 
$$G(g)\cup \{g^i(v_1(g)):i\ge 0\} \cup\{v_2(g)\} $$  varies isotopically   for $g\in V(G(g_0),g_0)$. In fact $(G(g),g)$ varies continuously as a dynamical system, because, by \ref{2.1} and \ref{2.2}, backward orbits of vertices of $G(g)$ are dense in $G(g)$. Also,
$$G_1(g)\cup \{ g^i(v_1(g):i\ge 0\} $$
varies isotopically with $g\in V(G(g_0),g_0)$.  But, because $v_2(g)$ is not included in this isotopically varying set, it is not true that $G_n(g)$ varies isotopically for $n>1$. But nevertheless, it is possible to determine inductively all the possible graphs $G_n(g)$ up to  isotopy, for $g\in V(G(g_0),g_0)$. The different possibilities for $G_n(g)$, up to  isotopy, are determined from the different possibilities for $G_{n-1}(g)\cup \{ v_1(g),v_2(g)\} $ up to isotopy. Inductively, this means that the different possibilities for $G_n(g)$ (and ${\cal{P}}_n(g)$), up to  isotopy, are determined by $(Q_i(g):0\le i\le n-1)$, where:
\begin{itemize}
\item $Q_0=Q_0(g)$ is the set in ${\mathcal{P}}(g)$ with $v_2(g)\in \mbox{int}(Q_0)$;
\item  $Q_{i+1}(g)\subset Q_{i}(g)$ for $0\le i\le n-1$;
\item $Q_i(g)\in {\mathcal{P}}_i(g)$ or $Q_i(g)$ is an edge of $G_i(g)$ or a vertex of $G_i(g)$;

\item $v_2(g)\in Q_i(g)$ for $i\le n-1$ and $v_2(g)\in \mbox{int}(Q_i(g)$ if $Q_i\in {\mathcal{P}}_i(g)$, and $v_2(g)$ is not an endpoint of $Q_i(g)$ if $Q_i(g)$ is an edge of $G_i(g)$.\end{itemize}

 Inductively, this means that the different possibilities for $Q_n(g)$  are determined by $Q_i(g)$, for $0\le i\le n-1$, and hence so is the graph $G_{n}(g)$, up to homeomorphism of $\overline{\mathbb C}$, and the dynamical system $(G_n(g),g)$, up to isomorphism. So the different possibilities for any sequence $(Q_i:0\le i\le n-1)$ as above, or even any infinite sequence $(Q_i:i\ge 0)$ with these properties, are determined by $g_0:G_1(g_0)\to G(g_0)$, up to homeomorphism of $\overline{\mathbb C}$ which is the identity on $\partial Q_0$. We will write ${\mathcal{Q}}={\mathcal{Q}}([g_0])$ for the set of sequences, either finite or infinite, up to equivalence, where two sequences $(Q_i:i\ge 0)$ and $(Q_i':i\ge 0)$  are regarded as equivalent if there is a homeomorphism $\varphi $ of $\overline{\mathbb C}$ which  maps $Q_i$ to $Q_i'$ for all $i\ge 0$. We will write ${\cal{Q}}_\infty $ for the set of infinite sequences in ${\cal{Q}}$, and ${\cal{Q}}_n$ for the set of finite sequences $(Q_{0},\cdots Q_{n})$ in ${\cal{Q}}$.  For $\underline{Q}=(Q_0,\cdots Q_{n-1})\in {\cal{Q}}$, we write $V(\underline{Q},g_0)$ or $V(\underline{Q},[g_0])$ for the set of $g\in V(G(g_0),g_0)$ or $V([G(g_0)])$ such that $(Q_i(g):0\le i\le n-1)$ is equivalent to $(Q_0,\cdots Q_{n-1})$. For $g\in V(\underline{Q},g_0)$, the graphs $G_n(g)$ are all isotopic, and and for $g\in V(\underline{Q},[g_0])$, the dynamical systems  $(G_n(g),g)$,  are isomorphic. We write $G(\underline{Q})$ for this graph and ${\mathcal{P}}(\underline{Q})$  for the corresponding partition, where, because of the isomorphism of the dynamical systems, there is a canonical homeomorphism between $G(\underline{Q})$ and $G_n(g)$ for any $g\in V(\underline{Q},[g_0])$ which varies continuously with $g$, and therefore induces an isotopy of $G_n(g)$.  This  homeomorphism  is actually a bit more general, which will be important later.  Let $Q_{n-1}'\subset Q_{n-1}$ be an edge  or point  of $G(\underline{Q})$ (not necessarily a vertex). Then $g^{-1}(G_{n-1}(g)\setminus Q_{n-1}'(g))$ are all canonically homeomorphic, with homeomorphism varying continuously  for $g\in V(Q_0,\cdots Q_{n-1},[g_0])\cup V(Q_0,\cdots Q_{n-2},Q_{n-1}',[g_0])$. 
 
 For $g\in V([G(g_0)])$, we also define

$${\cal{P}}_\infty (g)= \bigcap _{n=0}^\infty \{ Q_n(g):Q_n(g)\subset Q_{n-1}(g),\ Q_n(g)\in {\cal{P}}_n(g)\mbox{ for all }n\ge 0\} .$$

Then ${\cal{P}}_\infty (g)$ is a collection of closed sets whose union is the whole sphere. If $v_2(g)$ is {\em{not persistently recurrent}} then all the sets in ${\cal{P}}_\infty (g)$ are either points or closures of Fatou components for $g$.  This follows from \cite{Roesch}. 

For any $\underline{Q}=(Q_i:i\ge 0)\in {\cal{Q}}_\infty $, we also define
$$V(\underline{Q},g_0)=\bigcap _{n=1}^\infty (V(Q_0,\cdots Q_n,g_0)\cup V(Q_0,\cdots Q_{n-1},\partial Q_n,g_0)),$$
where $V(Q_0,\cdots Q_{n-1},\partial Q_n,g_0)$ is the union of all those 
$V(Q_0,\cdots Q_{n-1},Q',g_0)$ such that $Q'\subset Q_n$ and $Q'$ is an edge or vertex of $G(Q_0,\cdots Q_{n-1})$ which is not a vertex of $G(Q_0,\cdots Q_{n-2})$. For each $n$, we have
$$V(G(g_0),g_0)=\bigcup _{\underline{Q}\in {\cal{Q}}_{n}}V(\underline{Q})=\bigcup _{\underline{Q}\in {\cal{Q}}_\infty }V(\underline{Q}).$$
Similar definitions and statements hold for $V([G(g_0)])$.

We now have the notation in place to state the main theorem of this section. THe hy[ptheses imply the hypotheses of \ref{3.1}. A branched covering $f$ of $\overline{\mathbb C}$ is said to be {\em{postcritically finite}} if the postcritical set $Z(f)=\{ f^n(c):c\mbox{ critical },n>0\} $  is finite.

\begin{theorem}\label{3.2} Let $V$ be the Riemann surface consisting of a connected component of the set  of quadratic rational maps $f$ with numbered critical values $v_1(f)$ and $v_2(f)$, such that $v_1(f)$ is of some fixed period, quotiented by M\"obius conjugation (all as previously stated). Let $g_0\in V$ be such that there exists a finite connected graph $G(g_0)\subset \overline{\mathbb C}$ with the following properties.
\begin{enumerate}
\item  $G(g_0)\subset g_0^{-1}(G(g_0))$.
\item For each edge $e$ of $G(g_0)$, $g_0^n(e)$ is more than a single edge of $G(g_0)$, for all sufficiently large $n$.
\item $G(g_0)$ separates $v_1(g_0)$ and $v_2(g_0)$.
\item  $v_1(g_0)$ and $v_2(g_0)$ are  separated from $G(g_0)$ by $g_0^{-t}(G(g_0))\setminus G(g_0)$ for some $t>0$. 
\end{enumerate}

If $g_1\in V([G(g_0)])$, and  $v_2(g_1)\in g_1^{-s}(G(g_1))\setminus G(g_1)$ for some $s>0$,  then $g_1\in V(G(g_0),g_0)$. 

Let ${\cal{Q}}$ be defined as in \ref{3.1} for $G(g_0)$.  Let $\underline{Q}\in {\cal{Q}}$. 
\begin{itemize}
\item $V(\underline{Q},g_0)$ is nonempty, connected and its complement in $V(G(g_0),g_0)$ is connected.
\item If there is some $n$ such that 
$$Q_i\subset G(Q_0,\cdots Q_{n-1})\cap \mbox{int}(Q_0(g))\mbox{ for all }i\ge n,$$
 or if there is $n$ such that 
 $$\bigcap _{i\ge 0}Q_i(g)\subset \mbox{int}(Q_n(g))\mbox{ for all }g\in V(Q_0,\cdots Q_n),$$ 
 $$g^m\left(\bigcap _{i\ge 0}Q_i(g)\right)\cap \mbox{int}(Q_n(g))=\emptyset \mbox{ for all }m>0,$$
 then $V(\underline{Q},g_0)$ is a single point.
\item  If $\underline{Q}=(Q_0,\cdots Q_n)\in {\cal{Q}}_n$ and  if  $Q_i\in{\cal{P}}(Q_0,\cdots Q_{i-1})$ for each $1\le i\le n$, then $V(\underline{Q},g_0)$ is open, and 
$$\overline{V(\underline{Q},g_0)}\subset V(\underline{Q},g_0)\cup V(Q_0,\cdots Q_{n-1},\partial Q_n,g_0),$$
where the closure is taken in $V(G(g_0),g_0)$.
\end{itemize}
Moreover
$$V(G(g_0),g_0)=V([G(g_0)]).$$
\end{theorem}

For the rest of this section, we keep the hypotheses of Theorem \ref{3.2}, and we use the notation that we have established. The following proposition shows that the possibilities for $\underline{Q}$ can be analysed by simply looking at those $\underline{Q}=(Q_i)\in {\cal{Q}}$ for which all the $Q_i$ are topological discs. 

\begin{proposition}\label{3.3} For any $(Q_0,\cdots Q_n)\in {\cal{Q}}_n$, there is $(Q_0,Q_1'\cdots Q_{n}')\in {\cal{Q}}_{n}$ such that $Q_i'$ is a topological disc for all $0\le i\le n$,  and $Q_i\subset Q_i'$ for $0< i\le n$, and there are isotopic subgraphs $G'(Q_0,\cdots Q_{n-1})$ and $G'(Q_0,Q_1'\cdots Q_{n-1}')$ of $G(Q_0,\cdots Q_{n-1})$ and $G(Q_0,Q_1',\cdots Q_{n-1}')$ such that $Q_i\subset G'(Q_0,Q_1',\cdots Q_{n-1})$ for all $1\le i\le n-1$ with $Q_i'\ne Q_i$, and the isotopy between $G'(Q_0,Q_1,\cdots Q_i)$ and $G'(Q_0,Q_1',\cdots Q_i')$ extends to the isotopy between $G'(Q_0,Q_1,\cdots Q_{n-1})$ and $G'(Q_0,Q_1',\cdots Q_{n-1}')$ for all $0\le i<n-1$.\end{proposition}

This is not difficult. The main step is the following.
\begin{lemma}\label{3.4} If $e$ is any edge of $G_n(g)\setminus G(g)$, for any $g\in V(G(g_0))$ and any integer $n\ge 1$, then $e\cap g^{-m}(e)=\emptyset $ for any $m\ge 1$. \end{lemma}

\noindent{\em{Proof.}} It suffices to prove this for $n=1$, because any edge $e$ of $G_n(g)\setminus G(g)$ is a contained in $g^{1-n}(e')$ for some edge $e'$ of $G_1(g)\setminus G(g)$. So now we assume that $e$ is an edge of $G_1(g)\setminus G(g)$. Now $G_1(g)=g^{-1}(G(g))$. So 
$$g^{-m}(G_1(g)\setminus G(g))=g^{-(m+1)}(G(g))\setminus g^{-m}(G(g)).$$
So 
$$g^{-m}(G_1(g) \setminus G(g))\cap g^{-m}(G(g))=\emptyset $$
for all $m\ge 0$. But $G(g)\subset g^{-1}(G(g))=G_1(g)$, and hence $G(g)\subset g^{-m}(G(g))$ for all $m\ge 0$ and $G_1(g)\subset g^{-m}(G(g))$ for all $m\ge 1$. So 
$$g^{-m}(G_1(g)\setminus G(g))\cap G_1(g)=\emptyset $$
for all $m\ge 1$, as required.

\ep

\noindent {\em{Proof of the proposition.}}  We prove this by induction on $n$. If $n=1$ then there is nothing to prove, because $G(g)$ is isotopic to $G(g_0)$. So we assume it is true for $n-1\ge 1$, and we need to prove that it is also true for $n$. If $Q_{n}$ is a topological disc, there is nothing to prove.  Otherwise, there is a least $1\le i\le n$ such that $Q_i$ is not a topological disc. Then $Q_i$ is an edge or point of $G(Q_0,\cdots Q_{i-1})$. Let $Q_i(g)$ be the corresponding isotopically varying edge or point of $G(Q_0,\cdots Q_{i-1})$ for $g\in V(Q_0,\cdots Q_{i-1},g_0)$. Fix such a $g$. Write $e=Q_i(g)$ if $Q_i(g)$ is an edge of $G_i(g)$. Otherwise, let $e$ be an edge of $G_i(g)$ in $\partial Q_{i-1}(g)$ which contains the point $Q_i(g)$. Let $Q_i'$ be any  closed topological disc such that $(Q_0,\cdots Q_{i-1},Q_i')\in {\cal{Q}}_i$ with $Q_i\subset Q_i'$. It has already been noted in \ref{3.1} that if $g\in V(Q_0,\cdots Q_{i-1},Q_i,g_0)$ and $h\in V(Q_0,\cdots Q_{i-1},Q_i',g_0)$ then $g^{-1}(G_i(g)\setminus Q_i(g))$ and $h^{-1}(G_i(h)\setminus Q_i(h))$ are isotopic.  Then by \ref{3.4}, $e\cap g^{-m}(e)=\emptyset $ for all $g\in V(Q_0,\cdots Q_i)\cup V(Q_0,\cdots Q_{i-1},Q_i')$ and all $m>0$.  So $Q_\ell\cap g^{i-\ell }(e)=\emptyset $ for all $i<\ell \le n$ and for all such $g$. For $i\le \ell \le n$ we choose a topological disc  $Q_\ell '$ so that $(Q_0,\cdots Q_{i-1},Q_i'\cdots Q_{\ell }')\in {\cal{Q}}_\ell $ and $Q_\ell \subset Q_\ell '$. Once $Q_i'$ has been chosen, the choice of $Q_\ell '$ for $\ell >i$ is unique. So then by induction on $\ell $, we have that if  $g\in V(Q_0,\cdots Q_\ell ,g_0)$ and $h\in V(Q_0,\cdots Q_{i-1},Q_i'\cdots Q_\ell ',g_0)$, then 
$$G_{\ell +1}'(g)=g^{-1}(G_\ell '(g))\setminus g^{-1}(Q_\ell )=G_{\ell +1} (g)\setminus \bigcup _{j=i}^\ell g^{-1-\ell +j}(Q_j(g))$$ 
and 
$$G_{\ell +1}(h)=G_{\ell +1}(h)\setminus \bigcup _{j=i}^\ell  h^{-1-\ell +j}(Q_j (h))$$
 are isotopic. The claimed extension property holds, by construction.\ep

The following lemma uses Thurston's theorem for critically finite branched coverings, and the set-up for this. Hopefully the explanation is sufficiently self-contained, but see \cite{R3} or \cite{D-H2} for more details. Two critically finite branched coverings $f_0$ and $f_1$ are said to be {\em{Thurston equivalent}} if there is a homotopy $f_t$ ($t\in [0,1]$ through critically finite branched coverings, such that the postcritical set $Z(f_t)$ varies isotopically for $t\in [0,1]$. Thurston's theorem gives a necessary and sufficient condition for a critically finite branched covering $f$ of $\overline{\mathbb C}$ to be Thurston equivalent to  a critically finite rational map. The rational map is then unique up to conjugation by a M\"obius transformation. The condition is in terms of non-existence of loop sets in $\overline{\mathbb C}\setminus Z(f)$ with certain properties. In the case of degree two branched coverings, the criterion reduces to the non-existence of a {\em{Levy cycle}}, as is explained in the proof below.

\begin{lemma}\label{3.5} Let $(Q_0,\cdots Q_{n})\in {\cal{Q}}_{n}={\mathcal{Q}}_n(g_0)$ where  $Q_{n}$ is a closed topological disc (that is, the closure of a component of $\overline{\mathbb C}\setminus G(Q_0,\cdots Q_{n-1})$ if $n\ge 1$, or $Q_0=Q_0(g_0)$ if $n=0$) and that $\underline {Q}=(Q_i:0\le i<N)\in {\cal{Q}}$  for $N>n+1$, possibly $N=\infty, $ with $Q_i\subset Q_n\cap G(Q_0,\cdots Q_n)\cap \mbox{int}(Q_{n-1})$ for $i>n$ and such that  $\bigcap _{i\ge 0}Q_i$ represents an eventually periodic point. Suppose that $V(Q_0,\cdots Q_n,[g_0])\ne \emptyset  $. Then $V(\underline{Q},[g_0])=\{ g_1\} $  for some $g_1\in V$.\end{lemma}

\begin{remark} Note that there is still no statement that  $g_1\in V(G(g_0),g_0)$. That will come later.\end{remark}

\noindent{\em{Proof.}} Let $g\in V(Q_1,\cdots Q_{n},[g_0])$. Then $G(Q_0,\cdots Q_{n})$ is canonically homeomorphic to $G_{n+1}(g)$, and the  homeomorphism carries $\cap _{i\ge 0}Q_i$ to a point $z_0$ in  $G_{n+1}(g)$, which, like $v_2(g)$, is in $\mbox{int}(Q_{n-1}(g))\cap Q_n(g)$. We can construct a path $\beta :[0,1]\to Q_{n}(g)\cap \mbox{int}(Q_{n-1}(g))$ with $\beta (0)=v_2(g)$ and $\beta (1)=\cap _{0\le i<N}Q_i(g)=z_0$. We can also choose $\beta $ so that $\beta ([0,1))\subset \mbox{int}(Q_n(g))$. The hypotheses ensure that 
$$z_0\in G_{n+1}(g)\setminus G_{n-1}(g)=(G_{n+1}(g)\setminus G_n(g))\cup(G_n(g)\setminus G_{n-1}(g)).$$
 The endpoint-fixing  homotopy class of $\beta $ is uniquely determined in 
 $$\overline{\mathbb C}\setminus \{ g^i(z_0):i>0\} .$$
 This means that the Thurston-equivalence class of the post-critically finite branched covering $\sigma _\beta \circ g$ is well defined, where $\sigma _\beta $ is a homeomorphism which is the identity outside an arbitrarily small neighbourhood of $\beta $ and maps $\beta (0)$ to $\beta(1)=z_0$. 

Then we claim that $\sigma _\beta \circ g$ is Thurston equivalent to a rational map. Since this is a branched covering of degree two, it suffices to prove the non-existence of a Levy cycle. By definition, a Levy cycle is an isotopy class of a collection of distinct and disjoint simple closed loops, where  the isotopy is in the complement of the postcritical set. In the present case, it is convenient to consider isotopy in the complement of a potentially larger forward invariant  set $X$ consisting of the union of the forward orbits of $z_0$, $c_1(g)$ and the vertices of $G_0(g)$. Thurston's Theorem adapts naturally to this setting. A Levy cycle for $\sigma _\beta \circ g$ is then the isotopy class in $\overline{\mathbb C}\setminus X$ of a finite  set $\{ \gamma _i:1\le i\le r\} $ of distinct and disjoint simple closed loops, such that there is a component $\gamma _i'$ of $(\sigma _\beta \circ g)^{-1}(\gamma _{i+1})$ (writing $\gamma _1=\gamma _{r+1}$, so that this also makes sense if $i=r$), such that $\gamma _i$ and $\gamma _i'$ are isotopic in $\overline{\mathbb C}\setminus X$, for $1\le i\le r$. We consider the case when $z_0\in \partial Q_n(g)\cap \mbox{int}(Q_{n-1}(g)\subset G_n(g)\setminus G_{n-1}(g)$. The other case, when $z_0\in G_{n+1}(g)\cap \mbox{int}(Q_n(g))\subset G_{n+1}(g)\setminus G_n(g)$, can be dealt with similarly. The $\gamma _i$ can also  be chosen to have only transversal intersections with $G_{n-1}(g)$. We have $z_0\notin G_{n-1}(g)$. So $(\sigma _\beta \circ g)^{-1}(G_{n-1}(g))=g^{-1}(G_{n-1}(g))=G_n(g)$. Now $(\sigma _\beta \circ g)^{-1}(\gamma _{i+1})$ has two components $\gamma _i'$ and $\gamma _i''$,  each of them  mapped homeomorphically to $\gamma _{i+1}$ by $\sigma _\beta \circ g$. Each  transverse intersection between $\gamma _i$ and $G_{n-1}(g)$ in $\overline{\mathbb C}\setminus X$ lifts to two transverse intersections between $\gamma _i'\cup \gamma _i''$ and $G_n(g)\supset G_{n-1}(g)$ in $\overline{\mathbb C}\setminus (\sigma _\beta\circ g)^{-1}(X)$, one of these intersections with $\gamma _i'$ and one with $\gamma _i''$. Because of the isotopy between $\gamma _i$ and $\gamma _i'$, the intersection on $\gamma _i'$ must be in $G_{n-1}(g)$ and must be essential in $\overline{\mathbb C}\setminus X$. So this means that each arc on $\gamma _{i+1}$ between essential intersections in $G_{n-1}(g)$ lifts to an arc on $\gamma _i'$ between essential intersections in $G_{n-1}(g)$, and this arc can be isotoped in the complement of $X$ to an arc on $\gamma _i$ between essential intersections in $G_{n-1}(g)$. Since $g^{-1}(G_{n-j}(g)\setminus G_{n-j-1}(g))=G_{n-j+1}(g)\setminus G_{n-j}(g)$, it follows by induction on $j\ge 1$ that all intersections between $\gamma _i$ and $G_{n-1}(g)$ are in $G_0(g)$. So every arc of intersection of $\gamma _i$ with $G_{n-1}(g)$ must be with $G_0(g)$, and  in a single set of ${\cal{P}}_{n-1}(g)$ adjacent to a vertex of $G_0(g)=G(g)$. If $n$ is large enough, this is clearly impossible, because successive arcs are too far apart. But we can assume $n$ is large enough to make this impossible, by replacing $\gamma _i$ by $\gamma _i^m$ if necessary, where $\gamma _i^0=\gamma _i$ and $\gamma _i^1=\gamma _i'$ and  $\gamma _i^{m+1}$ is isotopic to $\gamma _i^m$, obtained by lifting, under $\sigma _\beta \circ g$, the isotopy between $\gamma _{i+1}^m$ and $\gamma _{i+1}^{m-1}$, writing $\gamma _1^m=\gamma _{r+1}^m$. It follows that all intersections between $\gamma _i^m$ and $G_0(g)$ are in a single set of ${\cal{P}}_{n+m-1}(g)$, adjacent to a vertex of $G_0(g)$. If $m$ is large enough, this is, once again, impossible.

So Thurston's Theorem for critically finite branched coverings implies that $\sigma _\beta \circ g$ is Thurston equivalent to a unique rational map $g_1$. From the definitions, we have  $g_1\in V(\underline{Q})$. By the uniqueness statement in Thurston's Theorem, we have $V(\underline{Q})=\{ g_1\} $. For if $g_2\in V(\underline {Q}$  and $v_1(g_1)\in G_{m+1}(g)\setminus G_m(g_1)$ for $m=n$ or $n-1$ then there is a homeomorphism $\varphi $ of $\overline{\mathbb C}$ which maps $G_m(g_1)$ to  $G_m(g_2)$ which conjugates dynamics  of $g_1$ and $g_2$ on these graphs, and maps $v_2(g_1)$ to $v_2(g_2)$ and $g_1^i(v_1(g_1)$ to $g_2^i(v_1(g_2))$ for all $i\ge 0$. So $\varphi \circ g_1\circ \varphi ^{-1}$ and $g_2$ are homotopic through branched coverings which are constant on $G_m(g_2)$, and on the postcritical sets. 
\ep

The following lemma, like the preceding one, gives a condition under which $V(\underline{Q},[g_)])$ is nonempty. It has some overlap with the preceding one, but is of a rather different type. it uses the $\lambda$-Lemma of Mane, Sullivan and Sad \cite{M-S-S} rather than Thurston's Theorem, and is a result about connected sets of maps rather than postcritically finite maps. \ref{3.6} has no uniqueness statement. The two lemmas complement each other in the proof of \ref{3.2}.

\begin{lemma}\label{3.6} Let $g_1\in V([G(g_0)])$. Let $Q_{n-1}\in {\cal{P}}_{n-1}(g_1)$ and let  $v_2(g_1)\in \mbox{int}(Q_{n-1})\cap G_n(g_1)$ for some $n\ge 1$. Then $V(\underline{Q},g_1)\ne \emptyset $ for all $\underline{Q}=(Q_i')$ with $Q_i'=Q_i$ for $i\le n-1$ such that $\cap _iQ_i$ is in the same component of $G_n(g_1)\cap \mbox{int}(Q_0)\cap Q_{n-1}$ as $v_2(g_1)$. \end{lemma}

\noindent{\em{Proof.}}  From the hypotheses on $g_1$, the graph $G_n(g)$ varies isotopically for 
$$g\in V(Q_0,\cdots Q_{n-1}\cup \partial Q_{n-1},g_1),$$ and the dynamics of maps in  $V(Q_0,\cdots Q_{n-1}\cup \partial Q_{n-1},g_1)$ are conjugate in the following sense. There is a homeomorphism 
$$\varphi _{g,h}:G_n(h)\to G_n(g),\ \ \ \ (g,h)\in (V(Q_0,\cdots Q_{n-1}\cup \partial Q_{n-1},g_1))^2,$$ 
such that the map $(g,h)\mapsto \varphi _{g,h}$ is continuous, using the uniform topology on the image and $\varphi _{g,h}\circ h=g\circ \varphi _{g,h}$ on $G_i(h)$, and $\varphi _{h,h}$ is the identity.  Each preperiodic point in $G_n(g)$ varies holomorphically for $g\in V(Q_0,\cdots Q_{n-1},g_1)$, that is, $\varphi _{g,h}(z)$ varies holomorphically with $g$ for each preperiodic point $z\in G_n(g_1)$. But preperiodic points are dense in $G_n(g_1)$. (For example, the backward orbits of vertices of $G_n(g_1)$ are dense in $G_n(g_1)$, by the expansion properties of $g_1$ on $G_n(g_1)$ established in \ref{2.2}.) It follows by the $\lambda $-Lemma \cite{M-S-S} that $(z,g)\mapsto \varphi _{g,h}(z)$  is continuous in $(z,g)$, and  holomorphic in $g\in V(Q_0,\cdots Q_{n-1};g_1)$ for each $z\in G_n(g_1)$. (In fact it is also possible to prove this by standard hyperbolicity arguments.) Now we assume without loss of generality, conjugating by a M\"obius transformation if necessary, that $Q_{n-1}(g)\subset \mathbb C$ for $g\in V(Q_0,\cdots Q_{n-1};g_1)$, in particular, $\{ v_2(g)\} \cup (G_n(g)\cap Q_{n-1}(g))\subset \mathbb C$. We consider the maps
$$\psi (z,g)=\varphi _{g,g_1}(z)-v_2(g)$$
for $z\in G_n(g_1)\cap Q_{n-1}(g_1)$. The map $(z,g)\mapsto \psi (z,g)$ is, once again, continuous in $(z,g)$ and holomorphic in $g\in V(Q_0,\cdots Q_{n-1},g_1)$.  Now write $z_0=v_2(g_1)$, so that $z_0\in G_n(g_1)\setminus G_{n-1}(g_1)$. The map $g\mapsto \psi (z_0,g)$ is holomorphic in $g$ and the inverse image of a disc round $0$ is a topological disc containing $z_0$ in its interior. By continuity, the same is true for $z$ sufficiently near $z_0$. Hence for all $z$ sufficiently near $z_0$, the map $g\mapsto \psi (z,g)$ has a zero. This argument shows that the set of $z\in Q_{n-1}(g_1)\cap \mbox{int}(Q_0)\cap G_n(g_1)$ for which $g\mapsto \psi (z,g):V(Q_0,\cdots Q_{n-1})\to \mathbb C$ has a zero in $V(Q_0,\cdots Q_{n-1})$ is open, because $z_0$ can be replaced by any other  point $z$ in $Q_{n-1}(g_0)\cap \mbox{int}(Q_0)\cap G_n(g_0)$.   But the set is also closed in $\mbox{int}(Q_0(g_1))\cap Q_{n-1}(g_1)\cap G_n(g_1)$. For suppose $\psi (z_k,g_k)=0$ and $z_k\to z$.   Then either some subsequence of $g_k$ has a limit $g$, in which case $\psi (z,g)=0$ for any such $g$, and the proof is finished, or $g_k\to \infty $ in $V$. 

We now have to deal with the situation  that $g_k\to \infty $ in $V$. In this case, we can assume that all $z_k$ are in a single edge of $G_n(g_1)$. We will now show that this implies the existence of a Levy cycle for the unique map $h_1\in G(Q_0,\cdots Q_{n-1},Q_n')$, where $Q_n'$ is a vertex of $G_n(g_1)\setminus G_{0}(g_1)$. This contradicts the result of \ref{3.5}, and hence $g_k\to \infty $ is impossible. We use certain facts about the ends of $V$. These appear in Stimson's thesis \cite{Sti} and in various other papers, for example \cite{K-R}. Choosing suitable representatives of $g_k$ up to M\"obius conjugation,chosen, in particular, so that $c_1(g_k)=1$ for all $k$, $g_k$ converges to a periodic M\"obius transformation $g(z)=e^{2\pi ir/q}z$ for some integer $q\ge 2$ and some $r\ge 1$ which is coprime to $q$, and the set $\{ g_k^i(v_1(g_k)):i\ge 0\} \cup \{ v_2(g_k)\} =Z_1(g_k)$ converges $Z_1(g)=\{ e^{2\pi ij/q}:0\le j\le q-1\} $.   Let $\overline{V}$ be the compactification of $V$ obtained by adding the M\"obius transformations at infinity and consider a fixed $g\in \overline{V}\setminus V$.  The parametrisation can be chosen so that the other critical point $c_2(g_k)=1+\rho _k$ where $\lim _{k\to \infty }\rho _k=0$. Passing to a subsequence if necessary, we may assume that $g_k$ is in a single branch of $V$ near $g$.  Then $(g_k^q(1+z\rho _k)-1)/\rho_k$ has a limit  as $k\to \infty $ for $z$ bounded and bounded from $\frac{1}{2}$, which is the quadratic map
$$h:z\mapsto qa+z+\dfrac{1}{4(z-\frac{1}{2})}$$
for a constant $a\ne 0$. 

Because of the nature of $h$,  it follows that all the eventually periodic points of $g_k$ whose forward orbits have size $\le N$ lie in the $C|\rho _k|$-neighbourhood of $Z(g)$, if $k$ is sufficiently large given $N$, for a suitable constant $C$. We will call this neighbourhood $U_1$. So if $N$ is a bound on the number of vertices of $G_n(g_k)$ --- which is, of course, the same for all $k$ --- then all vertices of $G_n(g_k)$ lie in $U_1$, for all sufficiently large $k$. If the edge $e$ of $G_n(g_k)$  between one vertex and $v_2(g_k)$ is contained in a single component of $U_1$, then the boundary of $U_1$ provides a Levy cycle for $h_1$, where $Q_n'$ is taken to be  this vertex, and this gives the required contradiction. Now $e\subset G_n(g_k)\setminus G(g_k)$, and  we claim that $e\subset U_1$, up to isotopy preserving the set $X$ which is the union of the vertex set of $G_n(g_k)$ and the set $\{ g_k^i(v_1(g_k)):i\ge 0\} $.  We consider only essential intersections between $G_n(g_k)\setminus G(g_k)$ and $\partial U_1$ under isotopies preserving $X$. If $\gamma $ is an arc of essential intersection then it must be in the inverse image under $g_k$ of an arc which contains one or more arcs of essential intersection. Since the number of such arcs is finite, each arc must be in the inverse image of exactly one other, and the inverse image of each arc contains exactly one other. But then each edge must be contained in a periodic edge of $G_n(g_k)\setminus G_0(g_k)$. But there are none. So there are no essential intersections   with $\partial U_1$. In particular, $e\subset U_1$ up to isotopy preserving $X$, as required.
\ep

 \begin{corollary}\label{3.7} For all $(Q_0,\cdots Q_n)\in {\cal{Q}}_n$, if $V(Q_0,\cdots Q_n,[g_0])\ne \emptyset $, then it is connected. \end{corollary}
 \noindent{\em{Proof.}} By \ref{3.6}, if $g_1\in V([G(g_0)])$, for any nonempty component $V(Q_0,\cdots Q_n,1,g_1)$ of $V(Q_0,\cdots Q_n,g_1)$, 
 $$V(Q_0,\cdots Q_n,1,g_1)\cap V(\underline{Q},g_1)\ne \emptyset $$ 
 for any $\underline{Q}\in {\cal{Q}}$ such that $\underline{Q}$ extends $(Q_0,\cdots Q_n)$. 
 
 In particular, if $g_2\in V(G([g_0])$, possibly with $g_2=g_1$, and $V(Q_0,\cdots Q_n,2,g_2)$ is another component of $V(Q_0,\cdots Q_n,[g_0])$, then there is $\underline{Q}$ with $\bigcap _{i\ge 0}Q_i$ representing an eventually periodic point such that $V(\underline{Q})$ which intersects both components.  But this is impossible, because $V(\underline{Q},[g_0])$ contains a single postcritically finite map. So $V(Q_0,\cdots Q_n,[g_0])$ is connected.
\ep

\begin{lemma}\label{3.8} $V(\underline{Q},g_0)\ne \emptyset $ for all $\underline{Q}\in {\cal{Q}}$. Hence $V([G(g_0)])=V(G(g_0),g_0)$.
\end{lemma}

\noindent{\em{Proof.}} By \ref{3.6}, $V(\underline{Q},g_0)\ne \emptyset $ for all $\underline{Q}$  with $\cap _{i\ge 0}Q_i\subset \partial Q_n$ for any $(Q_0,\cdots Q_n)\in {\cal{Q}}_n$ and such that $V(Q_0,\cdots Q_{n-1},\partial Q_n,g_0)\ne \emptyset $ with $Q_n\subset \mbox{int}(Q_0)$, because then $\partial Q_n\cap \mbox{int}(Q_0)$ is connected. 
This means that if $V(\underline{Q})\cap V(G(g_0),g_0)\ne \emptyset $, then we have $V(\underline{Q'})\cap V(G(g_0),g_0)\ne \emptyset $ for any $\underline{Q'}$ which can be connected to $\underline{Q}$ by sets $\partial Q^i_{n_i}$, for varying $n_i$ and $\underline {Q^i}=(Q_0^i\cdots Q_{n_i}^i)$ with $Q_{n_i}^i\subset \mbox{int}(Q_0)$. But any $\underline{Q}$ and $\underline{Q'}$ can be connected in this way. 

The final statement of the lemma follows, using \ref{3.7}.

\ep

\begin{lemma}\label{3.9} $V(\underline{Q},[g_0])=V(\underline{Q},g_0)$ is singleton,  if there is $n\ge 1$ such that either $\bigcap _{i=0}Q_i(g)\subset  G_n(g)\cap \mbox{int}(Q_0(g))$ or  $\bigcap _{i=0}^\infty Q_i(g)=\underline{Q}(g)\subset \mbox{int}(Q_n(g))$ and such that $g^k(\underline{Q}(g))\cap \mbox{int}(Q_n(g))=\emptyset $ for all $k>0$, and for at least one $g\in V(\underline{Q},g_0))$.\end{lemma}

\noindent{\em{Proof.}}
  In both cases, the set $\underline{Q}(g)=\bigcap _{i=0}^\infty Q_i(g)$ is well-defined for all $g\in V(Q_0,\cdots Q_n)$. It is  a point, which follows from the result of \cite{Roesch} about non-persistently-recurrent points, but in any case the construction of a nested sequence of annuli of moduli bounded from $0$ is straightforward. Moreover $z(g)=\underline{Q}(g)$ is the limit of a sequence $z_\ell (g)$ of eventually periodic points  in $G_\ell (g)$ with the same property of being defined for all $g\in V(Q_0,\cdots Q_n,g_0)$.  Write $z(g_0)=\underline{Q}(g_0)$ and $z_\ell (g_0)$ for the sequence of eventually preperiodic points under $g_0$ with $\lim _{\ell \to \infty }z_\ell (g_0)=z(g_0)$. Then since $g\mapsto \psi (z_\ell (g_0),g)$ is holomorphic in $g$ and has a single zero $h_\ell $, the same is true for the limiting holomorphic function $g\mapsto \psi (z(g_0),g)$. The single zero is the unique point in $V(\underline{Q},g_0)$. 
\ep

Now the following lemma completes the proof of Theorem \ref{3.2}.

\begin{lemma}\label{3.10} The complement of $\overline{V(\underline{Q},g_0)}$ has exactly one component in $V(G(g_0),g_0)$ for all $\underline{Q}\in {\cal{Q}}$, for $\underline{Q}\in {\mathcal{Q}}_\infty $ and $\underline{Q}=(Q_0,\cdots Q_n)\in {\mathcal{Q}}_n$ such that $Q_0\setminus Q_n$ is connected.. 
\end{lemma}
\noindent{\em{Proof.}} If $\underline{Q}=(Q_i:i\ge 0)\in {\cal{Q}}_\infty $ and the complement of $V(\underline{Q},g_0)$ has more than one component in $V(G(g_0),g_0)$, then the same is true for the complement of $\overline{V(Q_0,\cdots Q_n,g_0)}$, for some $n$. So it suffices to show that the complement of $\overline{V(Q_0,\cdots Q_n,g_0)}$ has at most one component in $V(G(g_0),g_0)$ for each $(Q_0,\cdots Q_n)\in {\cal{Q}}_n$ such that $Q_0\setminus Q_n$ is connected, as this holds automatically if $n$ is sufficiently large. So suppose this is not true. Then $\partial V(Q_0,\cdots Q_n,g_0)$  is disconnected, taking boundary as a subset of $V(G(g_0),g_0)$. But 
$$\partial V(Q_0,\cdots Q_n,g_0)\subset V(Q_0,\cdots Q_{n-1},\partial Q_n\setminus \partial Q_0,g_0),$$ 
Moreover, if we fix $h\in V(G(g_0),g_0)$, there is a continuous surjective map 
$$\Phi :V(Q_0,\cdots Q_{n-1},\partial Q_n\setminus \partial Q_0,g_0)\to \partial Q_n(h)\setminus \partial Q_0(h),$$
defined by 
$$\Phi (g)=\varphi _{g,h}^{-1}(v_2(g)),$$
where $\varphi _{g,h}$ is as in the proof of \ref{3.6}. By \ref{3.6} to \ref{3.8}, $\Phi ^{-1}(\Phi (g))$ is connected for each $g$. In fact if $v_2(g)$ is eventually periodic, then this already follows from \ref{3.5}.  Also, $\Phi (\partial V(Q_0,\cdots Q_n,g_0))\supset \partial Q_n(h)\setminus  \partial Q_0(h)$ by the proof of \ref{3.6}. So if $\partial V(Q_0,\cdots Q_n,g_0)$ can be written as a disjoint union of two nonempty closed sets $X_1$ and $X_2$ in $V(G(g_0),g_0)$, we have $\Phi ^{-1}(\Phi (x))\cap X_2=\emptyset $ for each $x\in X_1$, and similarly with $X_1$ and $X_2$ interchanged. So $\Phi(X_1)$ and $\Phi (X_2)$ are disjoint.  Since $X_j$ is closed and bounded (and hence compact), we see that $\Phi (X_j)$ is also closed (and bounded and compact). So then $\partial Q_n(h)\cap \mbox{int}(Q_0(h))$ is a union of two non-empty disjoint closed sets and  is disconnected, giving a contradiction. 

\ep

\end{document}